\documentclass[a4paper,11pt]{article}

\usepackage{amssymb}     
\usepackage{amsthm}      
\usepackage{indentfirst} 
\usepackage{paralist} 
\usepackage{amsmath}
\usepackage{txfonts}
\usepackage{mathrsfs}
\usepackage{hyperref}
\usepackage{geometry}
\theoremstyle{theorem}
\newtheorem{thm}{Theorem}[section]
\newtheorem{lem}[thm]{Lemma}
\newtheorem{prop}[thm]{Proposition}
\newtheorem{cor}[thm]{Corollary}
\newtheorem{conj}[thm]{Conjecture}
\theoremstyle{definition}
\newtheorem{defn}[thm]{Definition}
\newtheorem{rmk}[thm]{Remark}

\newtheorem*{notation}{Notation}
\numberwithin{equation}{section}

\newcommand{\IBr}{\operatorname{IBr}\nolimits}
\newcommand{\Irr}{\operatorname{Irr}\nolimits}
\newcommand{\PSL}{\operatorname{PSL}\nolimits}
\newcommand{\PSU}{\operatorname{PSU}\nolimits}
\newcommand{\GL}{\operatorname{GL}\nolimits}
\newcommand{\GU}{\operatorname{GU}\nolimits}
\newcommand{\SL}{\operatorname{SL}\nolimits}
\newcommand{\SU}{\operatorname{SU}\nolimits}
\newcommand{\bl}{\operatorname{bl}\nolimits}

\newcommand{\Aut}{\operatorname{Aut}\nolimits}
\newcommand{\Rad}{\operatorname{Rad}\nolimits}
\newcommand{\Res}{\operatorname{Res}\nolimits}
\newcommand{\Ind}{\operatorname{Ind}\nolimits}
\newcommand{\diag}{\operatorname{diag}\nolimits}
\newcommand{\Out}{\operatorname{Out}\nolimits}

\newcommand{\zero}{\mathbf{0}}
\newcommand{\F}{\mathbb{F}}
\newcommand{\barF}{\overline{\mathbb{F}}}
\newcommand{\Z}{\mathbb{Z}}
\newcommand{\cF}{\mathcal{F}}

\newcommand{\bG}{\mathbf{G}}

\newcommand{\bL}{\mathbf{L}}

\newcommand{\bH}{\mathbf{H}}
\newcommand{\fS}{\mathfrak{S}}

\newcommand{\cW}{\mathcal{W}}

\newcommand{\tDelta}{\tilde{\Delta}}

\newcommand{\bc}{\mathbf{c}}
\newcommand{\fb}{\mathfrak{b}}
\newcommand{\cR}{\mathcal{R}}
\newcommand{\sC}{\mathscr{C}}
\newcommand{\cO}{\mathcal{O}}

\begin{document}
	
\title{The blocks and weights of finite special linear and unitary groups\footnote{The author gratefully acknowledges financial support by SFB TRR 195.}}

\author{Zhicheng Feng \\
	Fachbereich Mathematik, Technische Universit\"at Kaiserslautern, Postfach 3049,	\\
	67653 Kaiserslautern,	Germany\\	
	\emph{E-mail address}:   feng@mathematik.uni-kl.de 
	}

\date{}

\maketitle

\begin{abstract}
This paper has two main parts. First, we give a classification of the $\ell$-blocks of finite special linear and unitary groups $\SL_n(\epsilon q)$ in the non-defining characteristic $\ell\ge 3$.
Second, we describe how the $\ell$-weights of $\SL_n(\epsilon q)$ can be obtained from the $\ell$-weights of $\GL_n(\epsilon q)$ when $\ell\nmid\mathrm{gcd}(n,q-\epsilon)$, and verify the Alperin weight conjecture for $\SL_n(\epsilon q)$ under the condition $\ell\nmid\mathrm{gcd}(n,q-\epsilon)$.
As a step to establish the Alperin weight conjecture for all finite groups, we
prove the inductive blockwise Alperin weight condition for any unipotent $\ell$-block of $\SL_n(\epsilon q)$ if $\ell\nmid\mathrm{gcd}(n,q-\epsilon)$.
\end{abstract}



\section{Introduction}

Let $q=p^f$ be a power of a prime $p$ and $\SL_n(\epsilon q)$ with $\epsilon=\pm1$ be the
finite special linear~(when $\epsilon=1$)~ and unitary~(when $\epsilon=-1$)~ group~($\SL_n(-q)$ is understood as $\SU_n(q)$,
for definitions, see Section \ref{notations-and-conventions}).
Let $\ell$ be a prime number different from $p$.
We are interested in parametrizing $\ell$-blocks of $\SL_n(\epsilon q)$.
It seems natural to proceed through the $\ell$-blocks of of general linear and unitary groups, which had been classified by Fong and Srinivasan \cite{FS82} for odd prime $\ell$ and by Brou\'e \cite{Br86} for $\ell=2$.
For arbitrary finite groups of Lie type, Cabanes and Enguehard \cite{CE99} gave a label for their $\ell$-blocks when $\ell\ge 7$ and this result was generalised by Kessar and Malle \cite{KM15} to its largest possible generality.

It is natural to try to relate the label in \cite{FS82} and \cite{Br86} to the label in \cite{CE99} and \cite{KM15} for an $\ell$-block $B$ of $\GL_n(\epsilon q)$.
In this paper, we compare these two kinds of labeling and then give the number of $\ell$-blocks of $\SL_n(\epsilon q)$ covered by $B$.
The proof here relies on some lemmas given in \cite{KM15} to investigate the relationship between the labeling of $\ell$-blocks of $\GL_n(\epsilon q)$ and $\SL_n(\epsilon q)$.
In this way we obtain a corresponding parametrization of $\ell$-blocks of $\SL_n(\epsilon q)$ when $\ell$ is odd~(see Remark \ref{blocksofslsu}).

One of the most important conjectures in the modular representation theory of finite groups is the Alperin weight conjecture,
which relates for a prime $\ell$ information about a finite group $G$ to properties of $\ell$-local subgroups of $G$, that is, normalizers of $\ell$-subgroups of $G$.
For a finite group $G$ and a prime $\ell$,
we write ${\rm Irr}(G)$ for the set of ordinary irreducible  characters of $G$, and ${\rm IBr}_\ell(G)$ for
the set of irreducible $\ell$-Brauer characters of $G$.
Moreover, ${\rm Irr}(B)$ and  ${\rm IBr}_\ell(B)$ denote the sets of ordinary irreducible characters and irreducible
$\ell$-Brauer characters of $B$, respectively,  where $B$ is an $\ell$-block of $G$.
An \emph{$\ell$-weight} of $G$ means a pair $(R,\varphi)$, where $R$ is an $\ell$-subgroup of $G$ and $\varphi\in\Irr(N_G(R))$ with $R\subseteq\ker\varphi$ is of $\ell$-defect zero viewed as a character of $N_G(R)/R$.
When such a character $\varphi$ exists, $R$ is necessarily an $\ell$-radical subgroup of $G$.
For an $\ell$-block $B$ of $G$, a weight $(R,\varphi)$ is called a \emph{$B$-weight} if $\bl_\ell(\varphi)^G=B$, where $\bl_\ell(\varphi)$ is the $\ell$-block of $N_G(R)$ containing $\varphi$. We denote by $\cW_\ell(B)$  the set of all $G$-conjugacy classes of $B$-weights.
In \cite{Al87}, Alperin gave the following conjecture.

\begin{conj}[Alperin]
\label{weiconj}
Let $G$ be a finite group, $\ell$ a prime.
If $B$ is an $\ell$-block of $G$, then $|\cW_\ell(B)|=|\IBr_\ell(B)|$.
\end{conj}

The blockwise Alperin weight Conjecture \ref{weiconj} (BAWC) was proved by Isaacs and Navarro \cite{IN95} for $\ell$-solvable groups.
It was also shown to hold for groups of Lie type in defining characteristic by Cabanes \cite{Ca88}.
By work of Alperin, An and Fong, there is a combinatorial description for the $\ell$-weights of general linear and unitary groups if $\ell$ is not the defining characteristic
and from this (BAWC) holds for general linear and unitary groups for any prime, see \cite{AF90}, \cite{An92}, \cite{An93} and \cite{An94}.
In this paper we give a description of the $\ell$-weights of special linear and unitary groups $\SL_n(\epsilon q)$ with the assumption $\ell\nmid\mathrm{gcd}(n,q-\epsilon)$~(see Remark \ref{weightofslsu}).
Here, the $\ell$-weights of $\SL_n(\epsilon q)$ are obtained from the $\ell$-weights of $\GL_n(\epsilon q)$.
We will prove the following statement.

\begin{thm} \label{alpofsl}
Let $X\in\{\SL_n(q),\SU_n(q)\}$ and $\ell\nmid |Z(X)|$.
Then the blockwise Alperin weight Conjecture \ref{weiconj} holds for $X$.
\end{thm}

Even though (BAWC) has been verified in many particular instances, it has not been possible so far to find a general proof for arbitrary finite groups.
In the recent past, the conjecture has been reduced to certain (stronger) statements about finite (quasi-)simple groups by Navarro and Tiep \cite{NT11} for original version and by Sp\"ath \cite{Sp13} for blockwise version.
More precisely, it was shown that in order for the (BAWC) to hold for all finite groups, it is sufficient that all non-abelian finite simple groups satisfy a system of conditions, which is called the \emph{inductive blockwise Alperin weight (iBAW) condition}.
In this paper we use another version of the (iBAW) condition~(see Definition \ref{induc})~ which was given by Koshitani and Sp\"ath  \cite{KS16}.

The (iBAW) condition has been verified for some cases, such as many of the sporadic groups, simple alternating groups and any prime, simple groups of Lie type and the defining characteristic.
But for non-defining characteristic, only a few simple groups of Lie type have been proved to satisfy the (iBAW) condition~(see \cite{CS13}, \cite{Ma14}, \cite{Sc16} and \cite{Sp13}).

It seems that there is no general method yet to verify the (iBAW) condition for arbitrary finite simple groups of Lie type and the non-defining characteristic, even for simple groups of type $A$.
Using the description of $\ell$-weights of $\GL_n(\epsilon q)$ in \cite{AF90}, \cite{An92}, \cite{An93} and \cite{An94}, Li and Zhang  \cite{LZ18} proved that if
the pair $(n,q)$ is such chosen that the outer automorphism group of $\PSL_n(\epsilon q)$ is cyclic, then the simple group $\PSL_n(\epsilon q)$ satisfies the (iBAW) condition for any prime.
In this paper, we consider the (iBAW) condition for the unipotent blocks of $\SL_n(\epsilon q)$ without any restriction for $n$ and $q$.
Our results are the following:

\begin{thm}\label{ibawmain}
Let $X\in\{\SL_n(q),\SU_n(q)\}$ and $\ell\nmid |Z(X)|$.
Suppose that $b$ is a unipotent $\ell$-block of $X$,
then the inductive blockwise Alperin weight~(iBAW) condition~(cf. Definition \ref{induc})~ holds for $b$.
\end{thm}

This paper is built up as follows.
In Section \ref{preliminaries}, we introduce the general notation around characters, weights and general linear and unitary groups.
In Section \ref{Brauercharacterofslsu}, we recall the results of \cite{KT09} and \cite{De17} about irreducible Brauer characters of special linear and unitary groups.
Then we determine when the labeling of blocks of general linear and unitary groups in \cite{FS82} and \cite{Br86} is the labeling given in \cite{CE99} and \cite{KM15},
and then classify the blocks of special linear and unitary groups in non-defining characteristic in Section \ref{blocksofslsuchap}.
In Section \ref{chapweightsofslsu} we give a description of weights of special linear and unitary groups in non-defining characteristic and prove Theorem \ref{alpofsl}.
Section \ref{extenofcharwei} gives the extendiblility of weight characters of unipotent blocks of special linear and unitary groups in non-defining characteristic, while Section \ref{proofof1.3} proves Theorem \ref{ibawmain}.

\section{Notation and preliminaries}\label{preliminaries}

In this section we establish the notation around groups and characters that is used
throughout this paper.

\begin{notation}
	The cardinality of a set, or the order of a finite group, $X$, is denoted by $|X|$.
	If a group $A$ acts on a finite set $X$, we denote by $A_x$ the stabilizer of $x\in X$ in $A$, analogously we denote by $A_{X'}$ the setwise stabilizer of $X'\subseteq X$.
	
	Let $\ell$ be a prime.
		If $A$ acts on a finite group $H$ by automorphisms, then there is a natural action of $A$ on $\Irr(H)\cup\IBr_\ell(H)$ given by ${}^{a^{-1}}\chi(g)=\chi^a(g)=\chi(g^{a^{-1}})$ for $g\in G$, $a\in A$ and $\chi\in\Irr(H)\cup\IBr_\ell(H)$.
	For $P\le H$ and $\chi\in\Irr(H)\cup\IBr_\ell(H)$, we denote by $A_{P,\chi}$ the stabilizer of $\chi$ in $A_P$.

	We denote the restriction of $\chi\in\Irr(H)\cup\IBr_\ell(H)$ to some subgroup $L\le H$ by $\Res^H_L\chi$, while $\Ind^H_L\psi$ denotes the character induced from $\psi\in\Irr(L)\cup\IBr_\ell(L)$ to $H$.
	For $N\unlhd H$ we sometimes identity the characters of $H/N$ with the characters of $H$ whose kernel contains $N$.
	
	For $N\unlhd H$, and $\chi\in\Irr(H)\cup\IBr_\ell(H)$, we denote by $\kappa^H_N(\chi)$ the number of irreducible constituents of $\Res^H_N(\chi)$ forgetting multiplicities.
	Let $B$ be an $\ell$-block of $H$, we denote by $\kappa^H_N(B)$ the number of $\ell$-blocks of $N$ covered by $B$.	
\end{notation}

\subsection{Clifford theory}

\begin{lem}\label{cliffordthm}
	Suppose that $H$ is a finite group and $N\unlhd H$ satisfies that $H/N$ is cyclic.
	\begin{enumerate}
		\item[(i)] Let $\chi\in\Irr(H)$ and $\theta\in\Irr(N\ |\ \chi)$, then every character in
		$\Irr(H\ |\ \theta)$ has the form $\chi\eta$ for some $\eta\in\Irr(H/N)$, and
		$\kappa^H_N(\chi)$ is equal to the cardinality of the set $\{\eta\in\Irr(H/N)\ |\ \chi\eta=\chi \}.$
		\item[(ii)]
		Let $\psi\in\IBr_\ell(H)$ and $\varphi\in\IBr_\ell(N\ |\ \psi)$,
		then every $\ell$-Brauer character in
		$\IBr_\ell(H\ |\ \varphi)$ has the form $\psi\tau$ for some $\tau\in\IBr_\ell(H/N)$ and
		the $\ell'$-part of $\kappa^H_N(\psi)$ is equal to the cardinality of the set $\{\tau\in\IBr_\ell(H/N)\ |\ \psi\tau=\psi \}.$
	\end{enumerate}
\end{lem}

\begin{proof}
This is a direct consequence of Clifford theory~(see, for example, \cite[\S 19]{Hu98} and \cite[Chap. 8]{Na98}).
	For (ii), see also \cite[Lem. 3.3 and 3.8]{KT09}.
\end{proof}

For a finite group $H$, we denote by $\Rad_\ell(H)$ the set of all $\ell$-radical subgroups of $H$ and $\Rad_\ell(H)/\thicksim_H$ a complete set of representatives of $H$-conjugacy classes of $\ell$-radical subgroups of $H$.

\begin{lem}\label{relationofradicalgroups}
	Let $H$ be a finite group, $N\unlhd H$ and $\ell$ a prime.
	\begin{enumerate}
		\item[(i)]  If $R$ is an $\ell$-radical subgroup of $H$, then $R\cap N$ is an $\ell$-radical subgroup of $N$.
		\item[(ii)] The map $ \Rad_\ell(H)\to \Rad_\ell(N)$, $R\mapsto R\cap N$ is surjective.
		\item[(iii)] Let $S$ be an $\ell$-radical subgroup of $N$. Assume that there is only one $\ell$-radical subgroup $R$ of $H$ such that $R\cap N=S$.
		Then $R=\mathcal O_\ell(N_H(S))$ and $N_H(S)=N_H(R)$.
	\end{enumerate}
\end{lem}

\begin{proof}
	(i) is \cite[(2.1)]{OU95}.
	For (ii),
	if $S$ is an $\ell$-radical subgroup of $N$, let $R=\mathcal O_\ell(N_H(S))$, then we claim that $R$ is an $\ell$-radical subgroup of $H$ with $R\cap N=S$.
	Indeed, $R\cap N$ is a normal $\ell$-subgroup of $N_N(S)$ and then $R\cap N\le S$ since $S=\mathcal O_\ell(N_N(S))$.
	Obviously $S\le R\cap N$. Thus $S= R\cap N$.
	Then $N_H(R)\le N_H(S)$.
	Now $R\unlhd N_H(S)$, so $N_H(R)= N_H(S)$.
	Then $R$ is an $\ell$-radical subgroup of $H$.
	Thus the claim holds and then (ii) holds and (iii) easily follows.
\end{proof}

\begin{lem}\label{corrofweights}
	Let $H$ be a finite group, $N\unlhd H$ and $\ell$ a prime.
	Assume that $H/N$ is cyclic and the map $ \Rad_\ell(H)\to \Rad_\ell(N)$, $R\mapsto R\cap N$ is bijective.
	\begin{enumerate}
		\item[(i)] If $(R,\varphi)$ is an $\ell$-weight of $H$, then $(S,\psi)$ is an $\ell$-weight of $N$ for $S=R\cap N$ and any $\psi\in\Irr(N_N(S)\ |\ \varphi)$.
		\item[(ii)] Let $(S,\psi)$ be an $\ell$-weight of $N$, and $R\in\Rad_\ell(H)$ such that $R\cap N=S$.
		Assume further that $\ell\nmid|N_H(R)_{\psi}/N_N(S)R|$.
		Then there exists an $\ell$-weight $(R,\varphi)$ of $H$ such that $\varphi\in\Irr(N_H(R)\ |\ \psi)$.
	\end{enumerate}
\end{lem}

\begin{proof}
	Let $R$ be an $\ell$-radical subgroup of $H$, $S=R\cap N$. By Lemma \ref{relationofradicalgroups} (iii), $N_H(R)= N_H(S)$ and then
	$N_N(S)=N_H(R)\cap N$.
	By the assumption, $N_H(R)/N_N(S)$ is cyclic.	
	Now $N_N(S)R/R\cong N_N(S)/S$, so there is a bijection $\Psi:\Irr(N_N(S)\ |\ 1_S)\to \Irr(N_N(S)R\ |\ 1_R)$ such that if $\psi\in \Irr(N_N(S)\ |\ 1_S)$ and
	$\psi'=\Psi(\psi)$, then $\psi'$ is an extension of $\psi$.
	Obviously, every character in $\Irr(N_N(S)\ |\ 1_S)$ is $R$-invariant.
	
	(i). Let $(R,\varphi)$ be an $\ell$-weight of $H$ and $\psi\in\Irr(N_N(S)\ |\ \varphi)$,
	then $\Res^{N_H(R)}_{N_N(S)}\varphi$ is multiplicity-free.
	So $\varphi(1)=t\psi(1)$ with $t=[N_H(R):N_H(R)_\psi]$.
	Hence $t\mid [N_H(R):N_N(S)R]$.
	Notice that $\varphi(1)_\ell=|N_H(R)/R|_\ell$, so $\psi(1)_\ell\ge|N_N(S)/S|_\ell$.
	Thus $\psi(1)_\ell=|N_N(S)/S|_\ell$.
	Hence $\psi$ is of $\ell$-defect zero as a character of $N_N(S)/S$ and then $(S,\psi)$ is an $\ell$-weight of $N$.
	
	(ii). Let $(S,\psi)$ be an $\ell$-weight of $N$, $\psi'=\Psi(\psi)$ and
	$\varphi\in\Irr(N_H(R)\ |\ \psi')$.
	Then the proof is similar to (i).
\end{proof}

\subsection{Blocks}
Let $H$ be a finite group, $\chi\in\Irr(H)$, $\ell$ a prime,
we denote by $\chi^\circ$ the restriction of $\chi$ to the set of all $\ell'$-elements of $H$.
Let $\theta$ be a linear character of $H$. Then $\theta\chi$ is an irreducible character of $H$ and the map $\chi\mapsto \theta\chi$ is a permutation on $\Irr(H)$.
Moreover, this permutation respects $\ell$-blocks.
The following is elementary.

\begin{lem}\label{per-block}
Suppose that $H$ is a finite group and $B$ is an $\ell$-block of $H$.
Let $\theta$ be a linear character of $H$ of $\ell'$-order.
Then there is an $\ell$-block of $H$, say $\theta\otimes B$, such that $\Irr(\theta\otimes B)=\{\theta\chi\mid \chi\in\Irr(B)\}$.
Moreover, $\IBr_\ell(\theta\otimes B)=\{\theta^\circ\phi\mid\phi\in\IBr_\ell(B)\}$.	
\end{lem}

\begin{proof}
	Let $(K,O,k)$ be a splitting $\ell$-modular system for $H$
	where $K$ is an extension of the $\ell$-adic field $\mathbb{Q}_\ell$.
	Let $J(O)$ be the maximal ideal of $O$ and ${}^*:O\to k=O/J(O)$	the canonical homomorphism.
	We denote by $\theta'$ be the linear character of $H$ such that $\theta'(h)=\theta(h)^{-1}$ for every $h\in H$.
	Now $\theta$ is of $\ell'$-order, so is $\theta'$.
	From this $\theta'$ induces a group homomorphism $\theta'^*:H\to k$.
	Define $\sigma:kH\to kH$ by $h\mapsto \theta'^*(h)h$. Then it is easy to check that $\sigma$ is an automorphism of $k$-algebra $kH$.	
	
	Now let $B'$ be the $\ell$-block of $H$ which is the image of $B$ under $\sigma$.
Then $\Irr(B)=\{\chi^{\sigma}\mid \chi\in\Irr(B)\}$ and $\IBr_\ell(B')=\{\phi^{\sigma}\mid \phi\in\IBr_\ell(B)\}$.
	For any $\chi\in \Irr(B)$ and $\phi\in\IBr_\ell(B)$,
	we see at once that $\chi^{\sigma}=\theta \chi$ and $\phi^{\sigma}=\theta^\circ \phi$.
	Now we take $\theta\otimes B=B'$, and we complete the proof.
\end{proof}

Let $Y\subseteq \IBr_\ell(H)$.
A subset $X \subseteq \Irr(H)$ is called a \emph{basic set} of $Y$ if
$\{\chi^\circ\mid \chi\in X\}$ is a $\mathbb Z$-basis of $\mathbb ZY$.
Let $\mathcal B$ be a union of some $\ell$-blocks of $H$.
If $Y=\IBr_\ell(\mathcal B)$, then we also say $X$ a  basic set of $\mathcal B$.

\begin{lem}\label{fromordtomod}
	Let $N\unlhd H$ be arbitrary finite groups, $B$ be an $\ell$-block of $N$ and $X\subseteq \Irr(B)$ a basic set of $B$. Suppose that the $\ell$-decomposition matrix associated with $X$ and $\IBr_\ell(B)$ is unitriangular with respect to a suitable ordering. Assume that every character in $X$ is invariant under $H$. Then every irreducible $\ell$-Brauer character of $B$ is invariant under $H$. Moreover, if every character in $X$ extends to $H$, then every irreducible $\ell$-Brauer character of $B$ extends to $H$.
\end{lem}

\begin{proof}
	This is \cite[Lem. 1.27 and Prop. 1.29]{Sc15}.
\end{proof}

We will make use of the following result.
\begin{lem}[{\cite[Lem. 2.3]{KS15}}]
	\label{covers}
	Let $K$ be a normal subgroup of finite group $H$ and $L$ a subgroup of $H$. Let $M=K\cap L$.
	Suppose that $b$ is an $\ell$-block of $M$ and $c$ is an $\ell$-block of $L$ such that $c$ covers $b$.
	If both $b^K$ and $c^H$ are defined, then $c^H$ covers $b^K$.
\end{lem}

\subsection{Cuspidal pairs}\label{defcuspidalpairs}

We will make use of the classification of the blocks of finite groups of Lie type in non-defining characteristic given by Cabanes-Enguehard \cite{CE99} and Kessar-Malle \cite{KM15}.
Algebraic groups are usually denoted by boldface letters.
Let $q$ be a power of some prime number $p$ and $\mathbb F_q$ the field of $q$ elements.
Suppose that $\mathbf G$ is a connected reductive linear algebraic group over the algebraic closure of $\mathbb F_q$ and $F:\mathbf G\to\mathbf G$ a Frobenius endomorphism endowing $\mathbf G$ with an $\mathbb F_q$-structure.
The group of rational points $\mathbf G^F$ is finite.
Let $\mathbf G^*$ be dual to $\mathbf G$ with corresponding Frobenius endomorphism also denoted $F$.

Let $d$ be a positive integer. We will make use of the terminology of Sylow $d$-theory~(see for instance \cite{BM92}).
For an $F$-stable maximal torus $\mathbf T$ of $\mathbf G$, denotes $(\mathbf T)_d$ its Sylow $d$-torus.
An $F$-stable Levi subgroup $\mathbf L$ of $\mathbf G$ is called \emph{$d$-split} if $\mathbf L=C_{\mathbf G}(Z^\circ (\mathbf L)_d)$, and $\zeta\in\Irr (\mathbf L^F)$ is called \emph{$d$-cuspidal} if ${}^*R^{\mathbf L}_{\mathbf M\subseteq \mathbf P}(\zeta)=0$ for all proper $d$-split Levi subgroups $\mathbf M < \mathbf L$ and any parabolic subgroup $\mathbf P$ of $\mathbf L$ containing $\mathbf M$ as Levi complement.

Let $s\in {\mathbf G^*}^F$ be semisimple.
Following \cite[Def.~2.1]{KM15},
we say $\chi\in\mathcal E(\mathbf G^F,s)$ is \emph{$d$-Jordan-cuspidal} if
\begin{itemize}
	\item $Z^\circ(C^\circ_{\mathbf G^*}(s))_d=Z^\circ(\mathbf G^*)_d$, and
	\item $\chi$ corresponds under Jordan decomposition~(see, for example, \cite[Prop. 5.1]{Lu88}) to the $C_{\mathbf G^*}(s)^F$-orbit of a $d$-cuspidal unipotent character of $C^\circ_{\mathbf G^*}(s)^F$.
\end{itemize}
If $\mathbf L$ is a $d$-split Levi subgroup of $\mathbf G$ and $\zeta\in\Irr(\mathbf L^F)$
is $d$-Jordan-cuspidal, then $(\mathbf L,\zeta)$ is called a \emph{$d$-Jordan-cuspidal pair} of $\mathbf G$.

Let $\ell$ be a prime number different from $p$.
Now we define an integer $e_0=e_0(q,\ell)$, which is denoted by ``$e$'' in \cite{KM15}~(in this paper, we will use ``$e$'' for another integer, see Section \ref{notations-and-conventions}):
\begin{equation}\label{definitionofe0}
e_0=e_0(q,\ell)= \text{multiplicative order of}\ q \ \text{modulo}\ \left\{\begin{array}{cl} \ell, & \quad \mbox{if}\ \ell>2,   \\
4, & \quad \mbox{if}\ \ell=2.
\end{array} \right.
\end{equation}

For a semisimple $\ell'$-element $s$ of ${\mathbf G^*}^F$,
we denote by $\mathcal E_{\ell}(\mathbf G^F,s)$ the union of all Lusztig series $\mathcal E(\mathbf G^F,st)$, where $t\in{\mathbf G^*}^F$ is a semisimple $\ell$-element commuting with $s$.
By \cite{BM89}, the set $\mathcal E_{\ell}(\mathbf G^F,s)$ is a union of $\ell$-blocks of $\mathbf G^F$.

Also, we denote by $\mathcal E(\mathbf G^F,\ell')$ the set of irreducible characters of $\mathbf G^F$ lying in a Lusztig series $\mathcal E(\mathbf G^F,s)$, where $s\in{\mathbf G^*}^F$ is a semisimple $\ell'$-element.

The paper \cite{CE99} gave a label for arbitrary $\ell$-blocks of finite groups of Lie type for $\ell\ge 7$ and it was generalised in \cite{KM15} to its largest possible generality.
Under the condition of \cite[Thm. A (e)]{KM15}, the set of $\mathbf G^F$-conjugacy classes of $e_0$-Jordan-cuspidal pairs $(\mathbf L,\zeta)$ of $\mathbf G$ such that $\zeta\in\mathcal E(\mathbf L^F,\ell')$, is a labeling set of the $\ell$-blocks of $\mathbf G^F$.

\subsection{The inductive blockwise Alperin weight conditions}

\begin{notation}
	For a finite group $H$ and a prime $\ell$, we denote by
\begin{itemize}
\item $\mathrm{dz}_\ell(H)$ the set of $\ell$-defect zero characters of $H$, and
\item $\mathrm{bl}_\ell(\varphi)$ the $\ell$-block of $H$ containing $\varphi$, for $\varphi\in\Irr(H)\cup\IBr_\ell(H)$.
\end{itemize}
If $Q$ is a radical $\ell$-subgroup of $H$ and $B$ an $\ell$-block of $H$, then we define the set
$$\mathrm{dz}_\ell(N_H(Q)/Q,B)=\{ \chi\in\mathrm{dz}_\ell(N_H(Q)/Q)\ | \ \mathrm{bl}_\ell(\chi)^H=B  \},$$
where we regard $\chi$ as an irreducible character of $N_G (Q)$ containing $Q$ in its kernel when considering the induced $\ell$-block $\mathrm{bl}_\ell(\chi)^H$.
\end{notation}

There are several versions of the (iBAW) condition.
Apart from the original version given in \cite[Def. 4.1]{Sp13}, there is also a version treating only blocks with defect groups involved in certain sets of $\ell$-groups \cite[Def. 5.17]{Sp13}, or a version handling single blocks \cite[Def. 3.2]{KS16}.
We shall consider the inductive condition for a single block here~(in order to consider unipotent $\ell$-blocks of special linear or unitary groups).

\begin{defn}[{\cite[Def. 3.2]{KS16}}]
	\label{induc}
	Let $\ell$ be a prime, $S$ a finite non-abelian simple group and $X$ the universal $\ell'$-covering group of $S$. Let $b$ be an $\ell$-block of $X$. We say the \emph{inductive blockwise Alperin weight (iBAW) condition} holds for $b$ if the following statements hold:
	\begin{enumerate}
		\item[(i)] There exist subsets $\IBr_\ell(b\ |\ Q)\subseteq \IBr_\ell(b)$ for $Q\in \Rad_\ell(X)$ with the following properties:
		\begin{enumerate}
			\item[(1)] $\IBr_\ell(b\ |\ Q)^a=\IBr_\ell(b\ |\ Q^a)$ for every $Q\in \Rad_\ell(X)$, $a\in \Aut(X)_b$,
			
			\item[(2)] $\IBr_\ell(b)=\dot{\bigcup}_{Q\in \Rad_\ell(X)/\thicksim_X}\IBr_\ell(b\ |\ Q)$.
		\end{enumerate}
		\item[(ii)] For every $Q\in \Rad_\ell(X)$ there exists a bijection $$\Omega^X_Q:\IBr_\ell(b\ |\ Q)\to \mathrm{dz}_\ell(N_X(Q)/Q,b)$$such that $\Omega^X_Q(\phi)^a=\Omega^X_{Q^a}(\phi^a)$ for every $\phi\in \IBr_\ell(b\ |\ Q)$ and $a\in \Aut(X)_b$.
		
		\item[(iii)] For every $Q\in \Rad_\ell(X)$ and every $\phi\in \IBr_\ell(b\ |\ Q)$ there exist a finite group $A:=A(\phi, Q)$ and $\tilde{\phi}\in \IBr_\ell(A)$ and $\tilde{\phi'}\in \IBr_\ell(N_A(\overline{Q}))$, where we use the notation $$\overline{Q}:=QZ/Z\ \text{and}\ Z:=Z(X)\cap \ker(\phi)$$ with the following properties:
		\begin{enumerate}
			\item[(1)] for $\overline{X}:=X/Z$ the group $A$ satisfies $\overline{X}\unlhd A$, $A/C_A(\overline{X})\cong \Aut(X)_\phi$, $C_A(\overline{X})=Z(A)$ and $\ell \nmid |Z(A)|$,
			
			\item[(2)] $\tilde{\phi}\in \IBr_\ell(A)$ is an extension of the $\ell$-Brauer character of $\overline{X}$ associated with $\phi$,
			
			\item[(3)] $\tilde{\phi'}\in \IBr_\ell(N_A(\overline{Q}))$ is an extension of the $\ell$-Brauer character of $N_{\overline{X}}(\overline{Q})$ associated with the inflation of $\Omega^X_Q(\phi)^\circ\in \IBr_\ell(N_X(Q)/Q)$ to $N_X(Q)$,
			
			\item[(4)] $\rm{bl}_\ell(\Res^A_J(\tilde{\phi}))=\rm{bl}_\ell(\Res^{N_A(\overline{Q})}_{N_J(\overline{Q})}(\tilde{\phi'}))^J$ for every subgroup $J$ satisfying $\overline{X}\le J \le A$.
		\end{enumerate}
	\end{enumerate}
\end{defn}

\subsection{Some notations and conventions for $\GL_n(\epsilon q)$}\label{notations-and-conventions}

From now on to the end of this paper, we always assume that $p$ is a prime, $q=p^f$ with a positive integer $f$, and $\ell$ is a prime number different from $p$.

We follow mainly the notation from \cite{FS82}, \cite{Br86}, \cite{An92},  \cite{An93} and \cite{An94}.
We first give some notation and conventions used throughout this paper.

For a positive integer $d$, we denote by $I_{(d)}$ the identity matrix of degree $d$
and by $I_{d}$ the identity matrix of degree $\ell^d$.
Let $\epsilon=\pm 1$ and $G=\GL_n(\epsilon q)$, where $\GL_n(-q)$ denotes the general unitary group $\GU_n(q)=\{ A\in \GL_n(q^2)\ |\ F_q(A)^{tr}A=I_{(n)} \}$, where $F_q(A)$ is the matrix whose entries are the $q$-th powers of the corresponding entries of $A$, and ${}^{tr}$ denotes the transpose operation of matrices.

Denote $X=\SL_n(\epsilon q)$, where $\SL_n(-q)=\SU_n(q)=\GU_n(q)\cap\SL_n(q^2)$.
We also use the notation $\GL(n,\epsilon q)$~(and $\SL(n,\epsilon q)$, respectively) for $\GL_n(\epsilon q)$~(and $\SL_n(\epsilon q)$, respectively).
Let $F_p$ be the automorphism of $G$ defined by $F_p((g_{ij}))=(g_{ij}^p)$ and $\gamma$ the automorphism of $G$ defined by $\gamma(A)=(A^{-1})^{tr}$.
Denote $D=\langle F_p, \gamma \rangle$.
Then $D$ is an abelian group of order $2f$ and the group $G\rtimes D$ is well-defined.
For the unitary groups, $D$ is cyclic.
By \cite[Thm. 2.5.1]{GLS98},
the automorphisms of $X$ induced by $G\rtimes D$ equal  $\Aut(X)$.
If $n=2$, $\gamma$ is an inner automorphism.
If $n\ge 3$, then
$\Aut(X)\cong G/Z(G)\rtimes D$.
We denote by $\mathbb F=\mathbb F_{\epsilon q}=\mathbb F_q$ or $\mathbb F_{q^2}$ the field of $q$ or $q^2$ elements when $\epsilon=1$ or $\epsilon=-1$ respectively.
Let $e$ be the multiplicative order of $\epsilon q$ modulo $\ell$.

For a positive integer $d$, we denote by $\F_{q^d}[x]$ ($\Irr(\F_{q^d}[x])$, respectively) the set of all polynomials (all monic irreducible polynomials, respectively) over the field $\F_{q^d}$.
For a polynomial $$\Delta(x)=x^m+a_{m-1}x^{m-1}+\cdots+a_0$$ in $\F_{_{ q^{2d}}}[x],$ we define $\tDelta(x)=x^ma_0^{-{q^d}}\Delta^{q^d}(x^{-1})$, where $\Delta^{q^d}(x)$ means the polynomial in $x$ whose coefficients are the $q^d$-th powers of the corresponding coefficients of $\Delta(x)$.
Then $\alpha$ is a root of $\Delta$ if and only if $\alpha^{-q^d}$ is a root of $\tDelta$.
Now, we denote by
\begin{align*}
\cF_{0}(d)&=\left\{ \Delta\in\Irr(\F_{q^d}[x])\mid \Delta\neq x ~\right\},\\
\cF_{1}(d)&=\left\{ \Delta\in\Irr(\F_{q^{2d}}[x])\mid \Delta\neq x,\Delta=\tDelta ~\right\},\\
\cF_{2}(d)&=\left\{~ \Delta\tDelta ~|~ \Delta\in\Irr(\F_{q^{2d}}[x]),\Delta\neq x,\Delta\neq\tDelta ~\right\}.
\end{align*}
Let
\begin{equation} \label{def-cfd}
\cF(d)=\left\{ \begin{array}{ll} \cF_0(d) &\quad \textrm{if}~\epsilon^d=1,\\
\cF_1(d)\cup\cF_2(d) &\quad \textrm{if}~\epsilon^d=-1. \end{array}
\right.
\end{equation}
In particular, we abbreviate $\cF:=\cF(1)$ and $\cF_{i}:=\cF_{i}(1)$ for $i=0,1,2$.
We denote by $d_\Gamma$ the degree of any polynomial $\Gamma$.
For unitary groups, the polynomials in $\cF_1\cup\cF_2$ serve as the ``elementary divisors'' as polynomials in $\cF_0$ serve for linear groups~(see, for example, \cite[p.111-112]{FS82}).
For  $\Gamma\in\cF$, if $\sigma$ is a root of $\Gamma$, then $\sigma^{(\epsilon q)^h}$ is also a root of $\Gamma$ for any positive integer $h$.
So $d_\Gamma$ is the smallest integer $d$ such that $\sigma^{(\epsilon q)^d-1}=1$
and all the roots of $\Gamma$ are $\sigma$, $\sigma^{\epsilon q}$, $\ldots$, $\sigma^{(\epsilon q)^{d_\Gamma-1}}$.

Note that the meaning of our notation here for unitary groups, such as $e$ and $\GU_n(q)$, is the same as those in \cite{An93} and \cite{An94} which is slightly different from that in \cite{FS82}~(for details, see \cite[p.6]{An94}).
In particular, with the notation adopted here, there is no need to introduce the reduced degrees $\delta_\Gamma$ for the unitary groups.~(For the results in \cite{FS82} for unitary groups where $\delta_\Gamma$ appears, it is easy to reformulate them with the notation adopted here and $d_\Gamma$ replacing $\delta_\Gamma$ as in \cite{An94}).

Let $\barF$ be the algebraic closure of $\F_q$.
As usual, we denote $\bG=\GL_n(\barF)$~(a connected reductive linear algebraic group).
Define $F_q:=F_p^f$ and $F=\gamma^{\frac{1-\epsilon}{2}} \circ F_q$ which is a Frobenius endomorphism over $\bG$ defining an $\F_q$ structure on it.
We write $\bG^F$ for the group of fixed points,
then $G=\bG^F$.

Now, for $\Gamma\in\cF$, let $(\Gamma)$ be the companion matrix of $\Gamma$.
Let $s$ be a semisimple element of $G$ and $s=\prod_\Gamma s_\Gamma$ is the primary decomposition of $s$~(see, for example,  \cite[p.112]{FS82}).
If the multiplicity $m_\Gamma(s)$ of $\Gamma$ in $s_\Gamma$ is not zero, we call $\Gamma$ an ``elementary divisor'' of $s$ although $\Gamma$ may not be irreducible in the unitary case.
Then there exists $g_\Gamma(s)$ such that $s_\Gamma^{g_\Gamma(s)}=I_{(m_\Gamma(s))}\otimes \diag(\sigma_\Gamma, \sigma_\Gamma^{\epsilon q}, \cdots, \sigma_\Gamma^{(\epsilon q)^{d_\Gamma-1}} )$ where $\sigma_\Gamma, \sigma_\Gamma^{\epsilon q}, \cdots, \sigma_\Gamma^{(\epsilon q)^{d_\Gamma-1}}$ are all the roots of $\Gamma$, and $v_\Gamma(s)=g_\Gamma(s)^{-1}F(g_\Gamma(s))$ is a blockwise permutation matrix corresponding to a $d_\Gamma$-cycle.
Now let $\bH=C_{\bG}(s)$, then $\bH=\prod_\Gamma\bH_\Gamma$, where $\bH_\Gamma=C_{\bG_\Gamma}(s_\Gamma)$ with $\bG_\Gamma = \GL(m_{\Gamma}(s)d_\Gamma,\barF)$.
Let $\bH_{\Gamma,0}:=\bH_\Gamma^{g_\Gamma(s)}$, then $\bH_{\Gamma,0}=\GL(m_{\Gamma}(s),\barF)\times\cdots\times\GL(m_{\Gamma}(s),\barF)$ with $d_\Gamma$ factors and $F$ acts on $\bH_\Gamma$ in the same way as $v_\Gamma(s)F$ acts on $\bH_{\Gamma,0}$.
Let $H_\Gamma=\bH_\Gamma^F$, then by \cite[Prop. (1A)]{FS82}, $H_\Gamma\cong \bH_{\Gamma,0}^{v_\Gamma(s)F}\cong \GL(m_{\Gamma}(s),(\epsilon q)^{d_\Gamma})$.
Also, $C_G(s)=\bH^F=\prod_{\Gamma}H_\Gamma$.
Let $\mathcal P(s)$ be the set of the symbols $\mu=\prod_{\Gamma}\mu_\Gamma$, such that $\mu_\Gamma$ is a partition of $m_\Gamma(s)$.
Then the unipotent characters of $C_G(s)$ are in bijection with $\Irr\left(\prod_\Gamma\fS(m_{\Gamma}(s)\right)$ and consequently with $\mathcal P(s)$~(see, for example, \cite[\S4.B2]{Br86}).
For $\mu\in\mathcal P(s)$, we denote by $\chi_\mu=\prod_{\Gamma}\chi_{\mu_\Gamma}$ the unipotent character of $C_G(s)$ corresponding to $\mu$.

\section{The characters and Brauer characters of $\SL_n(\epsilon q)$}
\label{Brauercharacterofslsu}

With the parametrization of pairs involving semisimple elements above, the irreducible
characters of $G$ can be constructed by the Jordan decomposition.
The irreducible characters of $G$ are in bijection with $G$-conjugacy classes of pairs $(s,\mu)$, where $s$ is a semisimple element of $G$ and $\mu\in\mathcal P(s)$.
The bijection is given as $$\chi_{s,\mu}=\epsilon_{\mathbf G} \epsilon_{C_{\mathbf{G}}(s)} \mathrm R ^{\mathbf G}_{C_{\mathbf{G}}(s)} (\hat s \chi_\mu),$$
where $\chi_\mu$ is a unipotent character of $H=C_{\mathbf G^F}(s)$ described as in the end of previous section,
and
$\hat s$ is the image of $s$ under the isomorphism~(see \cite[(1.16)]{FS82})
\begin{equation}\label{isomorphsimhat}
Z(H)\cong {\rm Hom}(H/[H,H],\overline{\mathbb{Q}}^\times_\ell).
\end{equation}
Here, $\overline{\mathbb{Q}}_\ell$ is an algebraic closure of the $\ell$-adic field $\mathbb{Q}_\ell$.

Let
\begin{equation}\label{def-frakZ}
\mathfrak Z:=\{ z\in \mathbb F^\times\ |\ z^{q-\epsilon}=1\}.
\end{equation}
Then we may identify the elements of $\mathfrak Z$ with the elements of $Z(G)$.
For $\Gamma\in\cF$, let $\xi$ be a root of $\Gamma$.
For $z\in \mathfrak Z$, define $z.\Gamma$ to be the unique polynomial in $\cF$ such that $z\xi$ is a root of $z.\Gamma$.
Note that $d_{\Gamma}=d_{z.\Gamma}$.
In fact, since all the roots of $\Gamma$ are $\xi,\xi^{\epsilon q},\ldots, \xi^{(\epsilon q)^{d_\Gamma-1}}$, we know that all the roots of $z.\Gamma$ are $z\xi,z\xi^{\epsilon q},\ldots, z\xi^{(\epsilon q)^{d_\Gamma-1}}$.
Now we define an action of $\mathfrak Z$ on the set of pairs $(s,\mu)$ with $\mu\in\mathcal P(s)$.
For  $z\in \mathfrak Z$,
define $z\mu=\prod_{\Gamma}(z\mu)_\Gamma$ with $(z\mu)_{z.\Gamma}=\mu_\Gamma$.
Then $z\mu\in\mathcal P(zs) $.

By Lemma \ref{cliffordthm}, for $\chi\in\Irr(G)$, in order to compute the number of irreducible constituents of $\Res^G_X(\chi)$~(recall that $X=\SL_n(\epsilon q)$ is defined as in Section \ref{notations-and-conventions}), we need to know when $\chi\zeta=\chi$, for $\zeta\in\Irr(G/X)$.
Note that the group $Z(G)$~(or $\mathfrak Z$)~ is isomorphic via $\ \hat{}\ $ to the group of linear characters of $G/X$.
The following proposition follows from \cite[Prop. 3.5]{De17}.

\begin{prop}\label{restrofordi}
$\hat z\chi_{s,\mu}=\chi_{zs,z\mu}$ for $z\in\mathfrak Z$.
\end{prop}

Thus, for a semisimple element $s\in\mathbf G^F$ and $z\in\mathfrak Z$,
if we write $\mathcal E(\mathbf G^F,s)=\{  \chi_1,\dots,\chi_k   \}$,
then
\begin{equation}\label{act-z-series}
\mathcal E(\mathbf G^F,zs)=\{  \hat z\chi_1,\dots,\hat z\chi_k  \}.
\end{equation}

If $z\in\mathcal O_{\ell'}(\mathfrak Z)$, we may regard $\hat z$ as an irreducible $\ell$-Brauer character of $G/X$.
Then the group $\mathcal O_{\ell'}(\mathfrak Z)$ is isomorphic via $\ \hat{}\ $ to the group of linear $\ell$-Brauer characters of $G$.
Recall that for a semisimple $\ell'$-element $s$ of $\mathbf G^F$, $\mathcal E_{\ell}(\mathbf G^F,s)$ is a union of $\ell$-blocks of $\mathbf G^F$~(cf. \cite{BM89}).
Then by (\ref{act-z-series}) and Lemma \ref{per-block}, we have:

\begin{cor}\label{corrofbrauseries}
	Let $s$ be a semisimple $\ell'$-element of $\mathbf G^F$.
	Suppose that $\IBr_\ell(\mathcal E_{\ell}(\mathbf G^F,s))=\{  \phi_1,\dots,\phi_k  \}$,
	then $\IBr_\ell(\mathcal E_{\ell}(\mathbf G^F,zs))=\{ \hat z \phi_1,\dots,\hat z\phi_k  \}$ for any $z\in\mathcal O_{\ell'}(\mathfrak Z)$.
\end{cor}

\begin{rmk}\label{decompositionmtx}
By \cite[Thm.~5.1]{GH91}, $\mathcal E(\mathbf G^F,s)$ is a basic set of $ \mathcal E_{\ell}(\mathbf G^F,s)$.
By the proof of Lemma \ref{per-block}, with a suitable ordering, the decomposition matrices associated with the basic sets $\mathcal E(\mathbf G^F,s)$ and $\mathcal E(\mathbf G^F,zs)$
of $ \mathcal E_{\ell}(\mathbf G^F,s)$ and $\mathcal E_{\ell}(\mathbf G^F,zs)$, respectively, are the same.
\end{rmk}

Now we may use the parameterisation $(s,\mu)$ of irreducible characters in $\mathcal E(\mathbf G^F,s)$ for the irreducible $\ell$-Brauer characters of $ \mathcal E_{\ell}(\mathbf G^F,s)$.
Let $\phi_{s,\mu}$ denote the irreducible $\ell$-Brauer characters corresponding to $(s,\mu)$.
Then it is convenient to assume that $\hat z\phi_{s,\mu}=\phi_{zs,z\mu}$ for all $z\in\mathcal O_{\ell'}(\mathfrak Z)$ by Corollary \ref{corrofbrauseries}.~(For $\epsilon=1$, this is just \cite[Lem. 4.1]{KT09}.)

The number of irreducible constituents of the restriction of irreducible $\ell$-Brauer characters of $G$ to $X$ was obtained by Kleshchev and  Tiep for $\epsilon=1$~(see {\cite[Thm. 1.1 and Cor. 1.2]{KT09}}),
and generalized by Denoncin for $\epsilon=\pm 1$~ (see \cite[Prop. 3.5, 4.2 and 4.9]{De17}).
We will state it as the following remark.

\begin{rmk}\label{brasuerofslsu}
We introduce the notations of the combinatorial description of irreducible $\ell$-Brauer characters of $G$ used in \cite{KT09}. 	For a partition $\mu=(\mu_1,\mu_2,\ldots)$, denote $|\mu|=\mu_1+\mu_2+\cdots$ and write $\mu'$ for the transposed partition.
Set $\Delta(\mu)=\mathrm{gcd}(\mu_1,\mu_2,\ldots)$.
	
For $\sigma\in\overline{\mathbb F}^\times$, we denote by $[\sigma]$ the set of all roots of the polynomial in $\cF$ which has $\sigma$ as a root.
Denote by $\mathrm{deg}(\sigma)$ the cardinality of $[\sigma]$.
Then $\mathrm{deg}(\sigma)$ is the minimal integer $d$ such that $\sigma^{(\epsilon q)^d-1}=1$
and $$[\sigma]=\{\ \sigma, \sigma^{\epsilon q}, \sigma^{(\epsilon q)^2},\ldots,  \sigma^{(\epsilon q)^{\mathrm{deg}(\sigma)-1}}    \ \}.$$
An \emph{$(n,\ell)$-admissible tuple} is a tuple
\begin{equation}\label{admtup}
(([\sigma_1],\mu^{(1)}),\dots,([\sigma_a],\mu^{(a)}))
\end{equation}
 of pairs, where $\sigma_1,\dots,\sigma_a\in \overline{\mathbb F}^\times$ are $\ell'$-elements, and $\mu^{(1)}, \dots, \mu^{(a)}$ are partitions such that
\begin{itemize}
	\item $[\sigma_i]\ne [\sigma_j]$ for all $i\ne j$, and
	\item $\sum\limits^{a}_{i=1}\mathrm{deg}(\sigma_i)|\mu^{(i)}|=n$.
\end{itemize}

An equivalence class of the $(n,\ell)$-admissible tuple (\ref{admtup}) up to a permutation of pairs $$([\sigma_1],\mu^{(1)}),\ldots,([\sigma_a],\mu^{(a)})$$ is called an \emph{$(n,\ell)$-admissible symbol} and is denoted as
\begin{equation}\label{admsym}
\mathfrak s=[([\sigma_1],\mu^{(1)}),\dots,([\sigma_a],\mu^{(a)})].
\end{equation}
The set of $(n,\ell)$-admissible symbols is the labeling set for irreducible $\ell$-Brauer characters of $G$.
Denote by $\phi_{\mathfrak s}$ the irreducible $\ell$-Brauer character corresponding to the $(n,\ell)$-admissible symbol $\mathfrak s$.

The group $\mathcal O_{\ell'}(\mathfrak Z)$ acts on the set of $(n,\ell)$-admissible symbols via
$$z\cdot [([\sigma_1],\mu^{(1)}),\dots,([\sigma_a],\mu^{(a)})]=[([z\sigma_1],\mu^{(1)}),\dots,([z\sigma_a],\mu^{(a)})]$$
for $z\in\mathcal O_{\ell'}(\mathfrak Z)$.
We denote by $\kappa_{\ell'}(\mathfrak s)$ the order of the stabilizer group in $\mathcal O_{\ell'}(\mathfrak Z)$ of an $(n,\ell)$-admissible symbol $\mathfrak s$.
Next, for an $(n,\ell)$-admissible symbol $\mathfrak s$ as (\ref{admsym}), let
$\kappa_{\ell}(\mathfrak s)$ be the $\ell$-part of $$\mathrm{gcd}(n,q-1,\Delta((\mu^{(1)})'),\cdots, \Delta((\mu^{(a)})')).$$

Let $\kappa(\mathfrak s)=\kappa_{\ell}(\mathfrak s)\kappa_{\ell'}(\mathfrak s)$.
By \cite{KT09} and \cite{De17},
$\kappa^G_X(\phi_{\mathfrak s})=\kappa(\mathfrak s)$~(\emph{i.e.} $\Res^{G}_{X}\phi_{\mathfrak s}$ is a sum of $\kappa(\mathfrak s)$ irreducible constituents).
For two $(n,\ell)$-admissible symbols $\mathfrak s$ and $\mathfrak s'$,
if they are in the same $\mathcal O_{\ell'}(\mathfrak Z)$-orbit, then $\Res^{G}_{X}\phi_{\mathfrak s}=\Res^{G}_{X}\phi_{\mathfrak s'}$.
	
If moreover, we write the decomposition $\Res^{G}_{X}\phi_{\mathfrak s}=\bigoplus^{\kappa(\mathfrak s)}_{j=1} (\phi_{\mathfrak s})_j$, then the set $\{(\phi_{\mathfrak s})_j\}$, where $\mathfrak s$ runs through the $\mathcal O_{\ell'}(\mathfrak Z)$-orbit representatives of $(n,\ell)$-admissible symbols and $j$ runs through the integers between $1$ and  $\kappa(\mathfrak s)$,
is a complete set of the irreducible $\ell$-Brauer characters of $X$.
\end{rmk}

Notice that Remark \ref{brasuerofslsu} also holds for complex irreducible characters if we set $\ell=0$ by Proposition \ref{restrofordi}. ~(For $\epsilon=1$, the complex irreducible characters of $\SL_n(q)$ were obtained in \cite{Le73}.)

For an $\ell$-block $B$ of $G$ and an $(n,\ell)$-admissible symbol $\mathfrak s$,
if $\phi_{\mathfrak s}\in\IBr_\ell(B)$,
then we say that $\mathfrak s$ \emph{belongs to} $B$.

\section{The blocks of $\SL_n(\epsilon q)$}
\label{blocksofslsuchap}

Let $\mathbf X=\SL_n(\overline{\F})$, then $\mathbf X=[\mathbf G,\mathbf G]$.
The labeling of $\ell$-blocks of $\mathbf G^F$ and $\mathbf X^F$~(using $e_0$-Jordan-cuspidal pairs)~ described in  \cite{CE99} and \cite{KM15} can be stated as following.

\begin{thm}\label{ecuspidalofblock}
	Let $\mathbf H\in\{\mathbf G, \mathbf X\}$ and $e_0=e_0(q,\ell)$ is defined as in Equation (\ref{definitionofe0}).
	\begin{enumerate}
		\item[(i)] For any $e_0$-Jordan-cuspidal pair $(\mathbf L,\zeta)$ of $\mathbf H$ such that
		$\zeta\in\mathcal E(\mathbf L^F,\ell')$, there exists a unique $\ell$-block $b_{\mathbf H^F}(\mathbf L,\zeta)$ of $\mathbf H^F$ such that all irreducible constituents of $R_{\mathbf L}^{\mathbf H}(\zeta)$ lie in $b_{\mathbf H^F}(\mathbf L,\zeta)$.
		\item[(ii)] Moreover, the map $\Xi: (\mathbf L,\zeta)\mapsto b_{\mathbf H^F}(\mathbf L,\zeta)$ is a surjection from the set of $\mathbf H^F$-conjugacy classes of $e_0$-Jordan-cuspidal pairs $(\mathbf L,\zeta)$ of $\mathbf H$ such that
		$\zeta\in\mathcal E(\mathbf L^F,\ell')$ to the $\ell$-blocks of $\mathbf H^F$.
		\item[(iii)] If $\ell$ is odd, then $\Xi$ is bijective.
	\end{enumerate}
\end{thm}

\begin{rmk}
By a result of Bonnaf\'e \cite{Bo00}, the Mackey formula holds for type $A$, hence the Lusztig induction in Theorem \ref{ecuspidalofblock} (i) is independent of the ambient parabolic subgroup~(containing $\mathbf L$).
Also, throughout this paper we always omit the parabolic subgroups when considering Lusztig inductions.	
\end{rmk}

Note that we let $e$ be the multiplicative order of $\epsilon q$ modulo $\ell$ throughout this paper.
Here, $e_0$ and $e$ may not equal.
In fact,
\begin{enumerate}
	\item[(i)] when $\ell$ is odd,
	\begin{itemize}
		\item if $\epsilon=1$, then $e=e_0$, and
		\item if $\epsilon=-1$, then $e=2e_0$, $e_0/2$, $e_0$ if $e_0$ is respectively odd, congruent to $2$ modulo $4$, or divisible by $4$, and
	\end{itemize}
	\item[(ii)]  when $\ell=2$, we have $e=1$ while $e_0=1$ or $2$ if $4\mid q-1$ or $4\mid q+1$ respectively.
\end{enumerate}

For a positive integer $d$, we let $\Phi_d(x) \in \mathbb Z[x]$ be the $d$-th cyclotomic polynomial over $\mathbb Q$, \emph{i.e.}, the monic irreducible polynomial whose roots are the primitive $d$-th roots of unity.
So if $\ell$ is odd, then $\Phi_e(\epsilon x)=\pm\Phi_{e_0}(x)$.

We will use the following lemma.

\begin{lem}\label{fortheblock}
	Assume that $\ell$ is odd.
Let $\lambda$ be an $e$-core of a partition of $n$, and $w=e^{-1}(n-|\lambda|)$.
Let $\mathbf T^{(e)}$ be a Coxeter torus of $(\GL(e,\overline{\mathbb F}),F)$, $\mathbf T= (\mathbf T^{(e)})^w\times I_{(|\lambda|)}$,
and $\mathbf L=C_{\mathbf G}(\mathbf T)=(\mathbf T^{(e)})^w\times \GL(|\lambda|,\overline{\mathbb F})$.
Let $\phi_\lambda$ be the unipotent character of $\GL(|\lambda|,\epsilon q)$ corresponding to $\lambda$ and $\phi=1_{\mathbf T^F}\times\phi_\lambda\in\Irr(\mathbf L^F)$.
Then every irreducible constituent of $R^{\mathbf G}_{\mathbf L}(\phi)$ has the form $\chi_\mu$ such that $\lambda$ is the $e$-core of $\mu$.
\end{lem}

\begin{proof}
	Let $\mathbf H=\GL(ew,\overline{\mathbb F})\times\GL(|\lambda|,\overline{\mathbb F})$,
	then $\mathbf H$ is an $F$-stable Levi subgroup of $\mathbf G$~(moreover, there exists a semisimple element $\rho\in\mathbf G^F$ such that $\mathbf H=C_{\mathbf G}(\rho)$).
	Then every irreducible constituent of $R^{\mathbf H}_{\mathbf L}(\phi)$ has the form $\phi_\nu\times \phi_\lambda$, where $\phi_\nu$ is a unipotent character of $\GL(ew,\epsilon q)$ corresponding to some $\nu \vdash we$.
	Since $R^{\mathbf G}_{\mathbf L}(\phi)=R^{\mathbf G}_{\mathbf H}(R^{\mathbf H}_{\mathbf L}(\phi))$,
	it suffices to prove that every irreducible constituent of $R^{\mathbf G}_{\mathbf H}(\phi_\nu\times \phi_\lambda)$ has the form $\chi_\mu$ such that $\lambda$ is the $e$-core of $\mu$ and then the result follows by \cite[(2.12)]{FS82}~(a result from the Murnaghan-Nakayama formula)~ and the remark following it.
\end{proof}

\begin{rmk}
In fact, with the hypothesis and setup of Lemma \ref{fortheblock}, the pair $(\mathbf L,\phi)$ is an
 $e_0$-cuspidal pair~(note that $\mathbf L=C_{\mathbf G}(\mathbf T_{e_0})$), and
the set of the irreducible constituents of $R^{\mathbf G}_{\mathbf L}(\phi)$ is exactly the $e_0$-Harish-Chandra series above $(\mathbf L,\phi)$.
So Lemma \ref{fortheblock} also follows from the proof of \cite[Thm. 3.2 and 3.3]{BMM93}.
\end{rmk}

Now we give the relationship between the  $e_0$-cuspidal pairs of $\mathbf G$  and the  $e_0$-cuspidal pairs of $\mathbf X$.
\begin{prop}\label{restrofesplit}
\begin{enumerate}
	\item[(i)] Let $(\mathbf L,\zeta)$ be an $e_0$-cuspidal pair of $\mathbf G$ and $b$ an $\ell$-block of $X$ covered by $B=b_{\mathbf G^F}(\mathbf L,\zeta)$, then $b=b_{\mathbf X^F}(\mathbf L_0,\zeta_0)$, where $\mathbf L_0=\mathbf L\cap \mathbf X$ and $\zeta_0$ is an irreducible constituent of $\Res^{\mathbf L^F}_{\mathbf L_0^F}\zeta$.
	\item[(ii)] Let $(\mathbf L_0,\zeta_0)$ be an $e_0$-cuspidal pair of $\mathbf X$ and $B$ an $\ell$-block of $G$ which covers $b=b_{\mathbf X^F}(\mathbf L_0,\zeta_0)$, then $B=b_{\mathbf G^F}(\mathbf L,\zeta)$ where the $e_0$-cuspidal pair $(\mathbf L,\zeta)$ satisfies that $\mathbf L_0=\mathbf L\cap \mathbf X$ and $\zeta_0$ is an irreducible constituent of $\Res^{\mathbf L^F}_{\mathbf L_0^F}\zeta$.
\end{enumerate}
\end{prop}

\begin{proof}
	Note that if $\mathbf L_0=\mathbf L\cap \mathbf X$ and $\zeta_0$ is an irreducible constituent of $\Res^{\mathbf L^F}_{\mathbf L_0^F}\zeta$, then by \cite[Lem. 2.3]{KM15}, $(\mathbf L,\zeta)$ is an $e_0$-cuspidal pair of $\mathbf G$ if and only if $(\mathbf L_0,\zeta_0)$ is an $e_0$-cuspidal pair of $\mathbf X$.
	Thus (i) follows by \cite[Lem. 3.7]{KM15}.
	
	For (ii), set $\mathbf L=\mathbf L_0Z(\mathbf G)$, then $\mathbf L_0=\mathbf L\cap \mathbf  X$.
	Also, $Z(\mathbf L)=Z(\mathbf L_0)Z(\mathbf G)$ and $Z(\mathbf L_0)_{e_0}\subseteq Z(\mathbf L)_{e_0}$ since $\mathbf G=Z(\mathbf G)\mathbf X$.
	Hence $C_{\mathbf G}(Z(\mathbf L)_{e_0})=Z(\mathbf G)C_{\mathbf X}(Z(\mathbf L)_{e_0})\subseteq Z(\mathbf G)C_{\mathbf X}(Z(\mathbf L_0)_{e_0})=Z(\mathbf G)\mathbf L_0=\mathbf L$, and then $\mathbf L$ is an $e_0$-split Levi subgroup of $\mathbf G$.
	Thus (ii) follow by \cite[Lem. 3.8]{KM15}.
\end{proof}

\begin{rmk}
	Proposition \ref{restrofesplit} is not restricted to the case of type $A$.
	In fact, it holds for any connected reductive linear algebraic group $\mathbf G$ and $\mathbf X=[\mathbf G,\mathbf G]$.
\end{rmk}

\begin{lem}\label{conofcharoflevi}
Let $\mathbf L$ be an $F$-stable Levi subgroup of $\mathbf G$, $\zeta\in\Irr(\mathbf L^F)$ and $\mathbf L_0=\mathbf L\cap \mathbf X$.
Let $\Delta:=\Irr(\mathbf L_0^F\mid \zeta)$, then  $N_{\mathbf X^F}(\mathbf L_0)_\Delta$ acts trivially on $\Delta$.
\end{lem}

\begin{proof}
Let $L=\mathbf L^F$ and $L_0=\mathbf L_0^F$.
Note that there exist integers $n_i$, $a_i$, $b_i$ ~($1\le i\le s$)~and $r$ such that $n_i\ne n_j$ for $i\ne j$ and
$L=L_0\times L_1^{b_1}\times \cdots \times L_s^{b_s}$ where $L_0\cong \GL(r,\epsilon q)$
and $L_i\cong \GL(n_i,(\epsilon q)^{a_i})$.
Then $N_{\mathbf G^F}(\mathbf L)= L_0\times\prod\limits_{1\le i\le s} N_i\wr \fS(b_i)$, where $N_i=\langle L_i,\sigma_i  \rangle$, $o(\sigma_i)=a_i$, and $\sigma_i$ act on $L_i\cong \GL(n_i,(\epsilon q)^{a_i})$ as a field automorphism of order $a_i$.
We denote by $\Out_{N_{\mathbf G^F}(\mathbf L)}(\mathbf L_0^F)$ the the subgroup of $\Out(\mathbf L_0^F)$ induced by $N_{\mathbf G^F}(\mathbf L)$~(\emph{i.e.} $\Out_{N_{\mathbf G^F}(\mathbf L)}(\mathbf L_0^F)\cong N_{\mathbf G^F}(\mathbf L)/\mathbf L_0^F Z(\mathbf L^F)$).
By comparing orders, we have $\Out_{N_{\mathbf G^F}(\mathbf L)}(\mathbf L_0^F)=\Out_{\mathbf L^F}(\mathbf L_0^F)\rtimes \Out_{N_{\mathbf X^F}(\mathbf L)}(\mathbf L_0^F)$ since $Z(\mathbf L_0^F)=Z(\mathbf L^F)\cap \mathbf L_0^F$.
Let $\Delta_0:=\Irr([\mathbf L^F,\mathbf L^F]\mid \zeta)$.

First, we consider the case $L=L_i=\GL(n_i,(\epsilon q)^{a_i})$. Then $N_{\mathbf G^F}(\mathbf L)=N_i$ and $\Out_{\mathbf L^F}([\mathbf L^F,\mathbf L^F])$ and $\Out_{N_{\mathbf X^F}(\mathbf L)}([\mathbf L^F,\mathbf L^F])$
commute.
Now by \cite[Thm. 4.1]{CS17}, there exists $\zeta_0\in \Delta_0$ such that $N_{\mathbf G^F}(\mathbf L)_{\zeta_0}=\mathbf L^F_{\zeta_0} N_{\mathbf X^F}(\mathbf L)_{\zeta_0}$.
So $\zeta_0$ is invariant under $N_{\mathbf X^F}(\mathbf L_0)_{\Delta}$ since $$\Out_{N_{\mathbf G^F}(\mathbf L)}([\mathbf L^F,\mathbf L^F])=\Out_{\mathbf L^F}([\mathbf L^F,\mathbf L^F])\times \Out_{N_{\mathbf X^F}(\mathbf L)}([\mathbf L^F,\mathbf L^F]).$$
Now $\mathbf L^F$ acts transitively on $\Delta_0$, then $N_{\mathbf X^F}(\mathbf L_0)_{\Delta_0}$ acts trivially on $\Delta_0$.
Hence $N_{\mathbf X^F}(\mathbf L_0)_{\Delta}$ acts trivially on $\Delta$
since the restriction of $\zeta$ to $[ L, L]$ is multiplicity-free.

Now we consider the case $L=L_i^{b_i}\cong \GL(n_i,(\epsilon q)^{a_i})^{b_i}$.
Then $N_{\mathbf G^F}(\mathbf L)=N_i\wr \fS(b_i)$.
Let $\zeta=\zeta_1\times \cdots\times\zeta_{b_i}$, where $\zeta_k\in \Irr(L_i)$ for $1\le k\le b_i$.
Then $\Delta_0=\prod_{1\le k\le b_i}\Delta_{0,k}$, where $\Delta_{0,k}=\Irr([L_i,L_i]\mid \zeta_k)$ for $1\le k\le b_i$.
Let $\zeta_0\in \Delta_0$ and $\zeta_0=\zeta_{0,1}\times\cdots\times \zeta_{0,b_i}$ where $\zeta_{0,k}\in \Delta_{0,k}$ for $1\le k\le b_i$.
Let $g\in N_{\mathbf G^F}(\mathbf L)$.
If $\zeta_0^g\in\Delta_0$, then without loss of generality, we may assume that $g=(\sigma_1,\ldots,\sigma_{b_i};\tau)$, where $\sigma_k\in N_i$, $\tau\in \fS(b_i)$ and $\tau=(1,\ldots, b_i)$.
Then $\zeta_0^g=\zeta_{0,b_i}^{\sigma_{b_i}}\times \zeta_{0,1}^{\sigma_1}\times\cdots \zeta_{0,b_i-1}^{\sigma_{b_i-1}}$.
Hence
there exist $l_1,\ldots, l_{b_i-1}\in L_i$ such that
$\zeta_{0,1}^{l_1}=\zeta_{0,b_i}^{\sigma_{b_i}}$ and $\zeta_{0,k}^{l_k}=\zeta_{0,k-1}^{\sigma_{k-1}}$ for $2\le k\le b_i-1$.
By the argument of above paragraph, it is easy to check that $\zeta_{0,b_i}^{l_{b_i}}=\zeta_{0,1}^{\sigma_1}$ for $l_{b_i}=(\prod\limits_{1\le k\le b_i-1}l_k)^{-1}$.
Now let $l=\diag(l_1,\ldots,l_{b_i})$, then $l\in L_0$ and $\zeta_0^l=\zeta_0^g$.
Then there exists $\zeta_0'\in \Delta$, such that $\zeta_0,\zeta_0^g\in\Irr([L,L]\mid \zeta_0')$ since $\Res_{[L, L]}^{L}\zeta$ is multiplicity-free.
So $N_{\mathbf X^F}(\mathbf L_0)_{\Delta}$ acts trivially on $\Delta$.

The assertion in general case now follows by reduction to the preceding cases.
\end{proof}

	Let $J$ be a subgroup of some general linear or unitary group $\GL_m(\epsilon q)$, we denote
	\begin{equation}\label{def-mathcalD}
	\mathcal D(J):=\{ \mathrm{det}(M)\ |\   M\in J \}.
	\end{equation}
	Then $\mathcal D(J)$ is a subgroup of $\mathfrak Z$~(where $\mathfrak Z$ is defined as in (\ref{def-frakZ})) and $J/(J\cap \SL_m(\epsilon q))\cong \mathcal D(J)$.

\begin{rmk}\label{esplitofslsu}
	Let $\mathbf L$ a Levi subgroup of $\mathbf G$, and $\mathbf L_0=\mathbf L\cap \mathbf X$.
	Note that $\mathcal D(\mathbf L^F)=\mathfrak Z$. Then $\mathbf G^F=\mathbf X^FN_{\mathbf G^F}(\mathbf L)$ and $\mathbf L^F/\mathbf L_0^F\cong \mathbf G^F/\mathbf X^F$.
	So the $\mathbf G^F$-conjugacy classes of $e_0$-split Levi subgroups of $\mathbf G$ are just the $\mathbf X^F$-conjugacy classes of $e_0$-split Levi subgroups of $\mathbf G$.
	
We denote by $\widetilde{\mathcal L}$ a complete set of representatives of the $\mathbf G^F$-conjugacy classes of $e_0$-Jordan-cuspidal pairs of $\mathbf G$ such that
	$\zeta\in\mathcal E(\mathbf L^F,\ell')$.
	We may assume that for $(\mathbf L,\zeta)$, $(\mathbf L',\zeta')\in \widetilde{\mathcal L}$, if  $\mathbf L$ and $\mathbf L'$ are $\mathbf G^F$-conjugacy, then $\mathbf L=\mathbf L'$.
	
	Define an equivalence relation on $\widetilde{\mathcal L}$ :  $(\mathbf L,\zeta)\sim (\mathbf L',\zeta')$ if and only if $\mathbf L=\mathbf L'$ and $\Res^{\mathbf L^F}_{\mathbf L_0^F}\zeta=\Res^{\mathbf L^F}_{\mathbf L_0^F}\zeta'$ where $\mathbf L_0=\mathbf L\cap \mathbf X$.
	Then by Lemma \ref{cliffordthm}, \ref{conofcharoflevi} and Proposition \ref{restrofesplit}, $\{(\mathbf L\cap \mathbf X,\zeta_0)\}$ is a complete set of representatives of $\mathbf X^F$-conjugacy classes of $e_0$-Jordan-cuspidal pairs of $\mathbf X$ such that
	$\zeta_0\in\mathcal E((\mathbf L\cap\mathbf X )^F,\ell')$, where $(\mathbf L,\zeta)$ runs through a complete set of representatives of the equivalence classes of $\widetilde{\mathcal L}/\sim$ and $\zeta_0$ runs through $\Irr( (\mathbf L\cap \mathbf X)^F \ |\ \zeta)$.
\end{rmk}

The $\ell$-blocks of $\mathbf G^F$ were classified in \cite{FS82} and \cite{Br86}.
For $\Gamma\in\cF$, we denote by $e_\Gamma$ the multiplicative order of $(\epsilon q)^{d_\Gamma}$ modulo $\ell$.
Obviously, $e_\Gamma=\frac{e}{\mathrm{gcd}(e,d_\Gamma)}$.
Note that for $\ell=2$, $e=e_\Gamma=1$.
Given a semisimple element $s$ of $\mathbf G^F$, let
$\mathcal{C}_\Gamma(s)$ be the set of $e_\Gamma$-cores of partitions of $m_\Gamma(s)$, and let
$\mathcal{C}(s)=\prod\limits_{\Gamma}\mathcal{C}_\Gamma(s)$. The following result is a combination of \cite[(5D) and (7A)]{FS82} and \cite[(3.2) and (3.9)]{Br86}.

\begin{thm} \label{blockcorres}
There is a bijection from the set of $\ell$-blocks of $G$ onto the set of $G$-conjugacy classes of pairs $(s, \lambda)$, where $s$ is a semisimple $\ell'$-element of $G$ and $\lambda\in \mathcal{C}(s)$.

Moreover, let $B$ be an $\ell$-block of $G$ with label $(s,\lambda)$. Then an irreducible character of $G$ of the form $\chi_{t, \mu}$ belongs to $B$ if and only if the $\ell'$-part of $t$ is $G$-conjugacy to $s$ and for every $\Gamma\in\mathcal F$, $\mu_\Gamma$ has $e_\Gamma$-core $\lambda_\Gamma$.
\end{thm}

We denote by $B(s,\lambda)$ the $\ell$-block of $G$ with label $(s,\lambda)$.
Note that, for $\ell=2$, $(s,\lambda)$ is always of the form $(s,-)$~(here, ``$-$" denotes the empty partition).

\vspace{2ex}
Now we give an $e_0$-Jordan-cuspidal pair for the $\ell$-block $B(s,\lambda)$.
Let $s\in\mathbf G^F$ be a semisimple element and $\lambda\in\mathcal C(s)$.
Take the primary decomposition $s=\prod_{\Gamma}s_\Gamma$ with $s_\Gamma=m_\Gamma(s)(\Gamma)$.
Then $C_{\mathbf G}(s)=\prod_{\Gamma}C_{\mathbf G_\Gamma}(s_\Gamma)$ with $\mathbf G_\Gamma=\GL(m_\Gamma(s)d_\Gamma,\overline {\mathbb F})$.
Let $w_\Gamma(s)=e_\Gamma^{-1}(m_\Gamma(s)-|\lambda_\Gamma|)$.

First, we assume that $\ell$ is odd.
Let $\mathbf T_{e_\Gamma,0}$ be a Coxeter torus of $(\GL(e_\Gamma, \overline{\F}),F)$,
$\mathbf M_{\Gamma,1}=(\mathbf T_{e_\Gamma,0})^{w_\Gamma(s)}\times\GL(|\lambda_\Gamma|,\overline {\mathbb F})$
and $\mathbf H_{\Gamma,0}=\mathbf M_{\Gamma,1}\times\cdots\times \mathbf M_{\Gamma,1}$ with $d_\Gamma$ factors.
Let $\mathbf H_\Gamma={}^{g_\Gamma(s)}\mathbf H_{\Gamma,0}\le C_{\mathbf G_\Gamma}(s_\Gamma)$.
Then $\mathbf H_\Gamma^F \cong \mathbf H_{\Gamma,0}^{v_\Gamma(s)F}\cong (\GL(1,(\epsilon q)^{e_\Gamma d_\Gamma}))^{w_\Gamma(s)}\times \GL(|\lambda_\Gamma|,(\epsilon q)^{d_\Gamma})$.
Let  $\mathbf H=\Pi_\Gamma \mathbf H_\Gamma$.
Obviously $s\in \mathbf H^F$.

Now let $\mathbf T_{\Gamma,0}=((\mathbf T_{e_\Gamma,0})^{w_\Gamma(s)}\times I_{(|\lambda_\Gamma|)})^{d_\Gamma}$.
Then ${}^{g_\Gamma(s)}\mathbf T_{\Gamma,0}$ is a torus of $C_{\mathbf G_\Gamma}(s_\Gamma)$.
Now let $\mathbf T_\Gamma$ be the Sylow $e_0$-torus of ${}^{g_\Gamma(s)}\mathbf T_{\Gamma,0}$.
Then $\mathbf T=\Pi_\Gamma \mathbf T_\Gamma$ is an $e_0$-torus of $\mathbf G$.
Let $\mathbf L=C_{\mathbf G}(\mathbf T)$, then $\mathbf L$ is an $e_0$-split Levi subgroups of $\mathbf G$.
Also, $s\in\mathbf L$ and $\mathbf H=C_{C_{\mathbf G}(s)}(\mathbf T)=C_{\mathbf L}(s)$.

Let $\phi_{\lambda_\Gamma}$ be the unipotent character of $\GL(|\lambda_\Gamma|,(\epsilon q)^{d_\Gamma})$ corresponding to $\lambda_\Gamma$ and
$\phi_\Gamma=1_{(\GL(1,(\epsilon q)^{e_\Gamma d_\Gamma}))^{w_\Gamma(s)}}\times\phi_{\lambda_\Gamma}\in\Irr(\mathbf H_\Gamma^F)$.
Then $\phi=\Pi_\Gamma \phi_{_\Gamma}$ is an $e_0$-cuspidal unipotent character of $\mathbf H^F$.
Let $\zeta\in\mathcal E(\mathbf L^F,s)$ correspond under Jordan decomposition to
$\phi\in\mathcal E(\mathbf H^F,1)$.
Then $\zeta=\varepsilon_{\mathbf L}\varepsilon_{\mathbf H}R_{\mathbf H}^{\mathbf L}(\hat s\phi)$ is $e_0$-Jordan-cuspidal.
We denote $\mathbf L=\mathbf L_{s,\lambda}$ and $\zeta=\zeta_{s,\lambda}$.

Now we assume that $\ell=2$. Then $\lambda$ is empty.

Let $\mathbf T_{\Gamma,0}$ be the maximal torus of $(\GL(m_\Gamma(s),\overline{\F}),F)$ satisfying that
\begin{enumerate}
\item[(1)] if $4\mid q-\epsilon$ or $d_\Gamma$ is even, $\mathbf T_{\Gamma,0}$ consists of all diagonal matrices,
\item[(2)] if $4\mid q+\epsilon$ and $d_\Gamma$ is odd,
\begin{itemize}
\item $\mathbf T_{\Gamma,0}^F\cong \GL(1,q^2)^{\frac{m_\Gamma(s)}{2}}$ if $m_\Gamma(s)$ is even,
\item $\mathbf T_{\Gamma,0}^F\cong \GL(1,q^2)^{\frac{m_\Gamma(s)-1}{2}}\times \GL(1,\epsilon q)$ if $m_\Gamma(s)$ is odd.
\end{itemize}
\end{enumerate}

Let $\mathbf H_{\Gamma,0}=\mathbf T_{\Gamma,0}\times\cdots\times \mathbf T_{\Gamma,0}$ with $d_\Gamma$ factors.
Let $\mathbf H_\Gamma={}^{g_\Gamma(s)}\mathbf H_{\Gamma,0}\le C_{\mathbf G_\Gamma}(s_\Gamma)$.
Then $\mathbf H_\Gamma$ is a maximal torus of $C_{\mathbf G_\Gamma}(s_\Gamma)$.
Let  $\mathbf H=\Pi_\Gamma \mathbf H_\Gamma$,
then $s\in\mathbf H^F$.
Let $\mathbf T_\Gamma$ be the Sylow $e_0$-torus of $\mathbf H_\Gamma$, then  $\mathbf T=\Pi_\Gamma \mathbf T_\Gamma$ is an $e_0$-torus of $\mathbf G$.
Let $\mathbf L=C_{\mathbf G}(\mathbf T)$, then $\mathbf L$ is an $e_0$-split Levi subgroup of $\mathbf G$.
Also, $s\in\mathbf L^F$ and $\mathbf H=C_{C_{\mathbf G}(s)}(\mathbf T)=C_{\mathbf L}(s)$.
Let $\zeta\in\mathcal E(\mathbf L^F,s)$ correspond under Jordan decomposition to the ~(unique)~ $e_0$-cuspidal unipotent character
$\phi=1_{\mathbf H^F}\in\mathcal E(\mathbf H^F,1)$.
Then $\zeta$ is $e_0$-Jordan-cuspidal.
We denote $\mathbf L=\mathbf L_{s,\lambda}$ and $\zeta=\zeta_{s,\lambda}$.

\vspace{2ex}

Thus $(\mathbf L_{s,\lambda}, \zeta_{s,\lambda})$ is an $e_0$-Jordan-cuspidal pair of $\mathbf G$ in both the cases $\ell$ is odd and the case $\ell=2$.

\begin{prop}\label{corrofblockspara}
Suppose that $s\in\mathbf G^F$ is a semisimple $\ell'$-element, $\lambda\in\mathcal C(s)$
and $B(s,\lambda)$ is an $\ell$-block of $\mathbf G^F$.
Then
$b_{\mathbf G^F}(\mathbf L_{s,\lambda}, \zeta_{s,\lambda})=B(s,\lambda)$.
\end{prop}

\begin{proof}
Abbreviate $\zeta=\zeta_{s,\lambda}$.
It is obvious for the case that $\ell=2$.
Now we assume that $\ell$ is odd and it suffices to prove that every irreducible constituent of $R^{\mathbf G}_{\mathbf L}\zeta$ lies in $B(s,\lambda)$ by Theorem \ref{ecuspidalofblock}.
First, with the notation above,
$$R_{\mathbf L}^{\mathbf G}(\zeta)=\varepsilon_{\mathbf L}\varepsilon_{\mathbf H}R_{\mathbf L}^{\mathbf G}(R_{\mathbf H}^{\mathbf L}(\hat s \phi))=\varepsilon_{\mathbf L}\varepsilon_{\mathbf H}R_{C_{\mathbf G}(s)}^{\mathbf G}(\hat s R^{C_{\mathbf G}(s)}_{\mathbf H}(\phi)).$$
We note that $\mathbf H=C_{\mathbf L}(s)$ is an $F$-stable Levi subgroup of $C_{\mathbf G}(s)$ since $\mathbf H=C_{C_{\mathbf G}(s)}(\mathbf T)$.
Hence it suffices to prove that every irreducible constituent of $R^{C_{\mathbf G}(s)}_{\mathbf H}(\phi)$ has the form $\chi_\mu$ where $\mu=\prod_{\Gamma}\mu_\Gamma\in\mathcal P(s)$ satisfies that $\mu_\Gamma$ has $e_\Gamma$-core $\lambda_\Gamma$ for all $\Gamma$ and this follows by Lemma \ref{fortheblock}.
\end{proof}

Thus according to Theorem \ref{ecuspidalofblock}, \ref{blockcorres} and Proposition \ref{corrofblockspara}, if $\ell\ge3$,
then the set $\{(\mathbf L_{s,\lambda}, \zeta_{s,\lambda})\}$,
where $s$ runs through a complete set of representatives of $\mathbf G^F$-conjugacy classes of
the semisimple $\ell'$-elements of $\mathbf G^F$ and $\lambda$ runs through $\mathcal C(s)$,
is a complete set of representatives of $\mathbf G^F$-conjugacy classes of $e_0$-Jordan-cuspidal pairs of $\mathbf G$.

For an $e_0$-Jordan-cuspidal pair $(\mathbf L, \zeta)$ of $\mathbf G$, let $\mathbf L_0=\mathbf L\cap \mathbf X$.
Now we consider the number of irreducible constituents of $\Res^{\mathbf L^F}_{\mathbf L_0^F}\zeta$.
Note that $\mathbf L^F/\mathbf L_0^F\cong \mathbf G^F/\mathbf X^F$ by Remark \ref{esplitofslsu}.
So $\Irr(\mathbf L^F/\mathbf L_0^F)$ can be identified to $\Irr(\mathbf G^F/\mathbf X^F)$ which is isomorphic to $Z(G)$~(hence to $\mathfrak Z$).
So we may regard $\hat z$ as a character of $\Irr(\mathbf L^F/\mathbf L_0^F)$ for $z\in\mathfrak Z$.

We define the action of $\mathfrak Z$ on $\mathcal C(s)$.
For $\lambda=\prod_{\Gamma}\lambda_\Gamma\in \mathcal C(s)$ and $z\in \mathfrak Z$,
define $z\lambda=\prod_{\Gamma}(z\lambda)_\Gamma$ with $(z\lambda)_{z.\Gamma}=\lambda_\Gamma$.
Then by the definition, for every $z\in \mathfrak Z$,
$\mathbf L_{z,\lambda}=\mathbf L_{zs,z\lambda}$ and $(\mathbf L, \zeta_{zs,z\lambda})$ is also an $e_0$-Jordan cuspidal pair for $\mathbf G$.

\begin{prop}\label{actiononlabel}
With the notation above, $\hat z\zeta_{s,\lambda}=\zeta_{zs,z\lambda}$ for $z\in \mathfrak Z$.
\end{prop}

\begin{proof}
Note that $\zeta_{s,\lambda}=\varepsilon_{\mathbf L}\varepsilon_{\mathbf H}R_{\mathbf H}^{\mathbf L}(\hat s\phi)$.
Then by \cite[Prop. 12.6]{DM91},
$$\hat z\zeta_{s,\lambda}=\varepsilon_{\mathbf L}\varepsilon_{\mathbf H}\hat zR_{\mathbf H}^{\mathbf L}(\hat s\phi)=\varepsilon_{\mathbf L}\varepsilon_{\mathbf H}R_{\mathbf H}^{\mathbf L}(\widehat{zs}\phi),$$
since $\mathbf H=C_{\mathbf L}(zs)$.
Obviously,
$\zeta_{zs,z\lambda}=\varepsilon_{\mathbf L}\varepsilon_{\mathbf H}R_{\mathbf H}^{\mathbf L}(\widehat{zs}\phi)$.
So $\hat z\zeta_{s,\lambda}=\zeta_{zs,z\lambda}$.
\end{proof}

For a positive integer $d$
 and $\Gamma\in\cF$, let $\Gamma_{(d)}$ be a polynomial in $\cF(d)$ such that $\Gamma_{(d)}$ and $\Gamma$ have a common root in $\overline{\F}$~(where $\cF(d)$ is defined as in (\ref{def-cfd})).
 Thus $\Gamma_{(d)}$ has degree $\frac{d_\Gamma}{\mathrm{gcd}(d,d_\Gamma)}$.
Moreover,
if the roots of $\Gamma$ are $\sigma$, $\sigma^{\epsilon q}$, $\ldots$, $\sigma^{(\epsilon q)^{d_\Gamma}}$, then we may take $\Gamma_{(d)}$ to be the polynomial in $\cF(d)$ whose roots are $\sigma$, $\sigma^{(\epsilon q)^d}$, $\ldots$, $\sigma^{(\epsilon q)^{d(\frac{d_\Gamma}{\mathrm{gcd}(d,d_\Gamma)}-1)}}$.

For a semisimple $\ell'$-element of $\mathbf G^F$, we denote by $E(s):=\{\ \Gamma\in \cF\mid m_\Gamma(s)> 0\}$ the set of all elementary divisors of $s$.
When $\ell$ is odd, we let $\Gamma^{(\ell)}:=\Gamma_{(e)}$ and $E_\ell(s):=\{~\Gamma\in E(s)\mid w_\Gamma(s)>0 ~\}$.
When $\ell=2$, we let $\Gamma^{(2)}:=\Gamma$ if $4\mid q-\epsilon$ or $d_\Gamma$ is even, and $\Gamma^{(2)}:=\Gamma_{(2)}$ if $4\mid q+\epsilon$ or $d_\Gamma$ is odd.
Also we define $E_2(s):=E(s)$ if $4\mid q-\epsilon$, and $E_2(s):=\{ \ \Gamma\in E(s) \mid d_\Gamma \ \text{is even or}\ m_\Gamma(s)>1  \ \}$.
Obviously, the degree of $\Gamma^{(\ell)}$ is $e_\Gamma d_\Gamma/e$ in all cases above.

\begin{cor}\label{forconjugateinl}
Let $z\in \mathfrak Z$.
Then $\hat z\zeta_{s,\lambda}=\zeta_{s,\lambda}$ if and only if $(s,\lambda)$ and $(zs,z\lambda)$ are $\mathbf G^F$-conjugate and $z.\Gamma^{(\ell)}=\Gamma^{(\ell)}$ for all $\Gamma\in E_\ell(s)$.
\end{cor}

\begin{proof}
First we assume that $\ell$ is odd.
Abbreviate $\mathbf L_{s,\lambda}=\mathbf L$.
By Proposition \ref{actiononlabel}, $\hat z\zeta_{s,\lambda}=\zeta_{s,\lambda}$ if and only if $(s,\lambda)$ and $(zs,z\lambda)$ are $\mathbf L^F$-conjugate.
Note that $\mathbf L^F=L_0\times L_1$, where
$L_0\cong\GL(\sum_\Gamma |\lambda_\Gamma|d_\Gamma  ,\epsilon q)$ and
$L_1\cong \prod_{\Gamma}\GL(\frac{e_\Gamma d_\Gamma}{e},(\epsilon q)^e)^{w_\Gamma(s)}$.
We write $s=s_0\times s_1$ the corresponding decomposition such that $s_0\in L_0$ and $s_1\in L_1$.
Then $(s,\lambda)$ and $(zs,z\lambda)$ are $\mathbf L^F$-conjugate if and only if $(s_0,\lambda)$ and $(zs_0,z\lambda)$ are $L_0$-conjugate and $s_1$ and $zs_1$ are $L_1$-conjugate.
The semisimple element of $\GL(\frac{e_\Gamma d_\Gamma}{e},(\epsilon q)^e)$ corresponding to the part of $(s_1)_\Gamma$ has a unique elementary divisor which may be assumed to be $\Gamma_{(e)}$.
So $s_1$ and $zs_1$ are $L_1$-conjugate if and only if $z.\Gamma_{(e)}=\Gamma_{(e)}$ for all $\Gamma\in E_\ell(s)$.
Hence $(s,\lambda)$ and $(zs,z\lambda)$ are $\mathbf L^F$-conjugate if and only if
$(zs,z\lambda)$ are $\mathbf G^F$-conjugate and $z.\Gamma_{(e)}=\Gamma_{(e)}$ for all $\Gamma\in E_\ell(s)$.
This proves the assertion for the case that $\ell$ is odd.

For $\ell=2$, the proof is entirely similar to the above.
\end{proof}

Suppose that $s\in\mathbf G^F$ is a semisimple $\ell'$-element.
By Proposition \ref{actiononlabel} and Corollary \ref{forconjugateinl}, if $\hat z\zeta_{s,\lambda}=\zeta_{s,\lambda}$, then $z\in\mathcal O_{\ell'}(\mathfrak Z)$.
Also, $\hat z\zeta_{s,\lambda}\in\mathcal{E}(\bL_{s,\lambda}^F,\ell')$ if and only if $z\in\mathcal O_{\ell'}(\mathfrak Z)$.
So in order to compute $\kappa^{\bL^F_{s,\lambda}}_{(\bL_{s,\lambda}\cap \mathbf X)^F}(\zeta_{s,\lambda})$,
we only need to consider the action of $\mathcal O_{\ell'}(\mathfrak Z)$ on the $\bG^F$-conjugacy classes of  pairs $(s,\lambda)$, where $s$ is a semisimple $\ell'$-element of $\bG^F$ and $\lambda\in\mathcal C(s)$.

\begin{rmk}\label{blocksofslsu}
Analogously with the description for irreducible $\ell$-Brauer characters of $X$ in Remark \ref{brasuerofslsu},
now we give a description for $\ell$-blocks of $X=\SL_n(\epsilon q)$ by summarizing the argument above.
We call a tuple
\begin{equation}\label{admissibletuples}
(([\sigma_1], m_1, \lambda^{(1)}),\dots,([\sigma_a], m_a, \lambda^{(a)}))
\end{equation}
of triples
an \emph{$(n,\ell)$-admissible block tuple},
if
\begin{itemize}
\item for every $1\le i\le a$, $\sigma_i\in \overline{\mathbb F}^\times$ is an $\ell'$-element, and
$m_i$ is a positive integer
such that $\lambda^{(i)}$ is the $e_i$-core of some partition of $m_i$,
where $e_i$ is the multiplicative order of $(\epsilon q)^{\mathrm{deg}(\sigma_i)}$ modulo $\ell$,
\item $[\sigma_i]\ne[\sigma_j]$ if $i\ne j$, and
\item $\sum\limits^{a}_{i=1}m_i\mathrm{deg}(\sigma_i)=n$.
\end{itemize}

An equivalence class of the $(n,\ell)$-admissible block tuple (\ref{admissibletuples}) up to a permutation of triples $$([\sigma_1], m_1, \lambda^{(1)}), \ldots, ([\sigma_a], m_a, \lambda^{(a)})$$ is called an \emph{$(n,\ell)$-admissible block symbol} and is denoted as
\begin{equation}\label{admissiblesyms}
\mathfrak b=[([\sigma_1],  m_1, \lambda^{(1)}),\dots,([\sigma_a], m_a, \lambda^{(a)})].
\end{equation}

Thus by Theorem \ref{blockcorres}, the set of $(n,\ell)$-admissible block symbols is a labeling set for $\ell$-blocks of $G$.
Denote by $B_{\mathfrak b}$ the $\ell$-block of $G$ corresponding to the $(n,\ell)$-admissible block symbol $\mathfrak b$.

The group $\mathcal O_{\ell'}(\mathfrak Z)$ acts on the set of $(n,\ell)$-admissible block symbols via
$$z\cdot [([\sigma_1], m_1, \lambda^{(1)}),\dots,([\sigma_a], m_a, \lambda^{(a)})]=[([z\sigma_1], m_1, \lambda^{(1)}),\dots,([z\sigma_a], m_a, \lambda^{(a)})]$$
for $z\in\mathcal O_{\ell'}(\mathfrak Z)$.
Now we denote by $C_1(\mathfrak b)$ the stabilizer group in $\mathcal O_{\ell'}(\mathfrak Z)$ of the $(n,\ell)$-admissible block symbol $\mathfrak b$.

For a positive integer $d$ and
 $\sigma\in\overline{\mathbb F}^\times$,
if $[\sigma]=\{\sigma, \sigma^{\epsilon q}, \ldots,\sigma^{(\epsilon q)^{\mathrm{deg}(\sigma)}}\}$, then we let
$$[\sigma]_{(d)}:=\{ \sigma, \sigma^{(\epsilon q)^d}, \sigma^{(\epsilon q)^{2d}}, \ldots, \sigma^{(\epsilon q)^{d(\frac{\mathrm{deg}(\sigma)}{\mathrm{gcd}(d,\mathrm{deg}(\sigma))}-1)}}\  \}.$$
We also define $$z[\sigma]_{(d)}:=\{~z\sigma, z\sigma^{(\epsilon q)^{d}}, z\sigma^{(\epsilon q)^{2d}},\ldots, z\sigma^{(\epsilon q)^{d(\frac{\mathrm{deg}(\sigma)}{\mathrm{gcd}(d,\mathrm{deg}(\sigma))}-1)}}~ \}$$ for $z\in\mathfrak Z$.

For an $(n,\ell)$-admissible block symbol $\mathfrak b$ as (\ref{admissiblesyms}), we define the sets $[\sigma_i]_{\mathfrak b}$ for $1\le i\le a$ as follows:
\begin{enumerate}
\item[(i)] When $\ell$ is odd, if $|\lambda^{(i)}|=m_i$, then $[\sigma_i]_{\mathfrak b}$ is empty, and if $|\lambda^{(i)}|<m_i$, then $[\sigma_i]_{\mathfrak b}=[\sigma_i]_{(e)}$.
\item[(ii)] When $\ell=2$,
\begin{itemize}
\item if $4\mid q-\epsilon$ or $\mathrm{deg}(\sigma_i)$ is even, then $[\sigma_i]_{\mathfrak b}=[\sigma_i]$,
\item if $4\mid q+\epsilon$, $\mathrm{deg}(\sigma_i)$ is odd and $m_i=1$, then $[\sigma_i]_{\mathfrak b}$ is empty, and
\item if $4\mid q+\epsilon$, $\mathrm{deg}(\sigma_i)$ is odd and $m_i>1$, then $[\sigma_i]_{\mathfrak b}=[\sigma_i]_{(2)}$.
\end{itemize}
\end{enumerate}

Obviously, if $[\sigma_i]_{\mathfrak b}$ is not empty, then it has cardinality $e_i\mathrm{deg}(\sigma_i)/e$.
If $[\sigma_i]_{\mathfrak b}$ is empty, we define the set $z[\sigma_i]_{\mathfrak b}$ to be the empty set for $z\in\mathfrak Z$.
Now we denote $$C_2(\mathfrak b):=\{~z\in\mathcal O_{\ell'}(\mathfrak Z)\mid z[\sigma_i]_{\mathfrak b}= [\sigma_i]_{\mathfrak b}\ \text{for all}\ 1\le i\le a~ \},$$
and let $\kappa(\mathfrak b):=|C_1(\mathfrak b)\cap C_2(\mathfrak b)|$.

Assume that $\ell$ is odd.
By Lemma \ref{cliffordthm}, Proposition \ref{restrofesplit} and Corollary \ref{forconjugateinl}, $\kappa^G_X(B_{\mathfrak b})=\kappa(\mathfrak b)$
~(\emph{i.e.} the number of $\ell$-blocks of $X$ covered by $B_{\mathfrak b}$ is $\kappa(\mathfrak b)$).
For two $(n,\ell)$-admissible block symbols $\mathfrak b$ and $\mathfrak b'$, if they are in the same $\mathcal O_{\ell'}(\mathfrak Z)$-orbit, then the sets of the $\ell$-blocks of $X$ covered by $B_{\mathfrak b}$ and $B_{\mathfrak b'}$ are the same.

If moreover, we let $(B_{\mathfrak b})_1$, $(B_{\mathfrak b})_2,\dots, (B_{\mathfrak b})_{\kappa(\mathfrak b)}$
the $\ell$-blocks of $X$ covered by $B_{\mathfrak b}$,
then by Remark \ref{esplitofslsu},
the set $\{(B_{\mathfrak b})_j\}$, where $\mathfrak b$ runs through the $\mathcal O_{\ell'}(\mathfrak Z)$-orbit representatives of $(n,\ell)$-admissible block symbols and $j$ runs through the integers between $1$ and  $\kappa(\mathfrak b)$,
is a complete set of the $\ell$-blocks of $X$.	

If $\ell=2$, then $\kappa^G_X(B_{\mathfrak b})\le\kappa(\mathfrak b)$, for any $(n,\ell)$-admissible block symbol $\mathfrak b$.
\end{rmk}

\begin{rmk}\label{brauerchartoblock}
Suppose that $\mathfrak b=[([\sigma_1], m_1, \lambda^{(1)}),\dots,([\sigma_a], m_a, \lambda^{(a)})]$ is an $(n,\ell)$-admissible block symbol.
Then the set of $\ell$-Brauer characters $\{\ \phi_{\mathfrak s}\ \}$, where $$\mathfrak s=[([\sigma_1],\mu^{(1)}),\dots,([\sigma_a],\mu^{(a)})]$$ runs through the $(n,\ell)$-admissible symbols such that
$|\mu_i|=m_i$ and
$\lambda^{(i)}$ is an $e_i$-core of $\mu^{(i)}$ where $e_i$ is the multiplicative order of $(\epsilon q)^{\mathrm{deg}(\sigma_i)}$ modulo $\ell$ for every $1\le i\le a$, is a complete set of irreducible $\ell$-Brauer characters of $B_{\mathfrak b}$.
Let $b$ be an $\ell$-block of $X$ covered by $B_{\mathfrak b}$, then $\Res^G_X\phi_{\mathfrak s}$ has exactly $\kappa(\mathfrak s)/\kappa(\mathfrak b)$ irreducible constituents lying in $b$ when $\ell$ is odd.

Moreover, if we write $\IBr_\ell(B_{\mathfrak b})=\{\ \phi_1,\dots,\phi_k \ \}$,
then by Corollary \ref{corrofbrauseries}, $\IBr_\ell(B_{z \mathfrak b})=\{\ \hat z \phi_1,\dots,\hat z\phi_k \ \}$ for all $z\in\mathcal O_{\ell'}(\mathfrak Z) $.
\end{rmk}

\begin{rmk}
Let $s$ be a semisimple $\ell'$-element of $G$, $\lambda\in\mathcal C(s)$, and $B$ the $\ell$-block of $G$ with label $(s,\lambda)$.
Suppose that $z\in\mathcal O_{\ell'}(\mathfrak Z)$ and $B'$ is the $\ell$-block of $G$ with label $(zs,z\lambda)$.
Then by Remark \ref{decompositionmtx} and \ref{brauerchartoblock}, with a suitable ordering, the decomposition matrices associated with the basic sets $\mathcal E(G,s)\cap \Irr(B)$ and $\mathcal E(G,zs)\cap \Irr(B')$ of $B$ and $B'$, respectively, are the same.
\end{rmk}

Now we consider the unipotent blocks.
The following result follows by Remark \ref{brasuerofslsu}, \ref{blocksofslsu} and \ref{brauerchartoblock} immediately~(also by \cite[Thm. C]{Ge93}).

\begin{lem}\label{restrforuni}
	Assume that $\ell\nmid \mathrm{gcd}(n,q-\epsilon)$.
	\begin{enumerate}
		\item[(i)] The restriction of $\ell$-Brauer characters gives a bijection from the set of irreducible $\ell$-Brauer characters in unipotent $\ell$-blocks of $G$ to the set of irreducible $\ell$-Brauer characters in unipotent $\ell$-blocks of $X$.
		\item[(ii)] Let $b$ be a unipotent $\ell$-block of $X$, then there exists a unique unipotent $\ell$-block $B$ of $G$ which covers $b$.
		Moreover, $\Res^G_X:\IBr_\ell(B)\to\IBr_\ell(b)$ is a bijection.
	\end{enumerate}	
\end{lem}

We can  consider the extendibility of the irreducible $\ell$-Brauer characters in unipotent $\ell$-blocks of $X=\SL_n(\epsilon q)$ now.

\begin{prop}\label{extenofchar}
	Let $\chi\in\Irr(G)$, then $\chi$  extends to $(G\rtimes D)_\chi$.
\end{prop}
\begin{proof}
	First, $(G\rtimes D)_\chi=G\rtimes D_\chi$. If $D_\chi$ is cyclic, then $\chi$ extends to $(G\rtimes D)_\chi$.
	If $D_\chi$ is not cyclic, then $D_\chi=\langle \gamma, F_p^i \rangle$ for some $i\ |\ f$.
	By \cite[Thm. 4.3.1 and Lem. 4.3.2]{Bo99}, there exists an extension $\tilde{\chi}$ of $\chi$ to $G\rtimes \langle  F_p^i \rangle$ such that $\tilde\chi (F_p^i)\ne 0$.
	Since $\gamma$ fixes $\chi$,  $\tilde{\chi}^\gamma$ is also an extension of $\chi$.
	Also we have $\tilde{\chi}^\gamma(F_p^i)=\tilde\chi (F_p^i)\ne 0$ since $\gamma$ and $F_p$ commute.
	By a direct consequence of Gallagher's theorem~(see \cite[Rmk. 9.3(a)]{Sp09}), we have $\tilde{\chi}^\gamma=\tilde\chi$, hence $\tilde{\chi}$ is $\gamma$-invariant.
	So $\tilde{\chi}$ has an extension to $G\rtimes D_\chi$ which is also an extension of $\chi$.
\end{proof}

\begin{cor}\label{extenforuG}
	Let $\phi\in\IBr_\ell(G)$ in a unipotent $\ell$-block of $G$, then $\phi$ extends to $G\rtimes D$.
\end{cor}
\begin{proof}
	It is well-known that every unipotent character of $G$ is $D$-invariant~(see, for example, \cite[Thm. 2.5]{Ma08}).
	By \cite[Thm.~5.1]{GH91}, $\mathcal E(\mathbf G^F, \ell')$ is a basic set of $\IBr_\ell(G)$
	and by \cite{Ge91}, after a suitable arrangement, the decomposition matrix of $G$ with respect to $\mathcal E(\mathbf G^F, \ell')$ is unitriangular.	
	Then the claim follows by Proposition \ref{extenofchar} and Lemma \ref{fromordtomod}.
\end{proof}

Thus by Lemma \ref{restrforuni} and Corollary \ref{extenforuG}, we have:

\begin{cor}\label{extofunpofsl}
	Let $\ell\nmid\mathrm{gcd} (n,q-\epsilon)$, and $\theta\in\IBr_\ell(X)$ in a unipotent $\ell$-block of $X$, then $\theta$ extends to $G\rtimes D$.
\end{cor}

\section{Weights of $\SL_n(\epsilon q)$}\label{chapweightsofslsu}

\subsection{Radical subgroups of $\GL_n(\epsilon q)$}\label{RadicalsubgroupsofGLn}

First, we consider the case that $\ell$ is an odd prime and let $a=v_\ell((\epsilon q)^e-1)$.
We first recall the basic constructions in \cite{AF90} and \cite{An94}.
Let $\alpha,\gamma$ be non-negative integers, $Z_\alpha$ be the cyclic group of order $\ell^{a+\alpha}$ and $E_\gamma$ be an extraspecial $\ell$-group of order $\ell^{2\gamma+1}$.
We may assume the exponent of $E_\gamma$ is $\ell$ by \cite[(4A)]{AF90} and \cite[(1B)]{An94}.
Denote by $Z_\alpha E_\gamma$ the central product of $Z_\alpha$ and $E_\gamma$ over $\Omega_1(Z_\alpha)=Z(E_\gamma)$.
Assume $Z_\alpha E_\gamma = \langle z,x_j,y_j \mid j=1,\ldots,\gamma \rangle$ with $\langle z\rangle=Z_\alpha$, $E_\gamma = \langle x_j,y_j \mid j=1,\ldots,\gamma \rangle$, $o(z)=\ell^{a+\alpha}$, $o(x_j)=o(y_j)=\ell~(1\leqslant j\leqslant\gamma)$,
$[x_i,x_j]=[y_i,y_j]=[x_i,y_j]=1$ if $i\ne j$, and
$[x_j,y_j]=x_jy_jx_j^{-1}y_j^{-1}=z^{\ell^{a+\alpha-1}}$.

The group $Z_\alpha E_\gamma$ can be embedded into $\GL(\ell^\gamma,(\epsilon q)^{e\ell^\alpha})$ uniquely up to conjugacy in the sense that $Z_\alpha$ is identified with $\cO_\ell(Z(\GL(\ell^\gamma,(\epsilon q)^{e\ell^\alpha})))$.
We denote by $R_{\alpha,\gamma}$ the image of $Z_\alpha E_\gamma$ in $\GL(\ell^\gamma,(\epsilon q)^{e\ell^\alpha})$.
Then by \cite[(1C)]{An94}, $R_{\alpha,\gamma}$ is unique up to conjugacy in $\GL(e\ell^{\alpha+\gamma},\epsilon q)$ in the sense that $Z(R_{\alpha,\gamma})$ is primary.

Let $R_{m,\alpha,\gamma}=R_{\alpha,\gamma}\otimes I_{(m)}$.
For each positive integer $c$, let $A_c$ denote the elementary abelian group of order $\ell^c$.
For a sequence of positive integers $\bc=(c_1,\ldots,c_t)$ with $t\geqslant0$,
we denote by $A_\bc=A_{c_1}\wr\cdots\wr A_{c_t}$ and $|\bc|=c_1+\cdots+c_t$.
Then $A_\bc$ can be regarded as an $\ell$-subgroup of the symmetric group $\mathfrak S(\ell^{|\bc|})$.
Groups of the form $R_{m,\alpha,\gamma,\bc}=R_{m,\alpha,\gamma}\wr A_\bc$ are called the basic subgroups.
$R_{m,\alpha,0,\zero}$ is just $R^{m,\alpha}$ in \cite{FS82} which we will write as $R_{m,\alpha}$ here.
By \cite[(4A)]{AF90} and \cite[(2B)]{An94}, any $\ell$-radical subgroup $R$ of $ G$ is conjugate to $R_0 \times R_1 \times \cdots \times R_u$, where $R_0$ is a trivial group and $R_i$ ($i\geqslant1$) is a basic subgroup.

Let $G_{m,\alpha}=\GL(me\ell^\alpha,\epsilon q)$,
$G_{m,\alpha,\gamma}=\GL(me\ell^{\alpha+\gamma},\epsilon q)$,
$C_{m,\alpha}=C_{G_{m,\alpha}}(R_{m,\alpha})$  and
$C_{m,\alpha,\gamma}=C_{G_{m,\alpha,\gamma}}(R_{m,\alpha,\gamma})$,
then $C_{m,\alpha,\gamma}=C_{m,\alpha}\otimes I_\gamma$.
Let $G_{m,\alpha,\gamma,\bc}=\GL(me\ell^{\alpha+\gamma+|\bc|},\epsilon q)$
and $C_{m,\alpha,\gamma,\bc}=C_{G_{m,\alpha,\gamma,\bc}}(R_{m,\alpha,\gamma,\bc})$.
Then $C_{m,\alpha,\gamma,\bc}=C_{m,\alpha}\otimes I_\gamma \otimes I_{\bc}$.
We will also use the notation that $N_{m,\alpha,\gamma}=N_{G_{m,\alpha,\gamma}}(R_{m,\alpha,\gamma})$.

Now we consider the case that $\ell=2$.
Assume that $q$ is odd and let $a$ be the positive integer such that $2^{a+1}=(q^2-1)_2$. We will use the following conventions:
\begin{compactitem}
	\item \textbf{Case 1} \quad ``$4~|~q-\epsilon$'' or ``$4~|~q+\epsilon$ and $\alpha\geqslant1$'',
	\item \textbf{Case 2} \quad ``$4~|~q+\epsilon$ and $\alpha=0$''.
\end{compactitem}
We first recall the basic constructions in \cite{An92} and \cite{An93}.

Let $\alpha,\gamma$ be non-negative integers.
We denote by $Z_\alpha$ the cyclic group of order $2^{a+\alpha}$ in \textbf{Case 1} and of order $2$ in \textbf{Case 2}.
Let $E_\gamma$ be an extraspecial group of order $2^{2\gamma+1}$.
Denote by $Z_\alpha E_\gamma$ the central product of $Z_\alpha$ and $E_\gamma$ over $\Omega_1(Z_\alpha)=Z(E_\gamma)$.
Thus in \textbf{Case 2}, $Z_\alpha E_\gamma=E_\gamma$.
Assume $Z_\alpha E_\gamma = \langle z,x_j,y_j \mid j=1,\ldots,\gamma \rangle$ with $\langle z\rangle=Z_\alpha$, $E_\gamma = \langle x_j,y_j \mid j=1,\ldots,\gamma \rangle$, $[x_j,y_j]=x_jy_jx_j^{-1}y_j^{-1}$ is $z^{2^{a+\alpha-1}}$ in \textbf{Case 1} and $z$ in \textbf{Case 2}.
Assume furthermore that $o(x_j)=o(y_j)=2$ for $j\geqslant2$ and $o(x_1)=o(y_1)=2$ or $2^2$ when $E_\gamma$ is of plus type or minus type respectively, which means that $\langle x_1,y_1\rangle$ is isomorphic to $D_8$ or $Q_8$.
We may assume $E_\gamma$ is of plus type in \textbf{Case 1}.

The group $Z_\alpha E_\gamma$ can be embedded into $\GL(2^\gamma,(\epsilon q)^{2^\alpha})$ uniquely up to conjugacy in the sense that $Z_\alpha$ is identified with $\cO_2(Z(\GL(2^\gamma,(\epsilon q)^{2^\alpha})))$ by \cite[p.509]{An92} and \cite[p.266]{An93}.
We denote by $R_{\alpha,\gamma}$ the image of $Z_\alpha E_\gamma$ in $\GL(2^\gamma,(\epsilon q)^{2^\alpha})$.
Then by \cite[p.510]{An92} and \cite[p.266]{An93}, $R_{\alpha,\gamma}$ is unique up to conjugacy in $\GL(2^{\alpha+\gamma},\epsilon q)$ in the sense that $Z(R_{\alpha,\gamma})$ is primary.
Set $R_{m,\alpha,\gamma}=R_{\alpha,\gamma}\otimes I_{(m)}$.

Now assume $4~|~q+\epsilon$, then $\GL(2,\epsilon q)$ has a Sylow $2$-subgroup isomorphic to the semi-dihedral group $S_{a+2}$ of order $2^{a+2}$, thus $S_{a+2}$ is unique up to conjugacy in $\GL(2,\epsilon q)$.
Denote by $S_{a+2} E_\gamma$ the central product of $S_{a+2}$ and $E_\gamma$ over $Z(S_{a+2})=Z(E_\gamma)$.
We may assume $E_\gamma$ is of plus type by \cite[(1F)]{An92} and \cite[(1I)]{An93}.
Also, $S_{a+2}E_\gamma$ can be embedded into $\GL(2^{\gamma+1},\epsilon q)$ and we denote by $S_{1,\gamma}$ the image of $S_{a+2}E_\gamma$.
By \cite[(1F)]{An92} and \cite[(1I)]{An93}, $S_{1,\gamma}$ is unique up to conjugacy in $\GL(2^{\gamma+1},\epsilon q)$.
Set $S_{m,1,\gamma}=S_{1,\gamma}\otimes I_m$.

For each $\alpha\geqslant0$, $\gamma\geqslant0$, $m\geqslant1$ and $1\leqslant i\leqslant2$, define
$$R^i_{m,\alpha,\gamma}=\left\{ \begin{array}{ll} S_{m,1,\gamma-1} & \textrm{in \textbf{Case 2} and}~\gamma\geqslant1,i=2,\\ R_{m,\alpha,\gamma} & \textrm{otherwise}. \end{array} \right.$$
For each positive integer $c$, let $A_c$ denote the elementary abelian group of order $2^c$.
For a sequence of positive integers $\bc=(c_1,\ldots,c_t)$ with $t\geqslant0$,
we denote by $A_\bc=A_{c_1}\wr\cdots\wr A_{c_t}$ and $|\bc|=c_1+\cdots+c_t$.
Set $R^i_{m,\alpha,\gamma,\bc}=R^i_{m,\alpha,\gamma}\wr A_\bc$.

Groups of the form $R^i_{m,\alpha,\gamma,\bc}$ are called the basic subgroups except in \textbf{Case 2} and $\gamma=0,c_1=1$.
By \cite[(2B)]{An92} and \cite[(2B)]{An93}, any $2$-radical subgroup $R$ of $ G$ is conjugate to $R_1 \times \cdots R_s \times R_{s+1} \times \cdots \times R_u$, where $R_i=\{\pm I_{m_i}\}$ for $1\leqslant i\leqslant s$ and $R_i$ ($i\geqslant s+1$) are basic subgroups.
Moreover, if $4~|~q-\epsilon$, then $s=0$.

When considering further the weights instead of only radical subgroups, we can exclude some basic subgroups which do not afford any weight by the remark on \cite[p.518]{An92} and \cite[p.275]{An93}.
Thus as in \cite{An92} and \cite{An93}, we may assume every component of a $2$-radical subgroup is of the form $D_{m,\alpha,\gamma,\bc}$ defined as follows:
\begin{equation}\label{equation:basic_subgp-2}
D_{m,\alpha,\gamma,\bc}=\left\{ \begin{array}{ll}
R_{m,\alpha,\gamma,\bc} & \textrm{in ``\textbf{Case 1}" or ``\textbf{Case 2} and}~\gamma=0, c_1\neq1",\\
S_{m,1,\gamma-1,\bc} & \textrm{in \textbf{Case 2} and}~\gamma\geqslant1,\\
R_{m,0,1,\bc'} & \textrm{in \textbf{Case 2} and}~\gamma=0, c_1=1,
\end{array} \right.
\end{equation}
where $\bc'=(c_2,\ldots,c_t)$ for $\bc=(c_1,\ldots,c_t)$ and in \textbf{Case 2} and $\gamma=0, c_1=1$, $R_{m,0,1}$ is a quaternion group.
We will use some obvious simplification of the notations, such as $D_{\alpha,\gamma}=D_{0,\alpha,\gamma,\zero}$.
Note also that $D_{m,0,0,\zero}$ in \textbf{Case 2} is just the group $\{\pm I_m\}$.

In order to deal with the two cases that $\ell$ is odd and $\ell=2$ simultaneously, we use the notation $D_{m,\alpha,\gamma,\bc}$ standing for the basic subgroups, so for an odd prime $\ell$, $D_{m,\alpha,\gamma,\bc}=R_{m,\alpha,\gamma,\bc}$, and for $\ell=2$, $D_{m,\alpha,\gamma,\bc}$ is as in (\ref{equation:basic_subgp-2}).

\vspace{2ex}

\begin{lem}\label{detofC}
	 Assume that $\ell\nmid\mathrm{gcd}(n,q-\epsilon)$.
	Let $R$ be an $\ell$-radical subgroup of $G$, then
	$\mathcal D(RC_G(R))=\mathcal D(N_G(R))=\mathfrak Z$~(where $\mathcal D$ is defined as in (\ref{def-mathcalD})).
\end{lem}

\begin{proof}\label{detofrad}
Note that $\mathcal O_\ell(Z(G))\le R$, hence $\mathcal D(R)=\mathcal O_{\ell}(\mathfrak Z)$ since $\ell\nmid\mathrm{gcd}(n,q-\epsilon)$.
So it suffices to show	 that
$\mathcal O_{\ell'}(\mathfrak Z)\le \mathcal D(C_G(R)).$
By the structure of $\ell$-radical subgroups,
	it suffices to show that for a basic subgroup $D_{m,\alpha,\gamma,\mathbf c}$ of $G$, we have
$\mathcal O_{\ell'}(\mathfrak Z)\le \mathcal D(C_{m,\alpha,\gamma,\mathbf c})$.

	By \cite{AF90}, \cite{An92}, \cite{An93} and \cite{An94}, $C_{m,\alpha,\gamma,\mathbf c}\cong C_{m,\alpha}\otimes I_{\gamma+|\mathbf c|}$ where $C_{m,\alpha}\cong \GL(m,q^{e\ell^\alpha})$.
	The elements of $C_{m,\alpha,\gamma,\mathbf c}$ have the form $\mathrm{diag}(g,\dots, g)$ where $g\in C_{m,\alpha}$.	
	Also,  $C_{m,\alpha}$ is the image under the embedding $\GL(m,(\epsilon q)^{e\ell^\alpha})\hookrightarrow \GL(me\ell^\alpha,\epsilon q)$.
	Let $c$ be a generator of the group $\{x\in\mathbb F_{(\epsilon q)^{e\ell^\alpha}}^\times\mid x^{(\epsilon q)^{e\ell^\alpha}-1}=1\}$ and $\Delta\in\cF$ such that $c$ is a root of $\Delta$.
	Then the roots of $\Delta$ are $c,c^{\epsilon q},\dots, c^{(\epsilon q)^{e\ell^\alpha-1}}$ and then $\mathrm{det}((\Delta))=c^{\frac{(\epsilon q)^{e\ell^\alpha}-1}{\epsilon q-1}}$.	
	From this $\mathrm{det}((\Delta))$ is a generator of the group $\mathfrak Z$.
	Thus $\mathcal D(C_{m,\alpha})=\mathfrak Z$.
	So $\mathcal O_{\ell'}(\mathfrak Z)\le \mathcal D(C_{m,\alpha,\gamma,\mathbf c})$ since $C_{m,\alpha,\gamma,\mathbf c}\cong C_{m,\alpha}\otimes I_{\gamma+|\mathbf c|}$.
	This completes the proof.
\end{proof}

\subsection{Radical subgroups of $\SL_n(\epsilon q)$}

Now we consider the $\ell$-radical subgroups of $X$.
Let $\hat X=XZ(G)$.
We will always assume $\ell \nmid \mathrm{gcd}(n,q-\epsilon)$ from now on to the end of this section.

By Lemma \ref{relationofradicalgroups}, the map $\Rad_\ell(G)\to\Rad_\ell(X)$ given by $R\mapsto R\cap X$ is surjective.
In fact, we have:

\begin{prop}\label{relaofrad}
$R\mapsto R\cap X$ gives a bijection from  $\Rad_\ell(G)$ to $\Rad_\ell(X)$ with inverse given by $S\mapsto S\mathcal O_\ell(Z(G))$.
\end{prop}
\begin{proof}
First, we have
$\Rad_\ell(G)=\Rad_\ell(\hat X)$,
since $\ell\nmid |G/\hat X|$.

Since $\hat X/Z(X)\cong X/Z(X)\times Z(G)/Z(X)$ and $Z(X)$ is a central $\ell'$-subgroup of $\hat X$,	
by the same argument as the proof of \cite[Lem. 4.5]{FLL17}~(use \cite[Lem. 4.3 and 4.4]{FLL17}), there is a bijection $\Rad_\ell(\hat X)\to \Rad_\ell(X)$ given by $R\mapsto R\cap X$ with inverse given by $S\mapsto S\mathcal O_\ell(Z(G))$.
\end{proof}

\begin{lem}\label{gexingz}
Let $R$ be an $\ell$-radical subgroup of $G$ and $S=R\cap X$. Then
\begin{enumerate}
\item[(i)] $C_X(S)=C_G(R)\cap X$, $SC_X(S)=RC_G(R)\cap X$, $N_X(S)=N_G(R)\cap X$,
\item[(ii)] $RC_G(R)/SC_X(S)\cong N_G(R)/N_X(S)\cong G/X$.
\end{enumerate}
\end{lem}

\begin{proof}
By Proposition \ref{relaofrad}, $R=S\mathcal O_\ell (Z(G))$, so we have $C_X(S)=C_G(R)\cap X$, $N_X(S)=N_G(R)\cap X$.
Also $RC_G(R)\cap X=SC_G(R)\cap X=S(C_G(R)\cap X)=SC_X(S)$ and then we obtain (i).
By Lemma \ref{detofrad}, we have $G=XRC_G(R)$ and then $G=XN_G(R)$.
Thus (ii) follows.
\end{proof}

Let $R$ be an $\ell$-radical subgroup of $G$,
by Lemma \ref{detofrad},
$G=XN_G(R)$.
So if two $\ell$-radical subgroups of $G$ are $G$-conjugate, then they are $X$-conjugate.
Thus by Proposition \ref{relaofrad} and Lemma \ref{gexingz}, we have:

\begin{cor}\label{relaofradicalconj}
	$R\mapsto R\cap X$ gives a bijection from  $\Rad_\ell(G)/\thicksim_G$ to $\Rad_\ell(X)/\thicksim_X$.
\end{cor}

\subsection{Weights of $\SL_n(\epsilon q)$}

Now we consider the $\ell$-weights of $X$ with $\ell\nmid\mathrm{gcd}(n,q-\epsilon)$.
By Lemma \ref{corrofweights} and Proposition \ref{relaofrad} and Lemma \ref{gexingz}, we have:

\begin{prop}\label{weightofslsubyres}
Assume that $\ell\nmid\mathrm{gcd}(n,q-\epsilon)$.
Let $(R,\varphi)$ be an $\ell$-weight of $G$ and $S=R\cap X$, then $(S,\psi)$ is an $\ell$-weight of $X$ for every $\psi\in\Irr(N_X(S)\ | \ \varphi)$.

Conversely, let $(S,\psi)$ be an $\ell$-weight of $X$ and $R=S\mathcal O_\ell (Z(G))$, then there exists $\varphi\in\Irr(N_G(R)\ | \ \psi)$ such that $(R,\varphi)$ is an $\ell$-weight of $G$.
\end{prop}

\begin{rmk}\label{restritowei}
Let $\mathcal W_\ell(G)$ be a complete set of representatives of all $G$-conjugacy classes of $\ell$-weights of $G$. We may assume that for $(R_1,\varphi_1),(R_2,\varphi_2)\in \mathcal W_\ell(G)$, $R_1$ and $R_2$ are $G$-conjugate if and only if $R_1=R_2$.
	
	Now define a equivalence relation on $\mathcal W_\ell(G)$ such that for $(R_1,\varphi_1),(R_2,\varphi_2)\in \mathcal W_\ell(G)$, $(R_1,\varphi_1)\sim(R_2,\varphi_2)$ if and only if $R_1=R_2$ and $\varphi_1=\varphi_2 \eta$ for some $\eta\in\Irr(N_G(R_1)/N_X(R_1))$.
	Then by  Lemma \ref{cliffordthm}, Corollary \ref{relaofradicalconj} and Proposition \ref{weightofslsubyres},
	the set $\{(R\cap X,\psi)\}$, where $(R,\varphi)$ runs through a complete set of representatives of the equivalence classes of $\mathcal W_\ell(G)/\sim$ and $\psi$ runs through $\Irr(N_X(R)\ |\ \varphi)$, is a complete set of representatives of all $X$-conjugacy classes of $\ell$-weights of $X$.
\end{rmk}

\begin{rmk}\label{forweightbelow}
	Let $(R,\varphi)$ be an $\ell$-weight of $G$, $(S,\psi)$ an $\ell$-weight of $X$ such that $S=R\cap X$ and $\varphi\in\Irr(N_G(R)\ | \ \psi)$.
	Let $b=\bl_\ell(\varphi)$, $b_0=\bl_\ell(\psi)$ and $B=b ^G$ and $B_0=b_0^X$.
By Lemma \ref{covers}, if $b$ covers $b_0$, then $B$ covers $B_0$.
	
	Let $B_0$ be an $\ell$-block of $X$.
	Denote by $\mathcal B_0$ the union of the $\ell$-blocks of $X$ which are $G$-conjugate to $B_0$ and ${\mathcal B}$  the union of the $\ell$-blocks of $G$ which cover $B_0$.
	Then
	\begin{itemize}
		\item if $(R,\varphi)$ is an $\ell$-weight of $G$ belonging to ${\mathcal B}$ and $S=R\cap X$, then for every $\psi\in\Irr(N_X(S)\ | \ \varphi)$, $(S,\psi)$ is an $\ell$-weight of $X$ belonging to $\mathcal B_0$, and
		\item if $(S,\psi)$ is an $\ell$-weight of $X$ belonging to $\mathcal B_0$ and $R=S\mathcal O_\ell (Z(G))$, then there exists $\varphi\in\Irr(N_G(R)\ | \ \psi)$ such that $(R,\varphi)$ is an $\ell$-weight of $G$ belonging to ${\mathcal B}$.
	\end{itemize}
\end{rmk}

Let $(R,\varphi)$ be an $\ell$-weight of $G$.
For some $\eta\in\Irr(N_G(R)/N_X(R))$, if $(R,\eta\varphi)$ is also an $\ell$-weight of $G$, then $\mathcal O_\ell(Z(G)) \subseteq\ker\eta$ since $\mathcal O_\ell(Z(G)) \subseteq R$ by Proposition \ref{relaofrad}.
Hence $\eta\in\mathcal O_{\ell'}(\Irr(N_G(R)/N_X(R)))$.

By Lemma \ref{gexingz}, $N_G(R)/N_X(R)\cong G/X\cong \mathfrak Z$~(where $\mathfrak Z$ is defined as in ({\ref{def-frakZ})).
Now we identify $\Irr(N_G(R)/N_X(R))$ with $\Irr(G/X)$.
So in order to compute $\kappa^{N_G(R)}_{N_X(R)}(\varphi)$, it suffices to consider when $\Res^{G}_{N_G(R)}(\hat z)\cdot \varphi=\varphi$ for $z\in\mathcal O_{\ell'}(\mathfrak Z)$.
We often abbreviate $\hat z$ for $\Res^{G}_{N_G(R)}(\hat z)$.

\vspace{2ex}

Now we recall the description of $\ell$-weights of $G$ in \cite{AF90}, \cite{An92}, \cite{An93} and \cite{An94} and give some more notations and conventions.

We denote by $\cF'$ the subset of $\cF$ consisting of polynomials whose roots are of $\ell'$-orders.
By \cite[(3.2)]{Br86}, given any $\Gamma\in\cF'$, there is a unique $\ell$-block $B_\Gamma$ of $G_\Gamma=\GL(m_\Gamma e\ell^{\alpha_\Gamma},\epsilon q)$ with $D_\Gamma=D_{m_\Gamma,\alpha_\Gamma}$ as a defect group.
This $\ell$-block has the label $(e_\Gamma(\Gamma),-)$.
Here $m_\Gamma,\alpha_\Gamma$ are non-negative integers determined by $m_\Gamma e\ell^{\alpha_\Gamma}=e_\Gamma d_\Gamma$ and $(m_\Gamma,\ell)=1$.
Also, note that there is no direct connection between $m_\Gamma$ and $m_\Gamma(s)$.
These results have been proved for odd primes in \cite[(5A)]{FS82}  and for $\ell=2$ on \cite[p.520]{An92} and \cite[p.276]{An93} using the results from \cite{Br86}.
Let $C_\Gamma=C_{G_\Gamma}(D_\Gamma)$ and $N_\Gamma=N_{G_\Gamma}(D_\Gamma)$.
Then $C_\Gamma\cong \GL(m_\Gamma,(\epsilon q)^{d_\Gamma})$.
The polynomial $\Gamma$ also determines a unique $N_\Gamma$-conjugacy classes of pairs $(\fb_\Gamma,\theta_\Gamma)$ where $\fb_\Gamma$ is a root $\ell$-block of $C_\Gamma D_\Gamma=C_\Gamma$ with defect group $D_\Gamma$ and $\theta_\Gamma$ is the canonical character of $\fb_\Gamma$.
The subpair $(D_\Gamma,\fb_\Gamma)$ has the label $(D_\Gamma,s_\Gamma,-)$ as in \cite[(3.2)]{Br86}.
Since $d_{\Gamma}=d_{z.\Gamma}$, $\alpha_{\Gamma}=\alpha_{z.\Gamma}$ and $m_{\Gamma}=m_{z.\Gamma}$, we may assume that $D_\Gamma=D_{z.\Gamma}$, $C_\Gamma=C_{z.\Gamma}$, and $N_\Gamma=N_{z.\Gamma}$.

Let $\Gamma\in\cF'$ and keep the notation of the previous sections.
Let $D_{\Gamma,\gamma,\bc}=D_{m_\Gamma,\alpha_\Gamma,\gamma,\bc}$ be a basic subgroup and let $G_{\Gamma,\gamma,\bc},C_{\Gamma,\gamma,\bc},N_{\Gamma,\gamma,\bc}$ be defined similarly.
Then $C_{\Gamma,\gamma,\bc}=C_\Gamma\otimes I_\gamma \otimes I_{\bc}$.
Let $\theta_{\Gamma,\gamma,\bc}=\theta_\Gamma\otimes I_\gamma \otimes I_{\bc}$, then $\theta_{\Gamma,\gamma,\bc}$ can be viewed as the canonical character of $C_{\Gamma,\gamma,\bc}D_{\Gamma,\gamma,\bc}$ with $D_{\Gamma,\gamma,\bc}$ in the kernel and all canonical characters are of this form.
Note that the equations \cite[(3.2)]{An92} and \cite[(3.1)]{An93} can be written also uniformly in this form (see the remarks before \cite[Prop. 4.2 and 4.3]{LZ18}).
Let $\cR_{\Gamma,\delta}$ be the set of all the basic subgroups of the form $D_{\Gamma,\gamma,\bc}$ with $\gamma+|\bc|=\delta$ and denote $I_\delta=I_\gamma\otimes I_{\bc}$.
Label the basic subgroups in $\cR_{\Gamma,\delta}$ as $D_{\Gamma,\delta,1}$, $D_{\Gamma,\delta,2}$, $\ldots$ and denote the canonical character associated to $D_{\Gamma,\delta,i}$ by $\theta_{\Gamma,\delta,i}$.
It is possible that there exists $\Gamma'\in\cF'$ such that $m_{\Gamma'}=m_\Gamma=:m$ and $\alpha_{\Gamma'}=\alpha_\Gamma=:\alpha$.
In this case, $\cR_{\Gamma,\delta}=\cR_{\Gamma',\delta}$ and naturally we may choose the labeling of $\cR_{\Gamma,\delta}$ and $\cR_{\Gamma',\delta}$ such that $D_{\Gamma,\delta,i}=D_{\Gamma',\delta,i}$ for $i=1,2,\ldots$.
We will denote $D_{m,\alpha,\gamma,\bc}$ as $D_{\Gamma,\delta,i}$ or $D_{\Gamma',\delta,i}$ depending on whether the related canonical character of $C_{m,\alpha}D_{m,\alpha}=C_{m,\alpha}$ considered is $\theta_\Gamma$ or $\theta_{\Gamma'}$.
Set $G_{\Gamma,\delta,i}=\GL(m_\Gamma e \ell^{\alpha+\delta},\epsilon q)$, and denote by $N_{\Gamma,\delta,i}$ and $C_{\Gamma,\delta,i}$ the normalizer and centralizer of $D_{\Gamma,\delta,i}$ in $G_{\Gamma,\delta,i}$ respectively.

For $z\in\mathcal O_{\ell'}(\mathfrak Z)$, $\hat z$ is a linear character of $G_{\Gamma,\delta,i}$.
By the proof of Lemma \ref{detofC},
$\mathcal O_{\ell'}(\mathcal D(G_{\Gamma,\delta,i}))=\mathcal O_{\ell'}(\mathcal D(C_{\Gamma,\delta,i}))$,
so $\hat z$ may be regarded as a character of $C_{\Gamma,\delta,i}$~(by restriction).
Here we need some precise information on $\hat z$.

\begin{rmk}\label{descofhat}
Now we recall the description of the map $\ \hat{}\ $ given in \cite{Br86}.
As pointed in  \cite[note${}^2$ (p.186)]{Br86}, the isomorphism in Equation (\ref{isomorphsimhat}) is not uniquely determined.
Also, the author introduces a set $\mathcal S(G)$ to replace the set of semisimple elements of $G$ in \cite{Br86}.

First, denote by $k$ a subfield of $\overline{\mathbb{Q}}_\ell$ of finite degree over $\mathbb Q_\ell$. Also, assume that $k$ is big enough for all finite groups considered.
Suppose that we have chosen an algebraic closure $\overline{\F}$ of $\F$, an isomorphism $\iota: \mu(\overline{\mathbb Q}_\ell)\to \mathbb Q/\Z$, and an isomorphism $\iota': \overline{\F}^\times\to (\mathbb Q/\Z)_{p'}$.

Let $s$ be a semisimple element of $G$, then $L=C_G(s)=\prod_\Gamma L_\Gamma$ with $L_\Gamma\cong \GL(m_\Gamma(s),(\epsilon q)^{d_\Gamma})$.
If $\F_\Gamma$ denotes the field generated by $Z(L_\Gamma)$ in $\mathrm{End}_{\F}(\F^n)$, the group $Z(L_\Gamma)$ is equal to the subgroup of order $|(\epsilon q)^{d_\Gamma}-1|$ of $\F_\Gamma^\times$.
Every family $\sigma$ of embeddings $\sigma_\Gamma:\F_\Gamma\to \overline{\F}$ over $\F$ is associated to a character $\zeta_\sigma(s)$ of $Z(L)$ with values in $k$ in the following way.
Let $g_\Gamma$ be the particular generator of $Z(L_\Gamma)$ defined by the corresponding embedding of $\F_\Gamma^\times$ in $\mathbb Q/\Z$.
The character $\zeta_\sigma(g)$ is defined by the equation $\iota(\zeta_\sigma(s)(g_i))=\iota'(\sigma_i(s_\Gamma))$.

We denote by $\mathcal S(G)$ the set of pairs $(L,\zeta)$ such that there exists semisimple a element $s$ of $G$ and an
embedding $\F\subseteq \overline{\F}$, $\iota$, $\iota'$, $\sigma=$ such that $L=C_G(s)$ and
 $\zeta=\zeta_\sigma(s)$.
Then by \cite[(4.4)]{Br86}, the $G$-conjugacy classes of $\mathcal S(G)$ are in bijection with the set of $G$-conjugacy classes of semisimple elements of $G$.

If $s=(L,\zeta)\in\mathcal S(G)$, we denote by $\hat s$
the linear character of $L=C_G(s)$ with values in $k$ obtained by composing $\zeta$ with the ~(surjective)~ morphism $\mathrm{det}_L: L\to Z(L)$~(defined in \cite[p.171]{Br86}. Indeed, If $h\in L$, we write $h=\prod_\Gamma h_\Gamma$ corresponding to the decomposition $L=\prod_\Gamma L_\Gamma$.
Also, we identity $Z(L_\Gamma)$ with $\F_{(\epsilon q)^{d_\Gamma}}$.
Then $\mathrm{det}_L(h)=\prod_\Gamma \mathrm{det}_L(h_\Gamma)$, where $\mathrm{det}_L(h_\Gamma)$ is the determinant of the matrix corresponding to $h_\Gamma$ in $\GL(m_\Gamma(s),(\epsilon q)^{d_\Gamma})$ ).

Let $s=(C_G(s),\zeta)\in\mathcal S(G)$, and $X$ an $\ell'$-subgroup of $C_G(s)$.
We set $C=C_G(X)$, and we define an element $s_X=(C_C(s),\zeta_X)$  in the following way:
we may suppose that $C_C(s)=\prod_\Gamma L_\Gamma$, where $L_\Gamma\cong \GL(m_\Gamma(s),(\epsilon q)^{d_\Gamma})$.
Then $Z(C_{L_\Gamma}(s))$ isomorphic to a product of $\GL(1,\F_{\Gamma,i})$ where $\F_{\Gamma,i}$ is a certain extension of $\F_{(\epsilon q)^{d_\Gamma}}$.
For any element $z$ of one such factor, we set then $\zeta_X(z)=\zeta(N_{\F_{\Gamma,i}/\F_{(\epsilon q)^{d_\Gamma}}}(z))$.
The surjectivity of the norm in the finite extensions of finite fields allows then to establish that $(C_C(s),\zeta_X)$ belongs to $S(C)$.
Noting that if $X$ is abelian, the linear character $\hat s_X$ is simply the restriction to $C_C(s)$ of the linear character $\hat s$ to $C_G(s)$.
Also, the map form $\mathcal S(G)$ into $\mathcal S(C)$ which associates $s_X$ to $s$ is surjective.
We often omit the index $X$ in $s_X$.
\end{rmk}

\begin{rmk}\label{forthecompz}
Abbreviate $D=D_{\Gamma,\delta,i}$.	
Now we consider the relationship between $\hat z\in\Irr(G_{\Gamma,\delta,i})$ and $\hat z_D\in\Irr(C_{\Gamma,\delta,i})$ for $z\in\mathcal O_{\ell'}(\mathfrak Z)$.
Choose a particular generator $\eta$ of $\F^\times$.
For $g\in G_{\Gamma,\delta,i}$, if $\mathrm{det}_{G_{\Gamma,\delta,i}}(g)=\eta^k$ for some $k\in \Z$, then $\hat{z}(g)=\iota^{-1}\circ\iota'(z)^k$.
Now we choose an isomorphism $\tau: C_{\Gamma}\to \GL(m_\Gamma,(\epsilon q)^{e\ell^{\alpha_\Gamma}})$.
Let $\eta_d\in \F_{(\epsilon q)^{e\ell^{\alpha_\Gamma}}}^\times$ such that $N_{\F_{(\epsilon q)^{e\ell^{\alpha_\Gamma}}}/\F}(\eta_d)=\eta$.
Let $c\in C_{\Gamma,\delta,i}$ with $c=c_0\otimes I_\delta$ and $c_0\in C_{\Gamma}$.
Suppose that $\mathrm{det}(\tau(c_0))=\eta_d^j$ for some $j\in \Z$, then
$\hat{z}_R(c)=\iota^{-1}\circ\iota'(z)^j$.
Also, $\mathrm{det}_{G_{\Gamma,\delta,i}}(\tau(c))=\eta^{j\ell^\delta}$,
so $\Res^{G_{\Gamma,\delta,i}}_{C_{\Gamma,\delta,i}}(\hat z)(c)=\iota^{-1}\circ\iota'(z)^{j\ell^\delta}$.
So $\Res^{G_{\Gamma,\delta,i}}_{C_{\Gamma,\delta,i}}(\hat z)=\hat z_D^{\ell^\delta}$.

Let $s=s_\Gamma=e_\Gamma(\Gamma)\otimes I_\delta$.
Then $C_{G_{\Gamma,\delta,i}}(s)\cong \GL(e_\Gamma\ell^\delta,(\epsilon q)^{d_\Gamma})$.
Let $\F_\Gamma$ be the field generated by $Z(C_{G_{\Gamma,\delta,i}}(s))$ in $\mathrm{End}_{\F}(\F^{e_\Gamma d_\Gamma})$ and $\sigma:\F_\Gamma \to \overline{\F}$ an embedding of fields.
Let $\xi$ be the particular generator of $\F_{(\epsilon q)^{d_\Gamma}}^\times$ and hence it can be regarded as a generator of $Z(C_{G_{\Gamma,\delta,i}}(s))$.
For $g\in C_{G_{\Gamma,\delta,i}}(s)$, if the determinant of the matrix corresponding to $g$ in $\GL(e_\Gamma\ell^\delta,(\epsilon q)^{d_\Gamma})$ is $\xi^{k'}$, then $\hat s(g)=\iota^{-1}\circ \iota'(\sigma(s))^{k'}$.
It is easy to check that if $N_{\F_{(\epsilon q)^{d_\Gamma}}/\F}(\xi)=\eta$,
then $\Res^{G_{\Gamma,\delta,i}}_{C_{G_{\Gamma,\delta,i}}(s)}(\hat z)\cdot \hat s=\widehat{zs}$.
Now $C_{C_{\Gamma,\delta,i}}(s)\cong \GL(1, \F_{(\epsilon q)^{m_\Gamma e\ell^{\alpha_\Gamma}}})$ is a Coxeter torus of ${C_{\Gamma,\delta,i}}$.
Let $\tau':C_{C_{\Gamma,\delta,i}}(s) \to\F_{(\epsilon q)^{m_\Gamma e\ell^{\alpha_\Gamma}}}^\times$ be an embedding morphism.
Notice that $m_\Gamma e\ell^{\alpha_\Gamma}=e_\Gamma d_\Gamma$, so $\F_{(\epsilon q)^{m_\Gamma e\ell^{\alpha_\Gamma}}}$ is also an extension of $\F_{(\epsilon q)^{d_\Gamma}}$.
Let $c\in C_{C_{\Gamma,\delta,i}}(s)$, then there exists a positive integer $j'$, such that
$N_{\F_{(\epsilon q)^{m_\Gamma e\ell^{\alpha_\Gamma}}}/\F_{(\epsilon q)^{d_\Gamma}}}(\tau'(c))=\xi^{j'}$.
So $\hat s_D(c)=\iota^{-1}\circ \iota'(\sigma(s))^{j'}$.
Also, if $N_{\F_{(\epsilon q)^{d_\Gamma}}/\F}(\xi)=\eta$, then $\Res^{C_{\Gamma,\delta,i}}_{C_{C_{\Gamma,\delta,i}}(s)}(\hat z_D)\cdot \hat s_D=\widehat{zs}_D$.

Now $z$ is an $\ell'$-element, so by the argument above, we can choose suitable $\tau$, $\sigma$ and $\tau'$ such that
$\Res^{G_{\Gamma,\delta,i}}_{C_{C_{\Gamma,\delta,i}}(s)}(\hat z)\cdot \hat s_D=\widehat{zs}_{D}$.
\end{rmk}

\begin{lem}\label{forthethetagamma}
$\Res^{G_{\Gamma,\delta,i}}_{C_{\Gamma,\delta,i}}(\hat z)\cdot\theta_{\Gamma,\delta,i}=\theta_{z.\Gamma,\delta,i}$ for $z\in\mathcal O_{\ell'}(\mathfrak Z)$.
\end{lem}

\begin{proof}
By \cite[Prop. 4.16]{Br86},
$\theta_{\Gamma,\delta,i}=\pm R^{C_{\Gamma,\delta,i}}_{C_{C_{\Gamma,\delta,i}}(s)}(\hat{s})$,
where $s$ is a semisimple $\ell'$-element of $C_{\Gamma,\delta,i}$ which has only one elementary divisor $\Gamma$ with multiplicity $e_\Gamma\ell^\delta$~(as in Remark \ref{forthecompz}).
Note that $C_{C_{\Gamma,\delta,i}}(s)=C_{C_{\Gamma,\delta,i}}(zs)=C_{C_{z.\Gamma,\delta,i}}(zs)$.
Then $\Res^{G_{\Gamma,\delta,i}}_{C_{\Gamma,\delta,i}}(\hat z)\cdot\theta_{\Gamma,\delta,i}
=\pm \Res^{G_{\Gamma,\delta,i}}_{C_{\Gamma,\delta,i}}(\hat z)\cdot R^{C_{\Gamma,\delta,i}}_{C_{C_{\Gamma,\delta,i}}(s)}(\hat{s})
=\pm R^{C_{\Gamma,\delta,i}}_{C_{C_{\Gamma,\delta,i}}(s)}(\Res^{G_{\Gamma,\delta,i}}_{C_{C_{\Gamma,\delta,i}}(s)}(\hat z)\cdot \hat s)$
by \cite[Prop. 12.6]{DM91}.
By Remark \ref{forthecompz}, we may assume that $\Res^{G_{\Gamma,\delta,i}}_{C_{C_{\Gamma,\delta,i}}(s)}(\hat z)\cdot \hat s=\widehat{zs}$.
Notice that $zs$ is a semisimple $\ell'$-element of $C_{\Gamma,\delta,i}=C_{z.\Gamma,\delta,i}$ which has only one elementary divisor $z.\Gamma$ with multiplicity $e_\Gamma\ell^\delta$.
This completes the proof.
\end{proof}

Let $\sC_{\Gamma,\delta,i}$ be the set of characters of $N_{\Gamma,\delta,i}(\theta_{\Gamma,\delta,i})$ lying over $\theta_{\Gamma,\delta,i}$ and of defect zero as characters of $N_{\Gamma,\delta,i}(\theta_{\Gamma,\delta,i})/D_{\Gamma,\delta,i}$ and $\sC_{\Gamma,\delta}=\bigcup\limits_{i} \sC_{\Gamma,\delta,i}$.
By Clifford theory, this set is in bijection with the set of characters of $N_{\Gamma,\delta,i}$ lying over $\theta_{\Gamma,\delta,i}$ and of defect zero as characters of $N_{\Gamma,\delta,i}/D_{\Gamma,\delta,i}$ for all $i$.
We assume $\sC_{\Gamma,\delta}=\{\psi_{\Gamma,\delta,i,j}\}$ with $\psi_{\Gamma,\delta,i,j}$ a character of $N_{\Gamma,\delta,i}(\theta_{\Gamma,\delta,i})$.
Note that for $\ell=2$, $j$ has only one choice.
Also, we may assume $D_{\Gamma,\delta,i}=D_{{z.\Gamma},\delta,i}$, $N_{\Gamma,\delta,i}=N_{{z.\Gamma},\delta,i}$, and
$C_{\Gamma,\delta,i}=C_{{z.\Gamma},\delta,i}$.
We choose the labeling of $\sC_{\Gamma,\delta}$ and $\sC_{z.\Gamma,\delta}$ such that
\begin{equation}\label{weight:convention-2}
\Res^{G_{\Gamma,\delta,i}}_{N_{\Gamma,\delta,i}}(\hat z)\cdot\psi_{\Gamma,\delta,i,j}=\psi_{z.\Gamma,\delta,i,j}.
\end{equation}

\begin{rmk}
We can make (\ref{weight:convention-2}) because if for some $z\in\mathcal O_{\ell'}(\mathfrak Z)$, $\Res^{G_{\Gamma,\delta,i}}_{C_{\Gamma,\delta,i}}(\hat z)\cdot\theta_{\Gamma,\delta,i}=\theta_{z.\Gamma,\delta,i}$,
then $\Res^{G_{\Gamma,\delta,i}}_{N_{\Gamma,\delta,i}}(\hat z)$ fixes every element of $\sC_{\Gamma,\delta,i}$.
In fact, if $\ell=2$, then $\sC_{\Gamma,\delta,i}$ has only one element by \cite{An92} and \cite{An93}.
If $\ell$ is odd, and we assume that $D_{\Gamma,\delta,i}=R_{m_\Gamma,\alpha_\Gamma,\gamma,\bc}$, then by \cite[p.14]{AF90} and \cite[p.10]{An94},
$N_{\Gamma,\delta,i}/D_{\Gamma,\delta,i}\cong N_{m_\Gamma,\alpha_\Gamma,\gamma}/R_{m_\Gamma,\alpha_\Gamma,\gamma}\times Y_{\bc}/A_{\bc}$, for some subgroups $Y_{\bc}$ and $A_{\bc}$.
Also, all elements of $Y_{\bc}$ and $A_{\bc}$ are permutation matrices and then have determinant $1$.
So we may assume that $|\bc|=0$.
By the construction of $(N_{m_\Gamma,\alpha_\Gamma,\gamma})_{\theta_\Gamma\otimes I_\gamma}$ in \cite{AF90} and \cite{An94},
we may assume that $\gamma=0$ and then $D_{\Gamma,\delta,i}=R_\Gamma$ and $N_{\Gamma,\delta,i}=N_\Gamma$.
By \cite[\S 4]{LZ18}, up to conjugation, $N_\Gamma=C_\Gamma\rtimes\langle P \rangle$, where $P$ is a permutation matrix.
Thus $\Res^{G_{\Gamma,\delta,i}}_{N_{\Gamma,\delta,i}}(\hat z)$ fixes every element of $\sC_{\Gamma,\delta,i}$.
\end{rmk}

We use the notation from \cite[\S 5]{LZ18} now.
Define $i\cW_\ell(G)$ to be the $G$-conjugacy classes of the set
$$\left\{~(s,\lambda,K)~\middle|~
\begin{array}{c}
s~\textrm{is a semisimple $\ell'$-element of}~ G,\\
\lambda=\prod_\Gamma\lambda_\Gamma,~\lambda_\Gamma~\textrm{is the $e_\Gamma$-core of a partition of}~m_\Gamma(s),\\
K=K_\Gamma,~K_\Gamma:\bigcup_\delta\sC_{\Gamma,\delta}\to\{~\ell\textrm{-cores}~\}~\textrm{s.t.}~\\
\sum_{\delta,i,j}\ell^\delta |K_\Gamma(\psi_{\Gamma,\delta,i,j})|=w_\Gamma,
m_\Gamma(s)=|\lambda_\Gamma|+e_\Gamma w_\Gamma.
\end{array}~\right\}.$$
Note that for $\ell=2$, the triple becomes $(s,-,K)$.

A bijection between $\cW_\ell(G)$ and $i\cW_\ell(G)$ has been constructed implicitly in \cite{AF90}, \cite{An92},  \cite{An93} and  \cite{An94} and can be described as follows.
Let $(R,\varphi)$ be an $\ell$-weight of $G$.
Set $C=C_G(R)$ and $N=N_G(R)$.
Then there exists an $\ell$-block $b$ of $CR$ with $R$ a defect group such that $\varphi=\Ind_{N(\theta)}^{N}\psi$ where $\theta$ is the canonical character of $b$ and $\psi$ is a character of $N(\theta)$ lying over $\theta$ and of $\ell$-defect zero as a character of $N(\theta)/R$.
Assume $R=D_0D_+$ with $D_0$ an identity group of degree $n_0$ and $D_+$ a product of basic subgroups.
Note that for $\ell=2$, $R=D_+$.
Then $C,N,\varphi,\theta,\psi,N(\theta)$ can be decomposed accordingly.

First, we have $C_0=N_0=\GL(n_0,\epsilon q)$ and $\varphi_0=\psi_0=\theta_0$ a character of $\GL(n_0,\epsilon q)$ of $\ell$-defect zero.
So it is of the form $\chi_{s_0,\lambda}$ where $s_0$ is a  semisimple $\ell'$-element of $\GL(n_0,\epsilon q)$ and $\lambda=\prod_\Gamma \lambda_\Gamma$ with $\lambda_\Gamma$ a partition of $m_{s_0,\Gamma}$ without $e_\Gamma$-hook which affords the second component of the triple $(s,\lambda,K)$.

Secondly, assume we have the following decomposition $\theta_+=\prod\limits_{\Gamma,\delta,i}\theta_{\Gamma,\delta,i}^{t_{\Gamma,\delta,i}}$,
$D_+=\prod\limits_{\Gamma,\delta,i}D_{\Gamma,\delta,i}^{t_{\Gamma,\delta,i}}.$
Now $\theta_\Gamma$ determines a  semisimple $\ell'$-element with canonical form $e_\Gamma(\Gamma)$ in $G_\Gamma$.
Thus $s=s_0\prod_{\Gamma,\delta,i}(e_\Gamma(\Gamma)\otimes I_\delta)^{t_{\Gamma,\delta,i}}$ is the first component of the triple $(s,\lambda,K)$.
We can view $b$ as an $\ell$-block of $C_G(R)$, then the Brauer pair $(R,b)$ has a label $(R,s,\lambda)$ as in \cite[(3.2)]{Br86}.
Thus $(R,\varphi)$ belongs to an $\ell$-block $B$ of $G$ with label $(s,\lambda)$.
In particular, $\lambda_\Gamma$ is the $e_\Gamma$-core of a partition of $m_{\Gamma}(s)$.

Finally, we have
$N_+(\theta_+)=\prod\limits_{\Gamma,\delta,i}
N_{\Gamma,\delta,i}(\theta_{\Gamma,\delta,i})\wr\fS(t_{\Gamma,\delta,i})$,
$\psi_+=\prod\limits_{\Gamma,\delta,i} \psi_{\Gamma,\delta,i}$
with $\psi_{\Gamma,\delta,i}$ a character of $N_{\Gamma,\delta,i}(\theta_{\Gamma,\delta,i})\wr\fS(t_{\Gamma,\delta,i})$ covering $\theta_{\Gamma,\delta,i}^{t_{\Gamma,\delta,i}}$ and of defect zero as a character of $\left(N_{\Gamma,\delta,i}(\theta_{\Gamma,\delta,i})\wr\fS(t_{\Gamma,\delta,i})\right)/
D_{\Gamma,\delta,i}^{t_{\Gamma,\delta,i}}$.
By Clifford theory, $\psi_{\Gamma,\delta,i}$ is of the form
\begin{equation}\label{weights:psi}
\Ind
_{N_{\Gamma,\delta,i}(\theta_{\Gamma,\delta,i})\wr
	\prod_j\fS(t_{\Gamma,\delta,i,j})}
^{N_{\Gamma,\delta,i}(\theta_{\Gamma,\delta,i})\wr\fS(t_{\Gamma,\delta,i})}
\overline{\prod_j\psi_{\Gamma,\delta,i,j}^{t_{\Gamma,\delta,i,j}}}\cdot
\prod_j\phi_{\lambda_{\Gamma,\delta,i,j}}
\end{equation}
where $t_{\Gamma,\delta,i}=\sum_j t_{\Gamma,\delta,i,j}$, $\overline{\prod_j\psi_{\Gamma,\delta,i,j}^{t_{\Gamma,\delta,i,j}}}$ is an extension of
$\prod_j\psi_{\Gamma,\delta,i,j}^{t_{\Gamma,\delta,i,j}}$ from $N_{\Gamma,\delta,i}(\theta_{\Gamma,\delta,i})^{t_{\Gamma,\delta,i}}$ to $N_{\Gamma,\delta,i}(\theta_{\Gamma,\delta,i})\wr
\prod_j\fS(t_{\Gamma,\delta,i,j})$, $\lambda_{\Gamma,\delta,i,j} \vdash t_{\Gamma,\delta,i,j}$ without $\ell$-hook and $\phi_{\lambda_{\Gamma,\delta,i,j}}$ a character of $\fS(t_{\Gamma,\delta,i,j})$ corresponding to $\lambda_{\Gamma,\delta,i,j}$.
Define $K_\Gamma:\bigcup_\delta\sC_{\Gamma,\delta} \to \{~\ell\textrm{-cores}~\},
\psi_{\Gamma,\delta,i,j} \mapsto \lambda_{\Gamma,\delta,i,j}$.
Then we get the third component $K=\prod_\Gamma K_\Gamma$ of the triple $(s,\lambda,K)$.

Now we define an action of $\mathcal O_{\ell'}(\mathfrak Z)$ on $i\cW_\ell(G)$ by setting $zK=\prod_{\Gamma}(zK)_\Gamma$ where $(zK)_{z.\Gamma}=K_\Gamma$.
For an $\ell$-weight $(R,\varphi)$ of $G$ with label $(s,\lambda,K)^G$, we also write $R=R_{s,\lambda,K}$ and $\varphi=\varphi_{s,\lambda,K}$.
Then by the conventions above, $R_{s,\lambda,K}=R_{zs,z\lambda,zK}$.

By Proposition \ref{gexingz}, $RC_G(R)/SC_X(S)\cong N_G(R)/N_X(S)\cong G/X$.
So we regard $\hat z$ as a character of $N_G(R)$~(or $C_G(R)$)~ for $z\in\mathcal O_{\ell'}(\mathfrak Z)$.

\begin{prop}\label{weightrestritoslsu}
$\hat z\varphi_{s,\lambda,K}=\varphi_{zs,z\lambda,zK}$ for $z\in\mathcal O_{\ell'}(\mathfrak Z)$.
\end{prop}

\begin{proof}
Let $(R,\varphi)$ be an $\ell$-weight of $G$ corresponding to $(s,\lambda,K)$ and assume $R$ can be decomposed as above.
Let $z\in\mathcal O_{\ell'}(\mathfrak Z)$.
We want to find which triple corresponds to $(R,\hat z\varphi)$.
Assume it be $(s',\lambda',K')$.

Now, $\hat z\varphi=\hat z\varphi_0\times\hat z\varphi_+$.
$\varphi_0$ is of the form $\chi_{s_0,\lambda}$ by construction.
By Proposition \ref{restrofordi}, $\hat z\chi_{s_0,\lambda}=\chi_{zs_0,z.\lambda}$.
Then we have $\lambda'={z.\lambda}$.

Secondly,
by Lemma \ref{forthethetagamma},
$\hat z\theta_{\Gamma,\delta,i}=\theta_{z.\Gamma,\delta,i}$ for $z\in\mathcal O_{\ell'}(\mathfrak Z)$.
Note that
$\hat z\theta_{\Gamma,\delta,i}$ corresponds to $e_\Gamma\ell^\delta(\Gamma)$ and $\theta_{z.\Gamma,\delta,i}$ corresponds to $e_{z.\Gamma}\ell^\delta(z.\Gamma)$.
Up to conjugacy, we have $s'=zs$.

Finally, by the conventions above, we may assume $D_{\Gamma,\delta,i}=D_{{z.\Gamma},\delta,i}$, $N_{\Gamma,\delta,i}=N_{{z.\Gamma},\delta,i}$, and
$C_{\Gamma,\delta,i}=C_{{z.\Gamma},\delta,i}$.
To determine $K'$, we note that $\hat z\psi_+=\prod_{\Gamma,\delta,i}\hat z \psi_{\Gamma,\delta,i}$.
By (\ref{weights:psi}), $\hat z\psi_{\Gamma,\delta,i}$ is
\begin{align*}
&\hat z\Ind
_{N_{\Gamma,\delta,i}(\theta_{\Gamma,\delta,i})\wr
	\prod_j\fS(t_{\Gamma,\delta,i,j})}
^{N_{\Gamma,\delta,i}(\theta_{\Gamma,\delta,i})\wr
	\fS(t_{\Gamma,\delta,i})}
\left(\overline{\prod_j\psi_{\Gamma,\delta,i,j}^{t_{\Gamma,\delta,i,j}}}\right)\cdot
\prod_j\phi_{\lambda_{\Gamma,\delta,i,j}}\\
=&\Ind
_{N_{z.\Gamma,\delta,i}(\theta_{z.\Gamma,\delta,i})\wr
	\prod_j\fS(t_{\Gamma,\delta,i,j})}
^{N_{z.\Gamma,\delta,i}(\theta_{z.\Gamma,\delta,i})\wr
	\fS(t_{\Gamma,\delta,i})}\hat z
\left(\overline{\prod_j\psi_{\Gamma,\delta,i,j}^{t_{\Gamma,\delta,i,j}}}\right)\cdot
\prod_j\phi_{\lambda_{\Gamma,\delta,i,j}}.
\end{align*}
Since $\hat z\theta_{\Gamma,\delta,i}=\theta_{z.\Gamma,\delta,i}$, we have $N_{\Gamma,\delta,i}(\theta_{\Gamma,\delta,i})=
N_{{z.\Gamma},\delta,i}(\theta_{{z.\Gamma},\delta,i})$.
We can fix the way to extend $\prod_j\psi_{\Gamma,\delta,i,j}^{t_{\Gamma,\delta,i,j}}$ as in \cite[Lem. 25.5]{Hu98}, then we have that
$\hat z\left(\overline{\prod_j\psi_{\Gamma,\delta,i,j}^{t_{\Gamma,\delta,i,j}}}\right)
=\overline{\prod_j\left(\hat z\psi_{\Gamma,\delta,i,j}\right)^{t_{\Gamma,\delta,i,j}}}.$
Since $\hat z\psi_{\Gamma,\delta,i,j}=\psi_{z.\Gamma,\delta,i,j}$ by (\ref{weight:convention-2}), $\hat z\psi_{\Gamma,\delta,i}$ would be
$$\Ind
_{N_{{z.\Gamma},\delta,i}(\theta_{{z.\Gamma},\delta,i})\wr
	\prod_j\fS(t_{\Gamma,\delta,i,j})}
^{N_{{z.\Gamma},\delta,i}(\theta_{{z.\Gamma},\delta,i})\wr
	\fS(t_{\Gamma,\delta,i})}
\overline{\prod_j\psi_{{z.\Gamma},\delta,i,j}^{t_{\Gamma,\delta,i,j}}}\cdot
\prod_j\phi_{\lambda_{\Gamma,\delta,i,j}}.$$
Then $K'_{z.\Gamma}=K_\Gamma$ which is just $K'={z. K}$.
Thus we complete the proof.
\end{proof}

Now by Proposition \ref{weightrestritoslsu},
for an $\ell$-weight $(R,\varphi)$ of $G$,
the number of irreducible constituents of $\Res^{N_G(R)}_{N_{X}(R)}\varphi$ can be obtained.

\begin{rmk}\label{weightofslsu}
Analogous to the description of irreducible Brauer characters of $G$ and $X$ in Remark \ref{brasuerofslsu}, now we give an analogous description of $\ell$-weights of $G$ and $X$ by summarizing the argument above.
	
For positive integers $h,w,d$,
we define $$I_d(h):=\{\ (d,k,j)\ |\ 1\le k\le h,\ 1\le j\le \ell^d \ \},$$
$I(h):=\coprod\limits_{d\ge 0}I_d(h)$,
and $$\mathscr A(h,w):=\{\ K:I(h)\to \{  \ell\text{-cores}\}\mid \sum_{d,k,j}\ell^d|K((d,k,j))|=w  \ \}.$$

We call a tuple
\begin{equation}\label{admissibleweighttuple}
(([\sigma_1], m_1, \lambda^{(1)},K^{(1)} ),\dots,([\sigma_a], m_a, \lambda^{(a)},K^{(a)}))
\end{equation}
of tuples an \emph{$(n,\ell)$-admissible weight tuple},
if
\begin{itemize}
\item for every $1\le i\le a$, $\sigma_i\in\overline{\mathbb F}^\times$ is an $\ell'$-element, and
$m_i$ is positive integers
 such that $\lambda^{(i)}$ is an $e_i$-core of some partition of $m_i$ and
$K^{(i)}\in\mathscr A(e_i,w_i)$
where $e_i$ is the multiplicative order of $(\epsilon q)^{\mathrm{deg}(\sigma_i)}$ modulo $\ell$
and $w_i=e_i^{-1}(m_i-|\lambda^{(i)}|)$,
\item $[\sigma_i]\ne[\sigma_j]$ if $i\ne j$, and
\item $\sum\limits^{a}_{i=1}m_i\mathrm{deg}(\sigma_i)=n$.
\end{itemize}

An equivalence class of the $(n,\ell)$-admissible weight tuple
(\ref{admissibleweighttuple})
up to a permutation of  tuples
$$([\sigma_1], m_1, \lambda^{(1)},K^{(1)} ), \ldots, ([\sigma_a], m_a, \lambda^{(a)},K^{(a)})$$
is called an \emph{$(n,\ell)$-admissible weight symbol} and is denoted as $$\mathfrak w=[([\sigma_1], m_1, \lambda^{(1)},K^{(1)} ),\dots,([\sigma_a], m_a, \lambda^{(a)},K^{(a)})].$$
Then by \cite{AF90}, \cite{An92}, \cite{An93} and \cite{An94}, the set of $(n,\ell)$-admissible weight symbols is a labeling set for the $G$-conjugacy classes of $\ell$-weights of $G$.
We denote by $(R_{\mathfrak w},\varphi_{\mathfrak w})$ the $\ell$-weight of $G$ corresponding to the $(n,\ell)$-admissible weight symbol $\mathfrak w$.

The group $\mathcal O_{\ell'}(\mathfrak Z)$ acts on the set of $(n,\ell)$-admissible weight symbols via
\begin{align*}
&z\cdot [([\sigma_1], m_1, \lambda^{(1)},K^{(1)} ),\dots,([\sigma_a], m_a, \lambda^{(a)},K^{(a)})]\\
=&[([z\sigma_1], m_1, \lambda^{(1)},K^{(1)} ),\dots,([z\sigma_a], m_a, \lambda^{(a)},K^{(a)})]
\end{align*}
for $z\in\mathcal O_{\ell'}(\mathfrak Z)$.
We denote by $\kappa(\mathfrak w)$ the order of the stabilizer group in $\mathcal O_{\ell'}(\mathfrak Z)$ of an $(n,\ell)$-admissible weight symbol $\mathfrak w$.

Assume that $\ell\nmid \mathrm{gcd}(n,q-\epsilon)$.
Then by Lemma \ref{cliffordthm} and Proposition \ref{weightrestritoslsu},
 $\kappa^{N_G(R_{\mathfrak w})}_{N_X(R_{\mathfrak w})}(\varphi_{\mathfrak w})=\kappa(\mathfrak w)$~(\emph{i.e.}, $\Res^{N_G(R_{\mathfrak w})}_{N_X(R_{\mathfrak w})}\varphi_{\mathfrak w}$ is a sum of $\kappa(\mathfrak w)$ irreducible constituents).
For two $(n,\ell)$-admissible weight symbols $\mathfrak w$ and $\mathfrak w'$, if they are in the same $\mathcal O_{\ell'}(\mathfrak Z)$-orbit, then
$R_{\mathfrak w}=R_{\mathfrak w'}$ and
the restrictions of $\varphi_{\mathfrak w}$ and $\varphi_{\mathfrak w'}$ to $N_X(R_{\mathfrak w}\cap X)$ are the same.

If moreover, we write the decomposition $\Res^{N_G(R_{\mathfrak w})}_{N_X(R_{\mathfrak w})}\varphi_{\mathfrak w}=\bigoplus^{\kappa(\mathfrak w)}_{j=1} (\varphi_{\mathfrak w})_j$,
then by Remark \ref{restritowei}, the set $\{(R_{\mathfrak w}\cap X,(\varphi_{\mathfrak w})_j)\}$, where $\mathfrak w$ runs through the $\mathcal O_{\ell'}(\mathfrak Z)$-orbit representatives of $(n,\ell)$-admissible weight symbols and $j$ runs through the integers between $1$ and  $\kappa(\mathfrak w)$,
is a complete set of representatives of $X$-conjugacy classes of the $\ell$-weights of $X$.
\end{rmk}

\begin{rmk}\label{weighforblocks}
Let $\mathfrak b=[([\sigma_1], m_1, \lambda^{(1)}),\dots,([\sigma_a], m_a, \lambda^{(a)})]$ be an $(n,\ell)$-admissible block symbol.
Then by \cite{AF90}, \cite{An92},  \cite{An93} and  \cite{An94},
the set of $\ell$-weights $\{\ (R_{\mathfrak w},\varphi_{\mathfrak w})\ \}$, where $\mathfrak w$ runs through the $(n,\ell)$-admissible symbols of the form
$$\mathfrak w=[([\sigma_1], m_1, \lambda^{(1)},K^{(1)} ),\dots,([\sigma_a], m_a, \lambda^{(a)},K^{(a)})],$$
is a complete set of representatives of $G$-conjugacy classes of $\ell$-weights of $B_{\mathfrak b}$.

Assume that $\ell\nmid\mathrm{gcd} (n,q-\epsilon)$.
If we write $\mathcal W_\ell (B_{\mathfrak b})=\{\ (R_1,\varphi_1),\dots,(R_l,\varphi_l) \ \}$,
then by Proposition \ref{weightrestritoslsu}, $\mathcal W_\ell(B_{z \mathfrak b})=\{\ (R_1,\hat z\varphi_1),\dots,(R_l,\hat z\varphi_l) \ \}$ for all $z\in\mathcal O_{\ell'}(\mathfrak Z) $.

Assume that $\ell$ is odd.
Let $b$ be an $\ell$-block of $X$ covered by $B_{\mathfrak b}$,
then the number of $\ell$-weights lying in $b$
of the form $(R_{\mathfrak w}\cap X,\varphi')$ where $\varphi'\in\Irr(N_X(R_{\mathfrak w})\ |\ \varphi_{\mathfrak w})$ is $\kappa(\mathfrak w)/\kappa(\mathfrak b)$.
\end{rmk}

For an $\ell$-block $B$ and $(n,\ell)$-admissible weight symbol $\mathfrak w$,
we say $\mathfrak w$ \emph{belongs to} $B$,
if $(R_{\mathfrak w},\varphi_{\mathfrak w})$ is a $B$-weight.

\vspace{2ex}

\begin{proof}[Proof of Theorem \ref{alpofsl}]
If $\ell=p$, then the assertion holds by \cite{Ca88}. Now we assume that $\ell\ne p$.	
For an $\ell$-block $b$ of $X$, let $B$ be an $\ell$-block associated to $B$.
By Remark \ref{brauerchartoblock}, \ref{weighforblocks} and \cite[(1A)]{AF90},
there is a natural bijection $\mathscr S$ from the $(n,\ell)$-admissible symbols belonging to $B$ onto the $(n,\ell)$-admissible weight symbols belonging to $B$.
For any two $(n,\ell)$-admissible symbols $\mathfrak s$, $\mathfrak s'$ which belong to $B$, by Remark \ref{brasuerofslsu} and \ref{weightofslsu}
and the construction of $\mathscr S$ in \cite[(1A)]{AF90},
we have
\begin{itemize}
\item $\kappa(\mathfrak s)=\kappa(\mathscr S(\mathfrak s))$,
\item $\mathfrak s$ and $\mathfrak s'$ are in the same $\mathcal O_{\ell'}(\mathfrak Z)$-orbit if and only if $\mathscr S(\mathfrak s)$ and $\mathscr S(\mathfrak s')$ are in the same $\mathcal O_{\ell'}(\mathfrak Z)$-orbit.
\end{itemize}
Hence $|\IBr_\ell(b)|=|\mathcal W_\ell(b)|$ by Remark \ref{forweightbelow}.
\end{proof}

\subsection{The unipotent blocks}

\begin{lem}\label{resofweightsuni}
Assume that $\ell\nmid\mathrm{gcd} (n,q-\epsilon)$.
Let $b$ be a unipotent $\ell$-block of $X$ and $B$  the unipotent $\ell$-block of $G$ which covers $b$.
Then $(R,\varphi)\mapsto (R\cap X, \Res^{N_G(R)}_{N_X(R)}\varphi)$ gives a bijection from $\mathcal W_\ell(B)$ to $\mathcal W_\ell(b)$.
\end{lem}

\begin{proof}
By Lemma \ref{restrforuni}, there is a unique unipotent $\ell$-block $B$ of $G$ which covers $b$.	
Then the claim follows by Remark \ref{weightofslsu} and \ref{weighforblocks} immediately.
\end{proof}

\begin{cor} \label{bijection}
Assume that $\ell\nmid \mathrm{gcd}(n,q-\epsilon)$.
If $b$ is a unipotent $\ell$-block of $X$,
then there exists an $\Aut(X)$-equivariant bijection between $\IBr_\ell(b)$ and $\mathcal W_\ell(b)$.
\end{cor}

\begin{proof}
Let $B$ be a unipotent $\ell$-block of $G$ which covers $b$.
By \cite[Thm. 1.1]{LZ18}, there exists a $D$-equivariant bijection between $\IBr_\ell( B)$ and $\mathcal W_\ell( B)$.
Then the assertion follows from Lemma \ref{restrforuni} and \ref{resofweightsuni} since
the automorphisms of $X$ induced by $G\rtimes D$ equal  $\Aut(X)$.
\end{proof}

Now note that the universal covering group of a simple group $\PSL_n(\epsilon q)$ is a group isomorphic to $\SL_n(\epsilon q)$, apart from a few exceptions, see \cite[6.1.8]{GLS98}.

\begin{cor} \label{iandii}
Assume that $\ell\nmid \mathrm{gcd}(n,q-\epsilon)$.
Let $b$ be a unipotent $\ell$-block of $X$,
then the conditions (i) and (ii) of Definition \ref{induc} hold for $b$.
\end{cor}
\begin{proof}
By Corollary \ref{bijection},
there is an $\Aut(X)$-equivariant bijection $\Omega_b:\IBr_\ell(b)\to\mathcal W_\ell(b)$.
Now for every $Q\in\Rad_\ell(X)$, we set
$$\IBr_\ell(b\ |\ Q):=\bigcup_{\psi\in\Irr^0(N_X(Q),b)}\{ \ \Omega_b^{-1}((Q,\psi))\ \}$$
and define a map
$$\Omega^X_Q:\IBr_\ell(b\ |\ Q)\to \mathrm{dz}_\ell(N_X(Q),b),$$
such that $\phi\mapsto \widetilde\Omega_b(\phi)$,
where $\widetilde\Omega_b(\phi)$ denotes the unique element in $\mathrm{dz}_\ell(N_X(Q),b)$ whose inflation $\psi$ to $N_X(Q)$ satisfies that $\Omega_b(\phi)=(Q,\psi)$.
Then by \cite[Lem. 3.8]{Sc15}~(or \cite[Lem.~2.10]{Sc16}),
the subsets $\IBr_\ell(b\ |\ Q)$ and maps $\Omega^X_Q$ defined here satisfy (i) and (ii) of Definition \ref{induc}.
\end{proof}

\begin{rmk}
In fact, we have a generalisation of Corollary \ref{iandii}.
Assume that $\ell\nmid \mathrm{gcd}(n,q-\epsilon)$.
Suppose that $s$ is a semisimple $\ell'$-element of $G$ such that $zs$ and $s$ are not $G$-conjugate for any $z\in\mathcal O_{\ell'}(\mathfrak Z)$.
Let $B$ be an $\ell$-block of $G$ with label $(s,\lambda)$ and $b$ the $\ell$-block of $X$ covered by $B$.
Then by the same argument, there exists an $\Aut(X)$-equivariant bijection between $\IBr_\ell(b)$ and $\mathcal W_\ell(b)$,
and then the conditions (i) and (ii) of Definition \ref{induc} hold for $b$.
\end{rmk}

To end this section, we give the following result for the general $\ell$-blocks.

\begin{prop}\label{for_general_blocks}
	Let $q=p^f$ be a power of a prime $p$ and $\ell$ a prime different from $p$.
	Assume that $X\in\{\SL_n(q),\SU_n(q)\}$ such that $\mathrm{gcd}(f,2|Z(X)|)=1$,
	$\ell\nmid |Z(X)|$ and $2\nmid |Z(X)|$.
	Then there is a blockwise bijection between the $\ell$-Brauer characters of $X$ and the $\ell$-weights of $X$ which is $\Aut(X)$-equivariant.
	
	In particular, the conditions (i) and (ii) of Definition \ref{induc} hold for any $\ell$-block of $X$.
\end{prop}

For a positive integer, we denote by $C_d$ the cyclic group of order $d$.
We will make use of the following lemma to prove Proposition \ref{for_general_blocks}.

\begin{lem}\label{conjugateofstab}
	Let $B_1$ and $E$ be cyclic groups of order $n_1$ and $n_2$ respectively.
	Suppose that $H=B\times E$ satisfies that either
	\begin{enumerate}
		\item[(i)] $B=B_1$, or
		\item[(ii)] $B=B_1\rtimes C_2$ is isomorphic to a dihedral group of order $2n_1$ and	$n_1$ is odd.
	\end{enumerate}
	Assume that $\mathrm{gcd}(|B|,|E|)=1$.
	Let $H_1$ and $H_2$ be two subgroups of $H$ such that $|H_1|=|H_2|$, $|H_1\cap B|=|H_2\cap B|$ and $H_1\cap B_1=H_2\cap B_1$.	
	Then $H_1$ and $H_2$ are conjugate in $H$.	
\end{lem}

\begin{proof}
	We first recall the result about the subgroups of direct products.
	A subgroup $H_0$ of $H=B\times E$ is determined by a tuple $(\tilde S_1, S_1,  \tilde S_2, S_2, \pi)$, where $ S_1\unlhd\tilde S_1$ are subgroups of $B$, $S_2\unlhd\tilde S_2$ are subgroups of $E$ and $\pi:\tilde S_1/S_1 \to \tilde S_2/S_2$ is a group isomorphism~(see for instance \cite[(1.1)]{Th97}).
	Now $\mathrm{gcd}(|B|,|E|)=1$, so $\tilde S_1=S_1$ and $\tilde S_2=S_2$, and hence $H_0=S_1\times S_2$ is also a direct product.
	Thus $H_1=(H_1\cap B)\times (H_1\cap E)$ and $H_2=(H_2\cap B)\times (H_2\cap E)$.
	
	Now $|H_1\cap B|=|H_2\cap B|$, so $|H_1\cap E|=|H_2\cap E|$.
	Then $H_1\cap E=H_2\cap E$ since $E$ is cyclic.
	Since $H_1\cap B_1=H_2\cap B_1$, we have that $H_1\cap B$ and $H_2\cap B$ are conjugate in $B$. So $H_1$ and $H_2$ are conjugate in $H$.
\end{proof}


\begin{proof}[Proof of Proposition \ref{for_general_blocks}]
	Thanks to \cite[Thm. C]{Sp13}, we can assume that $\ell\ne p$.	
	For any $\theta\in\IBr_\ell(X)$, let $\phi\in\IBr_\ell(G\mid \theta)$ and $(R,\varphi)$ the $\ell$-weight of $G$ corresponding to $\phi$ under the bijection induced by $\mathscr S$~(see the proof of Theorem \ref{alpofsl}). 	
	Let $S=R\cap X$ and $\psi\in\Irr(N_X(S)\mid \varphi)$.
	Now we consider the $G\rtimes D$-orbit of $\theta$~(and $(S,\psi)$, respectively)~ in $\IBr_\ell(X)$~(and $\mathcal W_\ell(X)$, respectively).
	Denote by $\Delta_1$ the $G\rtimes D$-orbit of $\theta$ in $\IBr_\ell(X)$
	and $\Delta_2$ the $\Aut(X)$-orbit of $(S,\psi)$ in $\mathcal W_\ell(X)$.
	By Remark \ref{brasuerofslsu} and \ref{weightofslsu} and the construction of $\mathscr S$~(note that it is $D$-equivariant by \cite[Thm. 1.1]{LZ18}), $\Delta_1$ and $\Delta_2$ have the same cardinality.
	Obviously, $\Aut(X)$ acts on $\Delta_1$~(or $\Delta_2$, respectively) as $\mathrm{Out}(X)$ does.
	Also $|\mathrm{Out}(X)_\theta|=|\mathrm{Out}(X)_\psi|$.
	
	Now we denote by $\mathrm{Outdiag}(X)$ the outer automorphisms induced by $G$ on $X$
	then $\mathrm{Outdiag}(X)\cong C_{\mathrm{gcd}(n,q-\epsilon)}$ is cyclic.
	Thus by Remark \ref{brasuerofslsu} and \ref{weightofslsu}, the stabilizers of $\theta$ and $\psi$ in $\mathrm{Outdiag}(X)$ are the same.
	If $n\ge 3$, by a similar argument of the paragraph above~(replace $\mathrm{Out}(X)$ by $\langle \mathrm{Outdiag}(X),\gamma \rangle$, where $\gamma$ is defined as in Section \ref{notations-and-conventions}), we have $|\langle \mathrm{Outdiag}(X),\gamma \rangle_\theta|=|\langle \mathrm{Outdiag}(X),\gamma \rangle_\psi|$.
	
	Now $$\mathrm{Out}(X)\cong
	\left\{ \begin{array}{ll}
	(\mathrm{Outdiag}(X)\rtimes C_2)\times C_f & \textrm{if}~\ n\ge 3,\\
	\mathrm{Outdiag}(X)\times C_f & \textrm{if}\ \ n=2.
	\end{array} \right.$$
	Thus by Lemma \ref{conjugateofstab}, $\mathrm{Out}(X)_\theta$ and $\mathrm{Out}(X)_\psi$ are conjugate in $\mathrm{Out}(X)$.
	Thus there exists an $\Aut(X)$-equivariant bijection between $\Delta_1$ and $\Delta_2$, hence there exists an $\Aut(X)$-equivariant bijection $\mathscr G$ between $\IBr_\ell(X)$ and $\mathcal W_\ell(X)$.
	Obviously, we can choose the bijection $\mathscr G$ satisfies that
	if $\theta\in\IBr_\ell(X)$, $\phi\in\IBr_\ell(G\mid \theta)$, $(R,\varphi)=\mathscr S(\phi)$, $S=R\cap X$, then
	$\mathscr G(\theta)=(S,\psi)$ for some $\psi\in\Irr(N_X(S)\mid \varphi)$.
	So $\mathscr G$ preserves blocks.
	Moreover, the conditions (i) and (ii) of Definition \ref{induc} hold for any $\ell$-block of $X$~(for details, see the proof of Corollary \ref{iandii}).
\end{proof}

\section{Extendibility of weight characters of unipotent blocks}
\label{extenofcharwei}

In this section, we will prove the following result.

\begin{prop}\label{extensionofunip}
Let $(R,\varphi)$ be an $\ell$-weight of $G$ which belongs to a unipotent $\ell$-block.
Then $\varphi$ extends to $(G\rtimes D)_{R,\varphi}$.
\end{prop}

We will use the following lemma.

\begin{lem}\label{extlemma}
Suppose that $H$ is a finite group, $C\unlhd H$, $N\unlhd H$, $D_0\le D\le H$, $\chi\in \Irr(N)$ satisfies that
\begin{itemize}
\item $H/N$ is abelian, $H=ND$, $N\cap D_0\le C_1$ and $H/ND_0$ is cyclic,
\item there are normal subgroups $C_0$, $C_1$, $N_0$ and $N_1$ of $H$ such that
$C=C_0\times C_1$, $N=N_0\times N_1$, $C_0= N_0$ and $C_1\le N_1$,
\item  $D_0$ acts trivially on $N_1/C_1$,
\item $N_0D=N_0\rtimes D$,
\item $\chi\in \Irr(N\ |\ \theta)$ where $\theta=\theta_0\times\theta_1$ with $\theta_0\in\Irr(C_0)$ and $\theta_1=1_{C_1}$,
\item  $H_\chi=H$ and $\theta_0$ extends to $N_0\rtimes (D/K)$, where $K$ is the kernel of the action of $D$ on $N_0$.
\end{itemize}
Then $\chi$ extends to $H$.
\end{lem}

\begin{proof}
Let $\chi=\chi_0\times\chi_1$ where $\chi_0=\theta_0$ and $\chi_1\in\Irr(N_1)$.
Now $\theta_0$ extends to $N_0\rtimes (D/K)$, so there exist an extension $\chi_0'\in\Irr(N_0D)$ of $\chi_0$ and a representation $\rho_0'$  affording $\chi_0'$
such that
if $n_0\in N_0$, $d,d'\in D$ satisfy that $d$ and $d'$ induce the same automorphism on $N_0$, then $\rho_0'(n_0d)=\rho_0'(n_0d')$.
Let $\tilde\rho_0=\Res^{N_0D}_{N_0D_0}\rho_0'$.

Let $\rho_1:N_1\to\GL_{\chi_1(1)}(\mathbb C)$ a representation of $N_1$ affording $\chi_1$.
Now let $\tilde\rho:ND_0\to\GL_{\chi(1)}(\mathbb C)$ satisfy
$\tilde\rho(n_0n_1d)=\tilde\rho_0(n_0d)\otimes\rho_1(n_1)$ for all $n_0\in N_0$, $n_1\in N_1$ and $d\in D_0$.
Here, $\tilde\rho$ is well-defined.
In fact, if $n_0,n_0'\in N_0$, $n_1,n_1'\in N_1$ and $d,d'\in D_0$ satisfy $n_0n_1d=n_0'n_1'd'$, then $n_0=n_0'$ and there exists $c\in C_1$ such that
$n_1=n_1'c$ and $d=c^{-1}d'$.
Hence $\rho_1(n_1)=\rho_1(n_1')$ since $C_1\le \ker \rho_1$.
Also, by the paragraph above, $\tilde\rho_0(n_0d)=\tilde\rho_0(n_0'd')$.
So $\tilde\rho(n_0n_1d)=\tilde\rho(n_0'n'_1d')$.

We claim that $\tilde\rho$ is a representation of $ND_0$.
In fact,
let $n_0,n_0'\in N_0$, $n_1,n_1'\in N_1$ and $d,d'\in D_0$,
\begin{align*}
\tilde\rho(n_0n_1dn_0'n_1'd')&=\tilde\rho(n_0(^dn_0')n_1(^dn_1')dd')\\
&=\tilde\rho_0(n_0(^dn_0')dd')\otimes\rho_1(n_1(^dn_1'))\\
&=\tilde\rho_0(n_0dn_0'd')\otimes\rho_1(n_1(^dn_1')).
\end{align*}
On the other hand,
$\tilde\rho(n_0n_1d)\tilde\rho(n_0'n_1'd')=\tilde\rho_0(n_0d)\tilde\rho_0(n_0'd')\otimes\rho_1(n_1)\rho_1(n_1')$.
Since $D_0$ acts trivially on $N_1/C_1$, then $^dn_1'=n_1'c$ for some $c\in C_1$.
Hence $\rho_1(n_1(^dn_1'))=\rho_1(n_1)\rho_1(^dn_1')=\rho_1(n_1)\rho_1(n_1')$ since $C_1\le\ker(\rho_1)$.
Thus the claim holds.

Let $g\in D$, $n_0\in N_0$, $n_1\in N_1$, $d\in D_0$,
then
$$\tilde\rho({}^g(n_0n_1d))=
\tilde\rho({}^g((n_0d)(^{d^{-1}}n_1)))
=\tilde\rho_0({}^g(n_0d))\otimes\rho_1(^{gd^{-1}}n_1).$$
Let $\tilde{\chi}$ be the character afforded by $\tilde\rho$,
then $\tilde{\chi}(n_0n_1d)=\tilde{\chi}_0(n_0d)\chi_1(d)$.
Hence
\begin{align*} \tilde\chi^g(n_0n_1d)&=\mathrm{Trace}(\tilde\rho_0({}^g(n_0d)))\mathrm{Trace}(\rho_1(^{gd^{-1}}n_1))\\
&=\tilde{\chi_0}({}^g(n_0d))\chi_1(^{gd^{-1}}n_1)
=\tilde{\chi_0}^g(n_0d)\chi_1^{gd^{-1}}(n_1)\\
&=\tilde{\chi_0}(n_0d)\chi_1(n_1)=\tilde{\chi}(n_0n_1d)
\end{align*}
since $\tilde\chi_0$, $\chi_1$ are $D$-invariant.
Thus $\tilde\chi$ is $D$-invariant.
Then $\tilde{\chi}$ extends to $H$ since $H/ND_0$ is cyclic.
So $\chi$ extends to $H$.
\end{proof}

\vspace{2ex}

First, by the uniqueness of $R_{m,\alpha,\gamma}$ and $R^i_{m,\alpha,\gamma}$ proved in \cite{An92}, \cite{An93} and \cite{An94}, $D$ acts trivially on the set of $G$-conjugacy classes of $\ell$-radical subgroups of $G$.
Denote $\sigma_1=F_p$ and $\sigma_2=\gamma$.
Then $D=\langle\sigma_1,\sigma_2 \rangle$.
So there exist $g^{(k)}\in G$ such that $g^{(k)}\sigma_k\in (G\rtimes D)_{R}$ for $k=1,2$.
Let $D'=\langle g^{(1)}\sigma_1,g^{(2)}\sigma_2 \rangle$.
Then we have the following result by direct calculation.

\begin{lem}\label{described}
With the notations above,
\begin{enumerate}
\item[(i)]  $(G\rtimes D)_{R}=N_G(R)D'$,
\item[(ii)] $D'/D'\cap G\cong D$,
\item[(iii)] $N_G(R)D'/N_G(R)\cong D$.
\end{enumerate}
\end{lem}

If $D$ is cyclic, then Proposition \ref{extensionofunip} holds immediately.
So we will assume that $D$ is not cyclic.
Then $\epsilon=1$,
that is $G=\GL_n(q)$. Let $q=p^f$ for some prime $p$ and integer $f$.
Then $f$ is even.
In particular, if $q$ is odd, then $4\mid q-1$.
Hence, by the description in Section \ref{chapweightsofslsu}, if $\ell=2$,  we always only have ``\textbf{Case 1}'' when considering basic subgroups. Then $D_{m,\alpha,\gamma,\bc}=R_{m,\alpha,\gamma,\bc}$ whenever $\ell$ is odd or $\ell=2$.

One embedding of $Z_\alpha E_\gamma$ can be constructed explicitly as follows (see, \cite{An92} and \cite{An94}).
Let $\xi$ be a fixed $\ell^{a+\alpha}$-th primitive root of unity in $\F_{(\epsilon q)^{e\ell^\alpha}}$ and $\zeta=\xi^{\ell^{a+\alpha-1}}$.
We first let $Z_0=\xi I_\gamma$ with $I_\gamma$ the identity matrix of degree $\ell^\gamma$ and
$$X_0= \diag ( 1,\zeta,\cdots,\zeta^{\ell-1} ),\quad Y_0 = \left[\begin{array}{cc} \zero & 1 \\ I_{\ell-1} & \zero \end{array}\right].$$
We then set $X_{0,j}=I_\ell \otimes \cdots \otimes X_0 \otimes \cdots \otimes I_\ell$ and $Y_{0,j}=I_\ell \otimes \cdots \otimes Y_0 \otimes \cdots \otimes I_\ell$ with $X_0$ and $Y_0$ appearing as the $j$-th components.
Define
$$\begin{array}{cccc} \rho_{\alpha,\gamma,0}:
& Z_\alpha E_\gamma & \longrightarrow & \GL(\ell^\gamma, (\epsilon q)^{e\ell^\alpha}) \\
& z & \longmapsto & Z_0 \\
& x_j & \longmapsto & X_{0,j} \\
& y_j & \longmapsto & Y_{0,j} \end{array}.$$

Now, let $\iota$ be an embedding of $\GL(\ell^\gamma, (\epsilon q)^{e\ell^\alpha})$ into $\GL(e\ell^{\alpha+\gamma},\epsilon q)$ with $\iota(\xi)$ being the companion matrix $(\Lambda_\alpha)$ of the polynomial $\Lambda_\alpha\in\cF$ having $\xi$ as a root.
Then we set $R_{\alpha,\gamma}$ the image of $Z_\alpha E_\gamma$ under $\rho_{\alpha,\gamma}=\iota\circ\rho_{\alpha,\gamma,0}$.	

For later use, we replace $R_{m,\alpha,\gamma}$ and $R_{m,\alpha,\gamma,\bc}$ by one of their conjugates.
Now define
$$Z_{m,0}=I_{(m)}\otimes Z_0,\ \ X_{m,0,j}=I_{(m)}\otimes X_{0,j},\ \ Y_{m,0,j}=I_{(m)}\otimes Y_{0,j}.$$
Define
$$\rho_{m,\alpha,\gamma,0}:Z_\alpha E_\gamma \to \GL(m\ell^\gamma, (\epsilon q)^{e\ell^\alpha})$$
in the same way as $\rho_{\alpha,\gamma,0}$ with $Z_0,X_{0,j},Y_{0,j}$ replaced by $Z_{m,0},X_{m,0,j},Y_{m,0,j}$.
Denote still by $\iota$ the embedding of $\GL(m\ell^\gamma, (\epsilon q)^{e\ell^\alpha})$ into $\GL(me\ell^{\alpha+\gamma},\epsilon q)$ and $\rho_{m,\alpha,\gamma}=\iota\circ\rho_{m,\alpha,\gamma,0}$.
Then we set $R_{m,\alpha,\gamma}$ the image of $\rho_{m,\alpha,\gamma}$.
Finally, we set $R_{m,\alpha,\gamma,\bc}=R_{m,\alpha,\gamma}\wr A_\bc$.

Now we give some precise information for $g^{(1)},g^{(2)}$ above.
Indeed, by \cite[Prop. 4.2 and 4.3]{LZ18},
if there is a decomposition $R=R_0\times R_1\times\cdots \times R_u$ where $R_0$ is a trivial group and $R_i\cong R_{m_i,\alpha_i,\gamma_i,\bc_i}$ ~($i\ge 1$) is a basic subgroup,
then $g^{(k)}$ is blockwise diagonal corresponding to the decomposition
$g^{(k)}=\mathrm{diag}(g^{(k)}_0,g^{(k)}_1,\dots, g^{(k)}_u )$ where $g^{(k)}_0$ is identity matrix and
$g^{(k)}_i=g_{m_i,\alpha_i}\otimes I_{\gamma_i}\otimes I_{\bc_i}$ with $g_{m_i,\alpha_i}\in G_{m_i\alpha_i}$
such that $g^{(k)}_i\sigma_k$ fixes $R_i$ for all $k=1,2$ and $0\le i\le u$.
Obviously, the action of $g^{(k)}_i\sigma_k$ on $G_{m_i,\alpha_i}\otimes I_\gamma\otimes I_{\bc}$,
$C_{m_i,\alpha_i}\otimes I_\gamma\otimes I_{\bc}$, and $N_{m_i,\alpha_i}\otimes I_\gamma\otimes I_{\bc}$
is just as the actions of
$g^{(k)}_{m_i,\alpha_i}\sigma_k$ on $G_{m_i,\alpha_i}$, $C_{m_i,\alpha_i}$ and $N_{m_i,\alpha_i}$, respectively, for all $k=1,2$ and $0\le i\le u$.
We also regard the actions above as the actions of $g^{(k)}\sigma_k$ ~( $k=1,2$).

\begin{lem}\label{defofdo}
With the notations above,
there exists a subgroup $D_0'$ of $D'$ independent of $m$, $\alpha$ and $\gamma$,
such that
$D_0'$ acts trivially on  $R_{m,\alpha,\gamma}$
and $D'/(D'\cap G)D_0'$ is cyclic.
In particular, $D_0'$ acts trivially on $N_{m,\alpha,\gamma}/C_{m,\alpha,\gamma}$.
\end{lem}

\begin{proof}
Denote $Z=\iota(Z_{m,0})$, $X_j=\iota(X_{m,0,j})$, $Y_j=\iota(Y_{m,0,j})$ and $B=\langle Z,X_j\mid j=1,\ldots,\gamma\rangle$, $H=\langle Y_j\mid j=1,\ldots,\gamma\rangle$.
Then $R_{m,\alpha,\gamma}=B\rtimes H$.	
By the proof of \cite[Lem. 4.1]{LZ18}, for $k=1,2$,
$$g^k\sigma_k(x)=\left\{ \begin{array}{ll} x^{h_k}, & \text{if}\ x\in B\\ x, & \text{if}\ x\in H \end{array} \right.$$
where $h_1=p$ and $h_2=-1$.

Now let $r$ be the multiplicative order of $p$ modulo $\ell$.
We take $D_0'=\langle (g^1\sigma_1)^r \rangle$ when $r$ is odd, and $D_0'=\langle (g^1\sigma_1)^{r/2}g^2\sigma_2,  \rangle$ when $r$ is even.
Then $D_0'$ acts trivially on $R_{m,\alpha,\gamma}$ and $D'/(D'\cap G)D_0'$ is cyclic.
\end{proof}

\begin{cor}\label{actiononcom}
With the notations above, $D_0'$ acts trivially on $N_{m,\alpha,\gamma,\bc}/C_{m,\alpha,\gamma,\bc}$.
\end{cor}

\begin{proof}
For $\bc=(c_1,\dots,c_t)$, we have
$N_{m,\alpha,\gamma,\bc}=N_{m,\alpha,\gamma}/R_{m,\alpha,\gamma}\otimes Y_{\bc}$	
by \cite{AF90} and \cite{An92}.
Here $ Y_{\bc}$ is the normalizer of $A_{\bc}$ in $\mathfrak S(\ell^{|\bc|})$ and then consists of permutation matrices.
By Lemma \ref{defofdo}, $D_0'$ acts trivially on $N_{m,\alpha,\gamma}/R_{m,\alpha,\gamma}C_{m,\alpha,\gamma}$.
Hence $D_0'$ acts trivially on $N_{m,\alpha,\gamma,\bc}/R_{m,\alpha,\gamma,\bc}C_{m,\alpha,\gamma,\bc}$
since $C_{m,\alpha,\gamma,\bc}=C_{m,\alpha,\gamma}\otimes I_{\bc}$.
\end{proof}

\begin{proof}[Proof of Proposition \ref{extensionofunip}]
By the argument after Lemma \ref{described}, we may assume that $\epsilon=1$ and $q=p^f$ with $f$ even.
Suppose that $G=\GL_n(q)=\GL(V)$, where $V$ is a vector space of dimension $n$ over $\F$.
By \cite[(4A)]{AF90} and \cite[(2B)]{An92}, $R=R_0\times R_+$ where $R_0$ is an identity group and $R_+$ is a direct product of basic subgroups.
Let $V=V_0\times V_+$ be the corresponding decomposition of $V$, such that $V_0$ is the underlying space of $R_0$ and $V_+$ is the underlying space of $R_+$.
Note that if $\ell=2$, then $\mathrm{dim}(V_0)=0$.
Then $C_G(R)=C_0\times C_+$ and $N_G(R)=N_0\times N_+$,
where $C_0=N_0=\GL(V_0)$, $C_+=C_{\GL(V_+)}(R_+)$, $N_+=N_{\GL(V_+)}(R_+)$.
Let $\theta\in\Irr(RC_G(R)\ | \ \varphi)$, then $\theta=\theta_0\times\theta_+$, where $\theta_0\in\Irr(R_0C_0)$ and $\theta_+\in\Irr(R_+C_+)$.
We write $\varphi=\varphi_0\times \varphi_1$, with $\varphi_0\in\Irr(N_0)$ and $\varphi_+\in\Irr(N_+)$.
Obviously, $\varphi_0=\theta_0$.

We write $R_+=R_1^{b_1}\times\cdots \times R_u^{b_u}$ as a direct product of basic subgroups, where $R_i$ appears $b_i$ times as a component of $R_+$.
Let $C_i=C_{\GL(V_i)}(R_i)$,  $N_i=N_{\GL(V_i)}(R_i)$,
where $V_i$ is the underlying space of $R_i$.
Then $C_+=C_1^{b_1}\times\cdots \times C_u^{b_u}$ and $\theta_+=\prod_{i=1}^{u}\prod^{v_i}_{j=1}\theta_{ij}^{b_{ij}},$
where $\theta_{i1},\dots,\theta_{iv_i}$ are distinct irreducible characters of $C_iR_i$ trivial on $R_i$ and $\theta_{ij}$ occurs $b_{ij}$ times as a factor in $\theta$.

Now $(R,\varphi)$ belongs to a unipotent $\ell$-block, so for all $1\le i\le u$, $1\le j \le v_i$, $\theta_{ij}$ has the form $\theta_{\Gamma,\delta,k}$ for some $\delta$ and $k$,
where $\Gamma=x-1$.
By the construction of $\theta_\Gamma$, we have $\theta_\Gamma=1_{C_\Gamma}$ when $\Gamma=x-1$~(since $\theta_\Gamma=\pm R^{C_\Gamma}_{C_{C_\Gamma}(1)}(\hat 1)=1_{C_\Gamma}$ by \cite[(4.16)]{Br86}).
Hence $\theta_{ij}=1_{C_iR_i}$ for all $1\le i\le u$, $1\le j \le v_i$.
Also, $\varphi_0=\theta_0$ is a unipotent character of $N_0=C_0$.
By Proposition \ref{extenofchar}, $\varphi_0$ extends to $N_0\rtimes D$.

Now $N_+(\theta_+)=\prod_{i=1^u}\prod_{j=1}^{v_i}(N_i(\theta_{ij})\wr \mathfrak S(b_{ij}))$.
By the argument above, $\theta_{ij}$ is the trivial character hence is invariant under $N_i$.
By Corollary \ref{actiononcom}, $D_0'$ acts trivially on $N_i/C_i$.
So $D_0'$ acts trivially on $N_+/C_+$ since $D_0'$ acts trivially on every $\mathfrak S(b_{ij})$.
Also, $N_G(R)D'/N_G(R)D'_0$ is cyclic since $D'/(D'\cap G)D_0'$ is cyclic by Lemma \ref{defofdo}.
Hence $\varphi$ extends to $(G\rtimes D)_R$ by Lemma \ref{extlemma}.
This completes the proof.
\end{proof}

\begin{cor}\label{extensofbrauer}
Assume that $\ell\nmid\mathrm{gcd}(n,q-\epsilon)$.
Let $(Q,\psi)$ be an $\ell$-weight of $X$ which belongs to a unipotent $\ell$-block.
Then $\psi$ extends to $(G\rtimes D)_{Q}$.	
\end{cor}

\begin{proof}
By Lemma \ref{resofweightsuni}, there is an $\ell$-weight $(R,\varphi)$ of $G$ in a unipotent $\ell$-block of $G$,
such that $Q=R\cap X$ and $\psi=\Res^{N_G(R)}_{N_X(Q)}\varphi$.
So $(G\rtimes D)_{R,\varphi}\le (G\rtimes D)_{Q,\psi}$.
Note that $(G\rtimes D)_{R,\varphi}=(G\rtimes D)_{R}$ and $(G\rtimes D)_{R}=(G\rtimes D)_{Q}$.
So $(G\rtimes D)_{Q,\psi}=(G\rtimes D)_{R}$.
Now by Proposition \ref{extensionofunip}, $\varphi$ extends to $(G\rtimes D)_{R}$, then $\psi$ extends to $(G\rtimes D)_{Q}$.
\end{proof}

\section{Proof of Theorem \ref{ibawmain}}
\label{proofof1.3}

Now we consider the condition (iii) of Definition \ref{induc}.

\begin{prop}\label{iii1-3}
Assume that $\ell\nmid\mathrm{gcd}(n,q-\epsilon)$ and $n\ge 3$.
Let $b$ be a unipotent $\ell$-block of $X$,
then the subsets $\IBr_\ell(b\ |\ Q)$ and maps $\Omega^X_Q$, for every $Q\in\Rad_\ell(X)$, defined as in the proof of Corollary \ref{iandii}, satisfy
Definition \ref{induc} (iii)(1)-(3) for $\phi\in\IBr_\ell(b,Q)$, $A:=A(\phi, Q)=(G\rtimes D)/\mathcal O_\ell(Z(G))Z$ with $Z=Z(X)\cap\ker(\phi)$.
\end{prop}

\begin{proof}
Now $\overline X=X/Z$.
It is easy to check (1) of Definition \ref{induc} (iii).
For (2), by Corollary \ref{extofunpofsl},
we have an extension $\phi'\in\IBr_\ell(G\rtimes D)$ of $\phi$.
Then $\mathcal O_\ell(Z(G))\le \mathcal O_\ell(G\rtimes D)\le \ker(\phi')$.
Also, $Z\le \ker(\phi')$.
Let $\tilde \phi$ be the Brauer character of $A$ associated with $\phi'$, then $\tilde \phi$ is an extension of the $\ell$-Brauer character of $\overline X$ associated with $\phi$.

For (3), we let $\psi \in \Irr(N_X(Q))$ be the inflation of $\Omega^X_Q(\phi)$ to $N_X(Q)$ and $\overline{\psi}$ be the character of $N_{\overline{X}}(\overline{Q})=\overline{N_X(Q)}$ associated with $\varphi$.
Moreover, we have $(G\rtimes D)_{Q,\psi}=N_{(G\rtimes D)_\phi}(Q)$ by \cite[Lem. 9.16]{Sc15} since $Z(G)=Z(G\rtimes D)$ and $\Aut(X)=({G}\rtimes D)/Z({G})$.
Now by Corollary \ref{extensofbrauer},
$\psi \in \Irr(N_X(Q))$ extends to a character $\tilde\psi$ of $(G\rtimes D)_{Q,\psi}=(G\rtimes D)_{Q}$,
then there exists an extension of $\overline{\psi}$ to $N_A(\overline{Q})=N_{G\rtimes D}(Q)/ \mathcal{O}_\ell(Z({G}))Z$ since $\mathcal{O}_\ell(Z({G}))\subseteq \ker(\tilde\psi)$ by the proof of Corollary \ref{extensofbrauer}.
Then $\overline{\psi}^\circ$ extends to $N_A(\overline{Q})$. This completes the proof.
\end{proof}

For condition (4) of Definition \ref{induc}(iii), we have:

\begin{lem}
	\label{Corr}
Keep the hypotheses and setup of Proposition \ref{iii1-3},
let $(S,\varphi)$ be an $\ell$-weight of $X$. Denote by $\phi'$ the inflation of $\Omega_S^X(\phi)^0$ viewed as $\ell$-Brauer character in $\mathrm{IBr}_\ell(N_{\overline{X}}(\overline{S})/\overline{S})$ to $N_{\overline{X}}(\overline{S})$. Let $\tilde{\phi'}$ be an extension of $\phi'$ to $N_A(\overline{S})$. Then there exists an extension $\tilde{\phi}\in \mathrm{IBr}_\ell(A)$ of $\phi$ to $A$ satisfying	$$\mathrm{bl}_\ell(\mathrm{Res}^{N_A(\overline{S})}_{N_J(\overline{S})}\tilde{\phi'})^J = \mathrm{bl}_\ell(\mathrm{Res}^{A}_{J}\tilde{\phi})$$
	for any $\overline{X}\leq J\leq A$.
\end{lem}

\begin{proof}
Since $A/\overline{X}$ is solvable, all Hall $\ell'$-subgroups of $A/\overline{X}$ are conjugate and every $\ell'$-element of $A$ is contained in some $J$ such that $J/\overline{X}$ is a Hall $\ell'$-subgroup of $A/\overline{X}$. Then by \cite[Lem 2.4 and 2.5(a)]{KS15}, to prove this proposition, it suffices to prove that $A=\overline{X}N_A(\overline{S})$ and that the proposition holds for certain~(thus every) $\overline{X}\leq J\leq A$ such that $J/\overline{X}$ is a Hall $\ell'$-subgroup of $A/\overline{X}$~(for details, see the proof of \cite[Prop. 9.21]{Sc15}).

First, let $R=S\mathcal O_\ell(Z(G))$, then by Proposition \ref{relaofrad}, $R$ is an $\ell$-radical subgroup of $G$ such that $R\cap X=S$.
As pointed in Section \ref{extenofcharwei}~(by the uniqueness of $R_{m,\alpha,\gamma}$ and $R^i_{m,\alpha,\gamma}$ proved in \cite{An92}, \cite{An93} and \cite{An94}), $G\rtimes D$ acts trivially on the $G$-conjugacy classes of $\ell$-radical subgroups of $G$.
Hence $G\rtimes D=GN_{G\rtimes D}(R)$ by Frattini's argument .
Then $A=\overline{X}N_A(\overline{S})$ since
$N_A(\overline{S})=N_{G\rtimes D}(S)/ \mathcal{O}_\ell(Z({G}))Z$ and $N_{G\rtimes D}(S)=N_{G\rtimes D}(R)$.

Now $(\tilde{G}\rtimes D)_\phi=\tilde{G}\rtimes D$.
Let $\overline{G}:=G/Z\mathcal{O}_\ell(Z({G}))$, then
$A=(G\rtimes D)_\phi/Z\mathcal{O}_\ell(Z({G})) =\overline{{G}}\rtimes D$.
Since $\ell\nmid\mathrm{gcd}(n,q-\epsilon)$,
$\overline{{G}}/\overline{X} \cong G/X\mathcal{O}_\ell(Z({G}))$ is an $\ell'$-group.
Thus there is a unique Hall $\ell'$-subgroup $\hat{A}/\overline{X}$ of $A/\overline{X}$ such that $\overline{G} \leq \hat{A}$.
Let $\hat{\phi'}=\mathrm{Res}^{N_A(\overline{R})}_{N_{\overline{{G}}}(\overline{R})}\tilde{\phi'}$, then $\mathrm{bl}(\hat{\phi'})^{\overline{{G}}}$ covers $\mathrm{bl}(\phi')^{\overline{X}}=\mathrm{bl}(\phi)$.
Note that $\phi$ extends to $\hat{\phi} \in \mathrm{IBr}_\ell(\mathrm{bl}(\hat{\phi'})^{\overline{{G}}})$
with $A_{\hat{\phi}}=A_\phi=A$ by Lemma \ref{restrforuni} and Corollary \ref{extensofbrauer}.
Replacing $\overline{X}$, $N_{\overline{X}}(\overline{R})$, $\phi$, $\phi'$ by $\overline{{G}}$, $N_{\overline{{G}}}(\overline{R})$, $\hat\phi$, $\hat\phi'$ respectively and noting that $A/\overline{{G}}$ is abelian, we can use the same arguments as in the first paragraph of the proof of \cite[Prop. 9.21]{Sc15} to prove that the proposition holds for $\hat{A}$. By the remarks at the beginning of the proof, the proposition holds for general $\overline{X} \leq J \leq A$.
\end{proof}

\begin{proof}[Proof of Theorem \ref{ibawmain}]
If $\ell=p$, then the assertion holds by \cite[Thm. C]{Sp13}. Now we assume that $\ell\ne p$.
Now the case when $n\ge 3$ is completely solved by our results in
Corollary \ref{iandii}, Proposition \ref{iii1-3} and Lemma \ref{Corr}.
Now we assume that $n=2$.
By Corollary \ref{iandii}, it suffices to check condition (iii) of Definition \ref{induc}.
Note that, if we define $D:=\langle F_p\rangle $ for this case, then it is easy to see that Proposition \ref{iii1-3} and Lemma \ref{Corr} also hold by the same argument. This completes the proof.
\end{proof}

\section*{Acknowledgements}

The author would like to express his gratitude to Gunter Malle for his support, useful conversations, keen advice he has provided and careful reading an earlier version of this paper and numerous helpful comments. Special thanks also go to Britta Sp\"ath, Conghui Li and Zhenye Li for fruitful discussions. Moreover, the author is grateful to the referee for useful comments and suggestions.

\end{document}